\documentclass[11pt,UTF-8,reqno]{amsart}
\usepackage{geometry}
\usepackage{amsfonts}
\usepackage{mathrsfs}
\usepackage[greek,english]{babel}
\usepackage{amsmath,amsfonts,amsthm,amssymb,mathtools,amscd,color,xcolor,mathrsfs,verbatim,microtype}
\usepackage{graphicx,eurosym}
\usepackage{amssymb,amsmath,comment}
\usepackage{amsfonts,mathrsfs}
\usepackage{color}
\usepackage{bbm}
\usepackage{bm}
\usepackage{hyperref}
\usepackage{stmaryrd} % \interleave need this
\usepackage{cleveref}

\usepackage[cyr]{aeguill}

\usepackage{enumerate, bbm}

\colorlet{darkblue}{blue!50!black}

\hypersetup{
	colorlinks,%
	citecolor=darkblue,%
	filecolor=darkred,%
	linkcolor=blue,%
	urlcolor=blue,%
	pdfnewwindow=true,%
	pdfstartview={FitH}
}

\usepackage{graphicx,amscd,mathrsfs,wrapfig,mathrsfs,lipsum}
\usepackage{eufrak}
\usepackage{float}
\usepackage{tikz}
\usepackage{multicol}
\usepackage{caption}
\usetikzlibrary{arrows}
\usepackage{capt-of}

\colorlet{darkblue}{blue!50!black}

\newcommand{\vertiii}[1]{{\left\vert\kern-0.25ex\left\vert\kern-0.25ex\left\vert #1
		\right\vert\kern-0.25ex\right\vert\kern-0.25ex\right\vert}}

\usepackage{marginnote}

\newcommand{\bae}{\begin{equation}\begin{aligned}}
		\newcommand{\eae}{\end{aligned}\end{equation}}
\newcommand{\baee}{\begin{equation*}\begin{aligned}}
		\newcommand{\eaee}{\end{aligned}\end{equation*}}

\theoremstyle{plain}

\newtheorem*{lemma*}{Lemma}
\newtheorem{theorem}{Theorem}[section]
\newtheorem{lemma}[theorem]{Lemma}

\theoremstyle{definition}
\newtheorem{definition}[theorem]{Definition}
\newtheorem{condition}[theorem]{Condition}

\theoremstyle{remark}

\newtheorem{remark}{Remark}[section]

\numberwithin{equation}{section}

\newcommand{\ea}{\end{array}}

\newcommand{\worknote}[1]{}

\newcommand{\eref}[1]{(\ref{#1})}
\def\<{{\langle}}
\def\>{{\rangle}}

\allowdisplaybreaks

\begin{document}
\title[Stochastic PDEs]
{Stochastic PDEs with Generalized Coercivity: Global Well-posedness and Finite Time Extinction$^\dagger$}
\thanks{$\dagger$
This work is supported by National Key R\&D program of China (No. 2023YFA1010101). The research of W. Hong is also supported by  NSFC (No.~12401177) and  NSF of Jiangsu Province (No.~BK20241048). The research of S. Li is also supported by NSFC (No.~12371147). The research of W. Liu is also supported by NSFC (No.~12171208, 12090011,12090010) and the PAPD of Jiangsu Higher Education Institutions.}

\maketitle
\centerline{\scshape Wei Hong,  Shihu Li, Wei Liu\footnote{Corresponding author: weiliu@jsnu.edu.cn} }
\medskip
\centerline{ School of Mathematics and Statistics, Jiangsu Normal University, Xuzhou, 221116, China}

\begin{abstract}
This work investigates the global existence, uniqueness, and Feller property for stochastic partial differential equations  under generalized coercivity conditions, particularly in cases where the corresponding deterministic equations possess only local solutions.
Furthermore, we reveal a novel phenomenon: for a potentially explosive deterministic system, the introduction of appropriate  multiplicative noise not only  prevents blow-up but also leads to the finite-time extinction of the stochastic dynamics.
Our main results are applicable to a broad range of models, including stochastic 3D Navier-Stokes equations, stochastic surface growth models, and stochastic $p$-Laplace equations with heat sources.

\bigskip
\noindent
\textbf{Keywords}:
Stochastic PDEs; Generalized coercivity;  Well-posedness;  Feller property; Finite-time extinction.\\
\textbf{Mathematics Subject Classification (2020)}: 60H15, 35R60

\end{abstract}

\tableofcontents

\section{Introduction}
This work focuses on stochastic partial differential
equations (SPDEs) within the variational framework. Originating from the pioneering work of Minty \cite{Mi62}, the variational approach was subsequently and systematically developed by Browder \cite{Bro63,Bro64}, Leray and Lions \cite{LL65}, and Hartman and Stampacchia \cite{HS66}.

We begin by reviewing the fundamental concepts of the variational approach and summarizing relevant developments in the literature.
Let $(U,\langle\cdot,\cdot\rangle_U)$ and $({\mathbb{H}}, \langle\cdot,\cdot\rangle_{\mathbb{H}}) $ be  separable Hilbert spaces, and ${\mathbb{H}}^*$ be the dual space of ${\mathbb{H}}$. Let $({\mathbb{V}},\|\cdot\|_{\mathbb{V}})$ denote a reflexive Banach space such that the embedding
$$\mathbb{V}\subset \mathbb{H}$$
is continuous and dense. Identifying $\mathbb{H}$ with its dual space in view of the Riesz isomorphism, we obtain a Gelfand triple
\begin{equation}\label{gel}
 {\mathbb{V}}\subset {\mathbb{H}}(\simeq {\mathbb{H}}^*)\subset {\mathbb{V}}^*.
 \end{equation}
The dualization between ${\mathbb{V}}$ and ${\mathbb{V}}^*$ is denoted by $_{{\mathbb{V}}^*}\langle\cdot,\cdot\rangle_{\mathbb{V}}$. It is clear that
$$_{{\mathbb{V}}^*}\langle\cdot,\cdot\rangle_{\mathbb{V}}|_{{\mathbb{H}}\times {\mathbb{V}}}=\langle\cdot,\cdot\rangle_{\mathbb{H}}.$$
Let $L_2(U,\mathbb{H})$ be  the space of all Hilbert-Schmidt operators from $U$ to $\mathbb{H}$.
For measurable maps
$$
\mathcal{A}:[0,T]\times {\mathbb{V}}\rightarrow {\mathbb{V}}^*,~~\mathcal{B}:[0,T]\times {\mathbb{V}}\rightarrow L_2(U,\mathbb{H}),
$$
we consider the following SPDE
\begin{equation}\label{eqSPDE}
dX_t=\mathcal{A}(t,X_t)dt+\mathcal{B}(t,X_t)dW_t,~~X_0=x,
\end{equation}
where $\{W_t\}_{t\in [0,T]}$ is an $U$-valued cylindrical Wiener process defined on a complete filtered probability space $\left(\Omega,\mathscr{F},\{\mathscr{F}_t\}_{t\in [0,T]},\mathbb{P}\right)$.

In the deterministic case (i.e.,~$\mathcal{B}\equiv0$), the classical result  states that Eq.~\eref{eqSPDE} has a unique solution if the operator $\mathcal{A}$ satisfies the monotonicity condition
\begin{eqnarray}\label{I04a}
_{\mathbb{V}^*}\langle \mathcal{A}(t,u)-\mathcal{A}(t,v),u-v\rangle_{\mathbb{V}}\leq
C\|u-v\|_{\mathbb{H}}^2
\end{eqnarray}
and the coercivity condition (see e.g.~\cite[Thoerem 30.A]{Z90})
\begin{eqnarray}\label{I04}
_{\mathbb{V}^*}\langle \mathcal{A}(t,u),u\rangle_{\mathbb{V}}+\delta\|u\|_{\mathbb{V}}^\alpha\leq
C\|u\|_{\mathbb{H}}^2+C,
\end{eqnarray}
 where $\delta,C>0$ and $\alpha>1$.

In the stochastic setting, the variational approach was initiated by Pardoux \cite{Par75} and Krylov and Rozovskii \cite{KR} for SPDEs with coefficients satisfying the monotonicity (\ref{I04a}) and coercivity (\ref{I04}) conditions. These classical results apply to a class of quasi-linear stochastic models, such as the stochastic porous medium equation and the stochastic $p$-Laplace equation (see, e.g., \cite{PR07,RRW07}).

It is important to note that the previous works, both in  deterministic and stochastic cases, require that the operator $\mathcal{A}$  satisfies the classical coercivity condition \eref{I04}, which ensures global-in-time bounds of solutions. In this work, we are interested in SPDEs  satisfying the following generalized coercivity condition
\begin{eqnarray} \label{I05}
2_{\mathbb{V}^*}\langle \mathcal{A}(t,u),u\rangle_{\mathbb{V}}+\delta\|u\|_{\mathbb{V}}^\alpha\leq g(\|u\|_{\mathbb{H}}^2)+C,
\end{eqnarray}
where $\delta,C>0$ and $g:[0,\infty)\to[0,\infty)$ is a non-decreasing continuous function. This condition was considered by Liu and R\"{o}ckner in \cite[Section 5.2]{LR1} (see also \cite{LR13}). Many important models satisfy the condition \eref{I05}, yet fail to satisfy the classical coercivity condition \eref{I04}. Typical  models that fit into this framework include
\begin{itemize}
  \item \textbf{3D Navier--Stokes equations}:
  \begin{eqnarray*}
\left\{
  \begin{aligned}
 &\partial_t{u}=\nu\Delta{u}-(u\cdot{\nabla})u-\nabla{p}+f, \\
    &\text{div}(u)=0.
  \end{aligned}
\right.
\end{eqnarray*}

  \item \textbf{Surface growth models}:
  $$\partial_t{u}= -\partial_x^4 u-\partial_x^2 u+\partial_x^2(\partial_x u)^2.$$

  \item \textbf{$p$-Laplace equations with heat sources}:
  $$\partial_t{u}= \text{div}(|\nabla u|^{p-2}\nabla u)+u^{2}.$$
\end{itemize}
Under this  assumption,  they proved the existence and uniqueness of \emph{local solutions} to SPDEs \eref{eqSPDE} perturbed by additive noise (i.e.,~$\mathcal{B}(t,u)=\mathcal{B}(t)$).
However, the global well-posedness of SPDEs under this generalized coercivity condition has not yet been established.

\subsection{Well-posedness}

The first aim  of this paper is to establish the well-posedness of SPDEs  satisfying the generalized coercivity condition (\ref{I05}).  Before presenting our main results, we first review the research progress achieved within the variational framework in the literature.   In \cite{LR2}, Liu and R\"{o}ckner  extended the classical variational framework by introducing the following local monotonicity condition
\begin{eqnarray}\label{I04b}
_{\mathbb{V}^*}\langle \mathcal{A}(t,u)-\mathcal{A}(t,v),u-v\rangle_{\mathbb{V}}\leq
(C+\rho(v))\|u-v\|_{\mathbb{H}}^2,
\end{eqnarray}
where $\rho: \mathbb{V}\rightarrow[0,+\infty)$ is a measurable function that is locally bounded on $\mathbb{V}$. This extension includes various  examples that cannot be treated previously, such as stochastic Burgers equations and stochastic 2D Navier-Stokes equations.

More recently, several notable advancements have been made in this direction. Agresti and Veraar \cite{AV24} introduced a critical variational framework for SPDEs under local Lipschitz conditions, greatly generalizing the classical setup.  By employing pseudo-monotone operators and compactness arguments, Shang et al.~\cite{RSZ1} further generalized the variational framework to a fully local monotonicity condition
\begin{eqnarray}\label{I04b}
_{\mathbb{V}^*}\langle \mathcal{A}(t,u)-\mathcal{A}(t,v),u-v\rangle_{\mathbb{V}}\leq
(C+\rho(v)+\eta(u))\|u-v\|_{\mathbb{H}}^2,
\end{eqnarray}
where $\rho$ and $\eta$ are measurable and  locally bounded on $\mathbb{V}$. The results in \cite{AV24} and \cite{RSZ1} can be applied to several new models, including the 3D tamed Navier-Stokes equations and Cahn-Hilliard equations. For further developments on the variational framework, we refer the interested reader to \cite{GC1,HHL,HLL25,HLL26,LR21,NTT21,W26,W15} and the references therein.

In the present paper, we study global solutions of SPDEs whose coefficients satisfy the generalized coercivity condition \eqref{I05} and fully local monotonicity condition \eref{I04b} under suitable stochastic perturbations. More precisely, we consider diffusion coefficients in the random noise satisfying
\begin{equation}\label{I06}
g(\|u\|_{\mathbb{H}}^2)+\|\mathcal{B}(t,u)\|_{L_2({U},{\mathbb{H}})}^2
\leq C(1+\|u\|_{\mathbb{H}}^2)+\eta_0\frac{\|\mathcal{B}(t,u)^*u\|_{U}^2}{(1+\|u\|_{\mathbb{H}}^2)},
\end{equation}
where  $\eta_0\in(1,2)$,  $C>0$, and  $g$ is the same function as in \eref{I05}.
This condition reflects an anisotropic stabilizing mechanism: only the component of the noise acting along the energy direction contributes to the prevention of blow-up.

Our first main results concern the global well-posedness and the Feller property of SPDEs under the fully local monotonicity condition \eqref{I04b} and the generalized coercivity condition \eqref{I05} and \eqref{I06}. A more precise statement of these results is provided in Theorems \ref{th1} and \ref{th2}.

\begin{theorem}\label{th01a}
Suppose that the embedding $ \mathbb{V}\subset \mathbb{H}$ is compact and that assumptions $(\mathbf{A_1})$-$(\mathbf{A_5})$ (see Subsection \ref{sec.2.2a} below) hold.
\begin{enumerate}[(i)]
  \item For any initial data $x\in\mathbb{H}$,
Eq.~(\ref{eqSPDE}) has a unique strong solution $X_t$. Moreover, the following energy moment estimates hold
\begin{equation}\label{apri001}
\sup_{t\in[0,T]}\mathbb{E}\|X_t\|_{\mathbb{H}}^{2-\eta_0}+\mathbb{E}\int_0^T\|X_t\|_{\mathbb{V}}^{\alpha-\eta_0}dt<\infty.
\end{equation}
  \item The solution depends continuously on the initial values. In particular, the corresponding Markov semigroup $(\mathcal{T}_t)_{t\geq 0}$ possesses the Feller property.
\end{enumerate}

\end{theorem}

We make several remarks regarding the above results.
\begin{itemize}
  \item Note that under the generalized coercivity condition \eqref{I05}, the corresponding deterministic PDE (i.e.,~when the noise term vanishes) may exhibit finite-time blow-up. Consequently, compared with existing literature, the key feature of our work lies in capturing the effect of random noise, which, to the best of our knowledge, is the first such result within the variational framework.

  \item To establish the global existence of solutions, it is essential to balance the influence of the noise intensity against the growth rate of the nonlinear terms. In this regard, condition \eqref{I06} plays a crucial role. Roughly speaking, when $g(x)=C_0 x$, which corresponds to the classical coercivity condition \eqref{I04}, the diffusion term can either vanish or take the form of additive, linear, or superlinear multiplicative noise. In contrast, if $g(x)=C_0 x^p$ with $p>1$, then $\mathcal{B}(t,\cdot)$ is required to exhibit superlinear growth.

  \item From the energy estimate \eqref{apri001}, one observes that, unlike existing results in the variational framework (see, e.g., \cite{LR2,RSZ1}), finite second moments are not available for the global solutions. This limitation  is inherent to the Lyapunov approach utilized to characterize the influence of stochastic perturbations.
\end{itemize}

Our new framework establishes general conditions that guarantee global existence and uniqueness. We demonstrate that various classical SPDEs, which lie outside the scope of the standard variational framework, indeed satisfy the generalized coercivity condition \eqref{I05}. In particular, our results apply to stochastic $p$-Laplace equations with heat sources. Furthermore, our main theorems can be applied to a broad class of stochastic models arising in fluid dynamics, such as the stochastic 3D Navier--Stokes equations and the stochastic surface growth models; these models cannot be accommodated by existing works \cite{AV24,LR2,NTT21, RSZ1}. Notably, the global well-posedness of the  deterministic 3D Navier--Stokes equations remains a long-standing open problem.

Moreover,  all locally or fully locally monotone examples studied in \cite{LR2} and \cite{RSZ1}, such as stochastic porous medium equations, stochastic 2D Navier--Stokes equations, stochastic Cahn--Hilliard equations, stochastic liquid crystal models, and stochastic Allen--Cahn--Navier--Stokes systems, can be accommodated within the more general framework established in this paper. Furthermore, for these models, our results can be applied to handle random noises with both linear and superlinear growth.

\subsection{Finite-time extinction}

Building on the global well-posedness result, we investigate the asymptotic behavior of Eq.~(\ref{eqSPDE}) in this setting.  In particular, a natural question concerns the influence of noise on the long-time dynamics of stochastic systems. Here, we mainly focus on whether suitable  stochastic perturbations can prevent or promote finite-time extinction of solutions.

The phenomenon of finite-time extinction is generally divided into the following three
cases
\begin{enumerate}[(i)]

\item \textbf{Local weak extinction:} For small initial values,  we have
$$\mathbb{P}(X_t~\text{extinct in finite time})>0;$$

\vspace{1mm}
\item \textbf{Global weak extinction:} For all initial values, we have
$$\mathbb{P}(X_t~\text{extinct in finite time})>0;$$

\vspace{1mm}
\item \textbf{Almost sure global extinction:} For all initial values, we have
$$\mathbb{P}(X_t~\text{extinct in finite time})=1.$$
\end{enumerate}
While Regimes (i) and (ii) are mathematically intriguing,  the robustness of the relaxation into subcritical states in self-organized criticality (SOC) is of fundamental importance in physics. Consequently,   Regime (iii) aligns most closely with the paradigm of SOC.

There have been many studies in the literature on the finite-time extinction of stochastic systems. For example,  Barbu, Da Prato and R\"{o}ckner \cite{BDR09}
investigated the finite-time extinction with positive probability for 1D  stochastic SOC models driven by linear multiplicative noise. In a subsequent work \cite{BR12}, Barbu and R\"{o}ckner established asymptotic extinction results with probability one for stochastic porous media equations in dimension $d \in \{1, 2, 3\}$. This result was further extended by R\"{o}ckner and Wang \cite{RW13} to more general cases. Specifically, they proved the
finite-time extinction with probability one for the Zhang model,  and with positive probability for the Bak-Tang-Wiesenfeld (BTW) model. Notably, the finite-time extinction of the BTW model with probability one remained an open problem until it was resolved by Gess \cite{Gess15}   for all dimensions  $d\geq1$. For more results of this topic, we refer interested readers to
\cite{BDR09b,BRR15,Hensel21} and references therein.

It is worth noting that existing results on finite-time extinction  relies on the structure of linear multiplicative noise.  In contrast, our work establishes almost sure  extinction  for any initial value $x\in\mathbb{H}$ (i.e.~Case (iii)) for a large class of quasi-linear SPDEs driven by linear or nonlinear multiplicative noises. This leads to the following main result (for details, see Theorem \ref{th3} below).

\begin{theorem}\label{th3a10}
Let $\tau_{e}$ denote the extinction time. Under the assumptions of Theorem \ref{th01a}, together with the additional conditions $(\mathbf{A_3^*})$ and $(\mathbf{A_5^*})$ (see Section \ref{sec.finite} for details), for any initial value $x\in\mathbb{H}$, the solution $X_t$ of Eq.~(\ref{eqSPDE}) is extinct in finite time almost surely. Furthermore, the extinction time satisfies the estimates
\begin{equation*}
\mathbb{P}(\tau_{e}>T) \leq \frac{c_0\|x\|_{\mathbb{H}}^{2-\alpha}}{T}
\end{equation*}
and
\begin{equation*}
\mathbb{E}\tau_{e} \leq c_0 \|x\|_{\mathbb{H}}^{2-\alpha},
\end{equation*}
where $c_0$ is an explicit constant.
\end{theorem}

As an application of our general results, we consider the following stochastic $p$-Laplace equation with nonlinear sources
\begin{equation}\label{I07}
dX_t=\text{div}(|\nabla X_t|^{p-2}\nabla X_t)dt+ \lambda X_t^{2}dt+\mathcal{B}(t,X_t)dW_t,
\end{equation}
where  $\lambda=\pm1$, Eq.~\eref{I07} is degenerate if $p > 2$ or singular if $1 < p < 2$.
Nonlinear parabolic equation like \eref{I07} appears in various applications. For instance, in combustion theory, the function $X_t$ represents the temperature, the term $\text{div}(|\nabla X_t|^{p-2}\nabla X_t)$ represents the thermal diffusion, and the nonlinear source $X_t^{2}$ is physically called the ``hot source'', while the source $- X_t^{2}$ is known as the ``cool source''. These different sources have completely different influences on the properties of solutions (cf.~e.g.~\cite{CY07,Nguyen,TY08}).
In the deterministic setting (i.e.,~$\mathcal{B}(t,u)\equiv0$) under the singular regime  $1 < p < 2$  and $\lambda=1$, solutions may blow up in finite time for some initial values, thereby there is no global-in-time solution  (see~e.g.~\cite {LC03,Nguyen}). In this work,  we demonstrate that  the addition of suitable random noise can effectively prevent such  blow-up.

More importantly, we discover a new phenomenon: for a deterministic model that  exhibits blow-up,  appropriate stochastic perturbations lead to the finite-time extinction of the  stochastic counterpart. This  result may shed new insight how noise affects the long-time behavior of stochastic systems.

\subsection{Proof strategy}
Now, we outline the main ideas  of the proof. First, to construct a probabilistically weak solution to \eqref{eqSPDE}, we combine a stochastic compactness argument with techniques from the theory of pseudo-monotone operators. In contrast to existing works (cf.~\cite{LR2,RSZ1}), under the generalized coercivity condition \eqref{I05}, we cannot guarantee that the Galerkin approximating solutions $\{X^{(n)}\}_{n\in\mathbb{N}}$ have finite second moments. Instead, we establish energy bounds of order $2-\eta_0$ with $\eta_0\in(1,2)$ for $\{X^{(n)}\}_{n\in\mathbb{N}}$  by constructing a suitable Lyapunov function. Then, by employing a refined stopping time argument, we prove tightness  in the space
$$\mathcal{Z}^1_T:=C([0,T];\mathbb{V}^*)\cap L^{\alpha}([0,T];\mathbb{H})\cap L^{\alpha}_w([0,T];\mathbb{V}),$$
where $ L^{\alpha}_w([0,T];{\mathbb{V}})$  denotes the space $L^{\alpha}([0,T];{\mathbb{V}})$  endowed with the weak topology.

Notably, the lack of finite second moments for $\{X^{(n)}\}_{n\in\mathbb{N}}$ introduces non-trivial technical difficulties in the proof. Specifically, we cannot follow the arguments in \cite{LR2,RSZ1} to obtain the weak convergence of the sequence $\{\mathcal{A}(\cdot,X^{(n)}_{\cdot})\}_{n\in\mathbb{N}}$ or the convergence of the sequence $\{\mathcal{B}(\cdot, X^{(n)}_{\cdot})\}_{n\in\mathbb{N}}$. To address this, we establish the tightness of $\{\mathcal{A}(\cdot,X^{(n)}_{\cdot})\}_{n\in\mathbb{N}}$ in the space
$$\mathcal{Z}^2_T:=L^{\frac{\alpha}{\alpha-1}}_w([0,T];\mathbb{V}^*).$$
Note that $\mathcal{Z}^1_T$ and $\mathcal{Z}^2_T$ are not Polish spaces. In this case, we apply the Jakubowski's  generalization of the Skorokhod representation theorem for nonmetric spaces, as presented by  Brze\'{z}niak and Ondrej\'{a}t \cite{BO}. This yields the almost sure convergence of both
$$\{X^{(n)}\}_{n\in\mathbb{N}}~\text{and}~\{\mathcal{A}(\cdot,X^{(n)}_{\cdot})\}_{n\in\mathbb{N}}$$ to limit elements $\tilde{X}$ in $\mathcal{Z}^1_T$ and $\tilde{\mathcal{A}}(\cdot)$ in $\mathcal{Z}^2_T$, respectively,  on a new probability space. To justify this step, it is crucial to prove that  $\mathcal{Z}^1_T$ and $\mathcal{Z}^2_T$ are  standard Borel spaces under an appropriate topology (cf.~Remark \ref{k2t} and Lemma \ref{thsb} for details).
Then, using a cut-off technique, we identify the limit and establish the strong convergence of
 $\mathcal{B}(\cdot, X^{(n)}_{\cdot})$ to $\mathcal{B}(\cdot, \tilde X_{\cdot})$.
Additionally, by exploiting the  pseudo-monotonicity of $\mathcal{A}(t,\cdot)$, we establish the weak convergence of  $\mathcal{A}(\cdot, X^{(n)}_{\cdot})$ to $\mathcal{A}(\cdot, \tilde X_{\cdot})$.

However, it is noted that we only have estimates of order $2-\eta_0$ for $\tilde X_t$,  the classical It\^{o}'s formula (cf.~Theorem 4.2.5 in \cite{LR1}) within the variational framework cannot be directly applied. To overcome this difficulty, we utilize a refined stopping time technique combined with a  localizaion procedure.   Based on the It\^{o}'s formula, we prove that $\tilde{X}\in C([0,T];\mathbb{H})$, which establishes that $\tilde{X}$ is a probabilistically weak solution. Then, the existence of a unique probabilistically strong solution follows from pathwise uniqueness and the Yamada--Watanabe theorem.

\subsection{Structure of the paper} The rest of paper is organized as follows. In Sect.~\ref{mainR}, we introduce  the main results about the existence, uniqueness, Feller property as well as finite-time extinction in Theorems \ref{th1}-\ref{th3}, respectively. Then in Sec.~\ref{exam}, we apply our general framework to concrete examples to illustrate the  wide applicability of the main results. In Sect.~\ref{proofM}, we give the proofs of Theorems \ref{th1}-\ref{th3}. We also recall some useful lemmas  in the Appendix.
Throughout this paper, $C_{p}$  denotes some positive constant which may change from line to line, where the subscript $p$ is used to emphasize that the constant depends on certain parameter $p$.

\section{Main results}\label{mainR}

We  recall  some fundamental definitions and notations that are frequently used in the paper.

\vspace{1mm}
For any Banach space $(\mathbb{B},\|\cdot\|_{\mathbb{B}})$, we denote by $\mathbb{C}_T(\mathbb{B}):=C([0,T];\mathbb{B})$ the space of all continuous functions from $[0,T]$ to $\mathbb{B}$, which is a Banach space equipped with  the uniform norm given by
$$\|u\|_{\mathbb{C}_T(\mathbb{B})}:=\sup_{t\in[0,T]}\|u_t\|_{\mathbb{B}},~u\in \mathbb{C}_T(\mathbb{B}).$$
Let $\mathscr{B}_b(\mathbb{B})$ (resp.~$C_b(\mathbb{B})$) be the space of all bounded and Borel measurable (resp.~continuous) functions on $\mathbb{B}$.

In this paper, we will employ the theory of pseudo-monotone operators. To this end, we first recall the definition of the pseudo-monotone operator. For abbreviation, we use the notation  ``$\rightharpoonup$'' for the weak convergence in a Banach space.

\begin{definition}\label{de2} An operator $\mathcal{A}$ from $\mathbb{V}$ to $\mathbb{V}^*$ is called to be a pseudo-monotone operator if $u_n\rightharpoonup u$ in $\mathbb{V}$ and
$$\liminf _{n \to \infty}\,_{\mathbb{V}^*}\langle \mathcal{A}(u_{n}), u_{n}-u\rangle_{\mathbb{V}} \geq 0,$$
then for any $v \in \mathbb{V}$,
$$\limsup _{n \to \infty}\,_{\mathbb{V}^*}\langle \mathcal{A}(u_{n}), u_{n}-v\rangle_{\mathbb{V}} \leq \,_{\mathbb{V}^*}\langle \mathcal{A}(u), u-v\rangle_{\mathbb{V}}.$$
\end{definition}

\begin{remark}
Note that Browder \cite{Browder1977} introduced a different definition  of pseudo-monotone operator: An operator $\mathcal{A}$ from $\mathbb{V}$ to $\mathbb{V}^*$ is called  pseudo-monotone if $u_n\rightharpoonup u$ in $\mathbb{V}$ and
$$\liminf _{n \to \infty}\,_{\mathbb{V}^*}\langle \mathcal{A}(u_{n}), u_{n}-u\rangle_{\mathbb{V}} \geq 0$$
implies $\mathcal{A}(u_n)\rightharpoonup\mathcal{A}(u)$ in $\mathbb{V}^*$ and
$$\lim _{n \to \infty}\,_{\mathbb{V}^*}\langle \mathcal{A}(u_{n}), u_{n}\rangle_{\mathbb{V}}=\,_{\mathbb{V}^*}\langle \mathcal{A}(u), u\rangle_{\mathbb{V}}.$$
This definition turns out to be equivalent to Definition \ref{de1} (cf.~\cite[Remark 5.2.12]{LR1}).
\end{remark}

\subsection{Well-posedness and Feller property}\label{sec.2.2a}
In this subsection, we consider the well-posedness and Feller property of SPDEs (\ref{eqSPDE}).
To this end, we first recall the (probabilistically) weak and strong solutions to SPDEs (\ref{eqSPDE}) as follows.
\begin{definition}\label{dew} $($Weak solution$)$ A pair $(X,W)$ is called a (probabilistically) weak solution to SPDE (\ref{eqSPDE}), if there exists  $(\Omega,\mathscr{F},\{\mathscr{F}_t\}_{t\in[0,T]},\mathbb{P})$ such that $X$ is an $\{\mathscr{F}_t\}$-adapted process and  $W$ is an $U$-valued cylindrical Wiener process on $(\Omega,\mathscr{F},\{\mathscr{F}_t\}_{t\in[0,T]},\mathbb{P})$ and the following holds:

\vspace{2mm}
(i) $X\in \mathbb{C}_T(\mathbb{H})$, $\mathbb{P}$-a.s.;

\vspace{2mm}
(ii) $\int_0^T\|\mathcal{A}(s,X_s)\|_{\mathbb{V}^*}ds+\int_0^T\|\mathcal{B}(s,X_s)\|_{L_2(U;\mathbb{H})}^2ds<\infty$, $\mathbb{P}$-a.s.;

\vspace{2mm}
(iii)
$X_t=X_0+\int_0^t \mathcal{A}(s,X_s)ds+\int_0^t \mathcal{B}(s,X_s)dW_s,t\in[0,T],\mathbb{P}\text{-a.s.}$
holds in ${\mathbb{V}}^*$.
\end{definition}

\begin{definition}\label{de1}
(Strong solution) We say that there exists a (probabilistically) strong solution to (\ref{eqSPDE}) if for every probability space $(\Omega,\mathscr{F},\{\mathscr{F}_t\}_{t\in[0,T]},\mathbb{P})$ with an $U$-valued cylindrical Wiener process $W$, there exists
an $\{\mathscr{F}_t\}$-adapted process $X$ such that properties (i)-(iii) in Definition \ref{dew} hold.
\end{definition}

In this part, we suppose that there are some constants $\alpha>1$,  $\beta\geq 2$, and $C,\delta>0$ such that the following conditions hold for a.e. $t\in[0,T]$.

\vspace{1mm}
\begin{enumerate}

\item [$(\mathbf{A_1})$]
 $($Hemicontinuity$)$ For any $u,v,w\in {\mathbb{V}}$, the map
\begin{equation*}
\mathbb{R}_{+}\ni\lambda\mapsto~_{{\mathbb{V}}^*}\langle \mathcal{A}(t,u+\lambda v),w\rangle_{\mathbb{V}}
\end{equation*}
is continuous.

\item [$(\mathbf{A_2})$]\label{H2}
 $($Local Monotonicity$)$ For any $u,v\in {\mathbb{V}}$,
\begin{eqnarray*}
\!\!\!\!\!\!\!\!&&2_{{\mathbb{V}}^*}\langle \mathcal{A}(t,u)-\mathcal{A}(t,v),u-v\rangle_{\mathbb{V}}+\|\mathcal{B}(t,u)-\mathcal{B}(t,v)\|_{L_2(U,{\mathbb{H}})}^2
\nonumber\\
\!\!\!\!\!\!\!\!&&\leq
(C+\rho(u)+\eta(v))\|u-v\|_{\mathbb{H}}^2,
\end{eqnarray*}
where $\rho,\eta:{\mathbb{V}}\to [0,\infty)$ are  measurable functions satisfying
\begin{equation}\label{esq22}
\rho(u)+\eta(u)\leq C(1+\|u\|_{\mathbb{V}}^{\alpha})(1+\|u\|_{\mathbb{H}}^{\beta}),~u\in {\mathbb{V}}.
\end{equation}

\item [$(\mathbf{A_3})$]
 $($Generalized Coercivity$)$ For any $u\in \mathbb{V}$,
\begin{eqnarray*}
2_{\mathbb{V}^*}\langle \mathcal{A}(t,u),u\rangle_{\mathbb{V}}+\delta\|u\|_{\mathbb{V}}^\alpha\leq g(\|u\|_{\mathbb{H}}^2)+C,
\end{eqnarray*}
where $g:[0,\infty)\to[0,\infty)$ is a non-decreasing continuous function.

\item [$(\mathbf{A_4})$]
 $($Growth$)$ For any $u\in \mathbb{V}$,
\begin{equation}\label{conA3}
\|\mathcal{A}(t,u)\|_{{\mathbb{V}}^*}^{\frac{\alpha}{\alpha-1}}\leq C(1+\|u\|_{\mathbb{V}}^{\alpha} )(1+\|u\|_{{\mathbb{H}}}^{\beta}).
\end{equation}

%$($Structure of $B$$)$
\item [$(\mathbf{A_5})$]  There exists a constant $\eta_0\in(1,\min\{2,\alpha\})$ such that for any $u\in \mathbb{V}$,
\begin{equation}\label{conb1}
g(\|u\|_{\mathbb{H}}^2)+\|\mathcal{B}(t,u)\|_{L_2({U},{\mathbb{H}})}^2
\leq C(1+\|u\|_{\mathbb{H}}^2)+\eta_0\frac{\|\mathcal{B}(t,u)^*u\|_{U}^2}{(1+\|u\|_{\mathbb{H}}^2)},
\end{equation}
where the function $g$ is the same as in $(\mathbf{A_3})$, and for any $u\in \mathbb{V}$,
\begin{equation}\label{conb2}
\|\mathcal{B}(t,u)\|_{L_2({U},{\mathbb{H}})}^2\leq C(1+\|u\|_{{\mathbb{H}}}^{\beta}).
\end{equation}

Moreover, for any sequence $\{u_n\}_{n\in\mathbb{N}}$ and $u$ in $\mathbb{V}$ with $\|u_n-u\|_{\mathbb{H}}\to0$,
\begin{equation}\label{conb3}
\|\mathcal{B}(t,u_n)-\mathcal{B}(t,u)\|_{L_2({U},{\mathbb{H}})}\to 0.
\end{equation}

\end{enumerate}

We present the main result concerning the global existence, uniqueness, and Markov property of (probabilistically) strong solutions to  SPDE (\ref{eqSPDE}).
\begin{theorem}\label{th1}
Suppose that the embedding $ \mathbb{V}\subset \mathbb{H}$ is compact and that $(\mathbf{A_1})$-$(\mathbf{A_5})$ hold.
For any initial data $x\in\mathbb{H}$,
Eq.~(\ref{eqSPDE}) has a unique strong solution in the sense of Definition \ref{de1}. Moreover, we have the estimates
\begin{equation}\label{apri0}
\sup_{t\in[0,T]}\mathbb{E}\|X_t\|_{\mathbb{H}}^{2-\eta_0}+\mathbb{E}\int_0^T\|X_t\|_{\mathbb{V}}^{\alpha-\eta_0}dt<\infty
\end{equation}
and for any $p\geq 2$,
\begin{equation}\label{apri6}
\mathbb{P}\Bigg(\sup_{t\in[0,T]}\|X_t\|_{\mathbb{H}}^p+\int_0^T\|X_t\|_{\mathbb{V}}^{\alpha}dt<\infty\Bigg)=1.
	\end{equation}
Furthermore, if $\mathcal{A}(t,\cdot),\mathcal{B}(t,\cdot)$  are  independent of $t\in[0,T]$, then the solution $(X_t)_{t\in[0,T]}$ of Eq.~(\ref{eqSPDE}) is a time-homogenous Markov process.
\end{theorem}

\begin{remark}\label{b.001}

 The main results are applicable to various new stochastic models that cannot be treated by the classical variational framework established in \cite{LR2,NTT21, RSZ1,W26}, such as  stochastic 3D Navier-Stokes equations and stochastic surface growth models. Although these models fail to satisfy the classical coercivity condition \eref{I04}, they do fit the generalized coercivity condition $(\mathbf{A_3})$ (see Section \ref{exam} for details).

\end{remark}

\begin{remark}\label{rem01a}
The generalized coercivity condition $(\mathbf{A_3})$ was first introduced by the last named author and R\"{o}ckner in \cite{LR13,LR1}, where they proved that Eq.~\eref{eqSPDE} has a unique local solution under this  condition in the case of additive noise (i.e.~$\mathcal{B}(t,u)=\mathcal{B}(t)$). Note that  the main idea of the proof in \cite{LR13,LR1}  is based on a shift transformation to reduce the SPDE \eref{eqSPDE} to a deterministic evolution equation with random coefficients. However,  this strategy cannot handle the more general nonlinear multiplicative noise. More importantly, compared to \cite{LR13,LR1}, we establish the global existence and uniqueness of solutions, rather than merely local ones.

\end{remark}

\begin{remark}\label{b3}
Condition \eref{conb1} in $(\mathbf{A_5})$ plays an essential role in this work, as it balances the interplay between the effective component of the noise in the energy direction  and the growth of nonlinear terms.
This condition can be interpreted as follows:
\begin{enumerate}[(i)]
\item
When $g(x)=C_0x$, corresponding to the classical coercivity condition \eref{I04}, the diffusion coefficient $\mathcal{B}(t,\cdot)$ can exhibit both linear and superlinear growth.

\vspace{1mm}
\item
When $g(x)=C_0x^p$ with $p>1$, the  coefficient $\mathcal{B}(t,\cdot)$  is required to exhibit superlinear growth.
\end{enumerate}
 This condition captures an anisotropic stabilizing mechanism: only the component of the noise acting in the energy direction contributes to the prevention of blow-up, see Remark \ref{remark222} below for concrete applications in fluid mechanics.
\end{remark}

%\begin{remark}
%As a consequence of  the proof of Theorem \ref{th1}, in fact we can obtain a more general result for the existence of weak solutions to (\ref{eqSPDE}) directly. More precisely,
%if the embedding $ \mathbb{V}\subset \mathbb{H}$ is compact, $(\mathbf{A_3})$-$(\mathbf{A_5})$ hold and $\mathcal{A}(t,\cdot)$ is pseudo-monotone from $\mathbb{V}$ to $\mathbb{V}^*$ for a.e.~$t\in[0,T]$, then there exists a weak solution to (\ref{eqSPDE}) and the estimate (\ref{apri0}) holds.
%
%
%We remark that if the Gelfand triple (\ref{gel}) reduces to the case of finite dimensions, i.e. $\mathbb{V}=\mathbb{H}=\mathbb{R}^d$, the pseudo-monotonicity of $\mathcal{A}(t,\cdot)$ is  equivalent to the continuity of $\mathcal{A}(t,\cdot)$. Thus, the above existence result of weak solutions is consistent with the classical theory in finite-dimensional SDEs  (cf.~\cite[Theorem C.3]{GRZ}).
%\end{remark}

Based on the existence and uniqueness of solutions to (\ref{eqSPDE}), we intend to investigate the continuous dependence on the initial data in probability (in other word, the well-posedness of Eq.~(\ref{eqSPDE})) and the Feller property of the associated  transition semigroup.

\vspace{1mm}
For any $\varphi\in \mathscr{B}_b(\mathbb{H})$, $t\geq 0$, we define a function $\mathcal{T}_t\varphi:\mathbb{H}\to \mathbb{R}$ by
$$\mathcal{T}_t\varphi(x):=\mathbb{E}\varphi(X_t(x)),~x\in \mathbb{H},$$
where  $X_t(x)$ is the  solution to (\ref{eqSPDE}) with the initial data $x$.
\begin{remark}
Based on Theorem \ref{th1}, it is a direct consequence that $(\mathcal{T}_t)_{t\geq 0}$ is a stochastically continuous Markov semigroup on $\mathscr{B}_b(\mathbb{H})$.
\end{remark}

\begin{theorem}\label{th2}
Suppose that the embedding $ \mathbb{V}\subset \mathbb{H}$ is compact and that $(\mathbf{A_1})$-$(\mathbf{A_5})$ hold. In addition, suppose that for any $u,v\in \mathbb{V}$,  $\mathcal{B}$ is locally Lipschitz in the sense that
\begin{equation}\label{conb4}
\|\mathcal{B}(t,u)-\mathcal{B}(t,v)\|_{L_2({U},{\mathbb{H}})}^2\leq (C+\rho(u)+\eta(v))\|u-v\|^{2}_{{\mathbb{H}}},
\end{equation}
where functions $\rho,\eta$ are the same as in $(\mathbf{A_2})$. Let $\{x_n\}_{n\in\mathbb{N}}$ and $x$ be a sequence with $\|x_n-x\|_{\mathbb{H}}\to 0$.
Then
\begin{equation}\label{apri3}
\|X(x_n)-X(x)\|_{\mathbb{C}_T(\mathbb{H})} \to 0~~\text{in probability as}~~n\to\infty.
\end{equation}
Furthermore, $(\mathcal{T}_t)_{t\geq 0}$ is a Feller semigroup, i.e., $\mathcal{T}_t$ maps $C_b(\mathbb{H})$ into itself.
\end{theorem}

\begin{remark}
Since we have established the Feller property of the Markov semigroup $(\mathcal{T}_t)_{t\geq 0}$ associated with (\ref{eqSPDE}), a natural and  important question arises investigating  the existence of an invariant measure. One commonly employed  technique for establishing the existence of an invariant measure is the Krylov-Bogolyubov procedure. According to  the criterion presented by Maslowski and Seidler \cite{MS}, we need to prove the following estimate: for any $\varepsilon>0$ there exists $R>0$ such that
\begin{equation}\label{ipm}
\sup_{T\geq 1}\frac{1}{T}\int_0^T\mathbb{P}\big(\|X_t(x)\|_{\mathbb{H}}>R\big)dt<\varepsilon.
\end{equation}
However, it is not sufficient to obtain (\ref{ipm}) based on the estimate (\ref{apri}) below.  This topic  deserves further investigation in the future work.

\end{remark}

\subsection{Finite-time extinction} \label{sec.finite}
The phenomenon of finite-time extinction for stochastic fast diffusion equations (SFDEs) driven by linear multiplicative noise has been rigorously established in a series of pioneering works by Barbu, Da Prato, and R\"{o}ckner~\cite{BDR09,BDR09b,BR12,BRR15} and Gess \cite{Gess15}. This phenomenon is of profound physical and mathematical interest, as it provides a rigorous description of the relaxation dynamics characteristic of SOC.

Concerning finite-time extinction, the problem  is generally divided into the following three cases:
\begin{enumerate}[(i)]

\item
Extinction with positive probability for small initial conditions;

\vspace{1mm}
\item
Extinction with positive probability for all initial values;

\vspace{1mm}
\item
Extinction with probability one for all initial values.
\end{enumerate}
While Regimes (i) and (ii) present mathematically intriguing local or conditional behaviors, they do not fully capture the physical mechanism of SOC. In physical systems exhibiting SOC, relaxation to a stable state must be a deterministic event rather than a stochastic occurrence. Consequently, Regime (iii) aligns most closely with the physical paradigm of SOC, as it guarantees that the system always decays to zero in finite time almost surely.

The primary objective of this subsection is to establish a unified and general framework to investigate the almost sure finite-time extinction (i.e., Regime (iii)) of solutions to  SPDEs \eqref{eqSPDE} under the influence of potentially linear or nonlinear multiplicative noise. Specifically, one may anticipate that a structurally suitable  noise will eventually dominate the underlying deterministic nonlinearities, thereby accelerating the decay of the system and driving it to extinction within a finite time horizon.

To explore this finite-time extinction, we  suppose that there exist constants $\alpha\in(1,2)$ and $\delta>0$ such that the following conditions hold for a.e.~$t\in[0,T]$.

\vspace{1mm}
\begin{enumerate}
\item [$(\mathbf{A_3^*})$]
 $($Enhanced Coercivity$)$ For any $u\in \mathbb{V}$,
\begin{eqnarray*}
2_{\mathbb{V}^*}\langle \mathcal{A}(t,u),u\rangle_{\mathbb{V}}+\delta\|u\|_{\mathbb{V}}^\alpha\leq g(\|u\|_{\mathbb{H}}^2),
\end{eqnarray*}
where the function $g$ is the same as in  $(\mathbf{A_3})$ with $g(0)=0$.

\item [$(\mathbf{A_5^*})$]  For any $u\in \mathbb{V}$,
\begin{equation*}\label{conb5}
\|\mathcal{B}(t,u)\|_{L_2(U,\mathbb{H})}=0~~\text{if}~~\|u\|_{\mathbb{H}}=0,
\end{equation*}
and
\begin{equation}\label{conb6}
\big(g(\|u\|_{\mathbb{H}}^2)+\|\mathcal{B}(t,u)\|_{L_2(U,\mathbb{H})}^2\big)\|u\|_{\mathbb{H}}^2\leq \alpha\|\mathcal{B}(t,u)^*u\|_{U}^2.
\end{equation}

\end{enumerate}

\begin{remark}
Note that Assumption (\ref{conb6}) is not comparable with Assumption (\ref{conb1}). However, by choosing $g(x)=C_0x^p$, it is easy to select a suitable diffusion coefficient   satisfying both assumptions simultaneously, see Subsection \ref{example lap} for details.
\end{remark}

Let $\tau_{e}$ be the following extinction time
$$\tau_{e}:=\inf\big\{t\geq 0:\|X_t\|_{\mathbb{H}}=0\big\},$$
where $(X_t)_{t\geq 0}$ is the solution to (\ref{eqSPDE}) given by Theorem \ref{th1} with initial value $x\in\mathbb{H}$.

\begin{theorem}\label{th3}
Suppose that the embedding $ \mathbb{V}\subset \mathbb{H}$ is compact and that $(\mathbf{A_1})$, $(\mathbf{A_2})$, $(\mathbf{A_3^*})$, $(\mathbf{A_4})$, $(\mathbf{A_5})$, and $(\mathbf{A_5^*})$ hold.
Let the initial data $x\in\mathbb{H}$.
Then for any $t\geq \tau_{e}$,
\begin{equation}\label{fte1}
\|X_t\|_{\mathbb{H}}=0,~\mathbb{P}\text{-a.s.}.
\end{equation}
Furthermore, there exists a constant $c_0>0$ such that for any $T>0$,
\begin{equation}\label{fte2}
\mathbb{P}(\tau_{e}>T)
\leq\frac{c_0\|x\|_{\mathbb{H}}^{2-\alpha}}{T}.
\end{equation}
Moreover, we also have the moment estimate of the extinction time
\begin{equation}\label{fte3}
\mathbb{E}\tau_{e}\leq c_0 \|x\|_{\mathbb{H}}^{2-\alpha}.
\end{equation}

\end{theorem}

\begin{remark}
(i) Compared with the existing works \cite{BDR09,BDR09b,BR12,BRR15}, the above theorem provides quantitative estimates for the extinction time, which derive the almost sure finite-time extinction (i.e., Regime (iii)) of solutions to  SPDEs \eqref{eqSPDE}.  Furthermore, the constant $c_0$ has an explicit  representation depending on the embedding constant of  $\mathbb{V}\subset\mathbb{H}$ (see Subsection \ref{sec4.6} for details).

As a result, while previous studies (e.g.~\cite{BDR09,BDR09b,BR12,BRR15,Gess15}) were limited to linear multiplicative noise,  we establish, for the first time, finite-time extinction for a class of quasilinear SPDEs driven by nonlinear multiplicative noise, encompassing stochastic fast diffusion equations and stochastic singular
$p$-Laplace equations (with source terms).

\vspace{1mm}
(ii)  In a forthcoming work, we aim to investigate finite-time extinction for SOC models (e.g., the BTW model and the Zhang model) perturbed by nonlinear  multiplicative noise, which describe dynamical systems possessing a critical point as an attractor.
\end{remark}

\section{Examples/Applications}\label{exam}

In this section, we will denote by $\Lambda\subseteq\mathbb{R}^d$ ($d\geq 1$) an open bounded domain with a smooth boundary.  Let
$C_0^\infty(\Lambda; \mathbb{R}^d)$ be the space of all infinitely differentiable functions from $\Lambda$ to $\mathbb{R}^d$ with compact support.  For $p\ge 1$, let $L^p(\Lambda; \mathbb{R}^d)$ denote the vector valued $L^p$-space with the norm $\|\cdot\|_{L^p}$.
For each integer $m\geq0$, we use $W_0^{m,p}(\Lambda; \mathbb{R}^d)$ to denote the classical Sobolev space defined on $\Lambda$
taking values in $\mathbb{R}^d$ with the equivalent norm:
$$ \|u\|_{W^{m,p}} := \left(  \int_{\Lambda} |D^m u(\xi)|^pd \xi \right)^\frac{1}{p}.$$
In particular, we denote
$$ \|u\|_m:=\|u\|_{W^{m,2}}.$$

We also recall the standard Gagliardo-Nirenberg interpolation inequality (cf.~\cite{Ni59}) for the reader's convenience.
If for any $1\leq q,r<\infty$, and $0\leq n<m$ satisfying
$$
\frac{1}{p}=\frac{n}{d}+\theta(\frac{1}{r}-\frac{m}{d})+(1-\theta)\frac{1}{q},\ \frac{n}{m}\le\theta<1,
$$
then there is a constant $C>0$ such that
\begin{equation}
\|u\|_{W^{n,p}}\le C\|u\|_{W^{m,r}}^{\theta}\|u\|_{L^{q}}^{1-\theta},\ \ u\in L^q(\Lambda; \mathbb{R}^d)\cap W^{m,2}(\Lambda; \mathbb{R}^d).\label{GN_inequality}
\end{equation}

%\vspace{1mm}
%In the following, we also present the numerical simulations to visually demonstrate the regularization effect of the nonlinear noise in Eq.~(\ref{sde02}).
%
%\begin{remark}\label{sde05}
%We first show the numerical simulations of solutions to ODE (\ref{sde01}) and SDE~(\ref{sde02}) driven by linear multiplicative noise ($c_0=1, m=1$) in the following figures
%\begin{figure}[H]
%\centering
%\begin{minipage}{0.5\textwidth}
%  \centering
%  \includegraphics[width=\linewidth]{0-0-eps-converted-to.pdf}
% \newline $\frac{dX_t}{dt}=X_t^2,~x=1$
%  %\label{fig:image1}
%\end{minipage}%
%\begin{minipage}{0.5\textwidth}
%  \centering
%  \includegraphics[width=\linewidth]{1-1-eps-converted-to.pdf}
%\newline $dX_t=X_t^2dt+X_tdW_t,~x=1$
%  %\label{fig:image2}
%\end{minipage}
%\end{figure}
%
%
%Next, we present the numerical simulations of SDE~(\ref{sde02}) driven by superlinear multiplicative noise ($c_0=1, m=2$), from which we can see the regularization effect of superlinear  noise.
%\begin{figure}[H]
%      \centering
%      \includegraphics[width=0.65\textwidth,height=0.45\textwidth]{2-15-eps-converted-to.pdf}
%       \newline $dX_t=X_t^2dt+X_t^{2}dW_t,~x=1$
%
%\end{figure}
%\end{remark}

We now apply our general framework to various SPDE models. To this end, we specify the conditions on the diffusion coefficient $\mathcal{B}$ with respect to a given Gelfand triple \eqref{gel} and a non-decreasing continuous function $g$.

\vspace{1mm}
Specifically, we assume there exist constants   $C>0,\beta\geq 2$, and $\eta_0,\alpha\in(1,2)$ such that
\begin{enumerate}
\item [$(\mathbf{A}^1_{\mathcal{B}})$]
 For any $u,v\in {\mathbb{V}}$,
	\begin{equation*}
		\|\mathcal{B}(t,u)-\mathcal{B}(t,v)\|_{L_2(U,{\mathbb{H}})}^2\leq
		(C+\|u\|_{\mathbb{H}}^{\beta}+\|v\|_{\mathbb{H}}^{\beta})\|u-v\|_{\mathbb{H}}^2
	\end{equation*}
and
\begin{equation*}
	\|\mathcal{B}(t,u)\|_{L_2({U},{\mathbb{H}})}^2\leq C(1+\|u\|_{{\mathbb{H}}}^{\beta}).
\end{equation*}

\item [$(\mathbf{A}^2_{\mathcal{B}})$]  For any $u\in \mathbb{V}$,
\begin{equation*}
	g(\|u\|_{\mathbb{H}}^2)+\|\mathcal{B}(t,u)\|_{L_2({U},{\mathbb{H}})}^2
	\leq C(1+\|u\|_{\mathbb{H}}^2)+\eta_0\frac{\|\mathcal{B}(t,u)^*u\|_{U}^2}{(1+\|u\|_{\mathbb{H}}^2)}.
\end{equation*}

\item [$(\mathbf{A}^{2*}_{\mathcal{B}})$]  For any $u\in \mathbb{V}$,
\begin{equation*}
	\|\mathcal{B}(t,u)\|_{L_2(U,\mathbb{H})}=0~~\text{if}~~\|u\|_{\mathbb{H}}=0,
\end{equation*}
and
\begin{equation*}
	\big(g(\|u\|_{\mathbb{H}}^2)+\|\mathcal{B}(t,u)\|_{L_2(U,\mathbb{H})}^2\big)\|u\|_{\mathbb{H}}^2\leq \alpha\|\mathcal{B}(t,u)^*u\|_{U}^2.
\end{equation*}

\end{enumerate}

As discussed in Remark \ref{b.001}, a typical choice of function  $g$ in infinite-dimensional applications is
$$g(x)=C_0x^p~~\text{with}~~ p>1,C_0>0.$$
 Below, we present concrete examples of the diffusion coefficient $\mathcal{B}$ that satisfy conditions $(\mathbf{A}^1_{\mathcal{B}})$-$(\mathbf{A}^{2*}_{\mathcal{B}})$.

\vspace{1mm}
\noindent\textbf{Example.}
The diffusion coefficient can be constructed as follows:
	
	\vspace{1mm}
	(i) For $y\in U$, we define
	\begin{equation}\label{exaB}
		\mathcal{B}(t,u)y:=\sum_{k=1}^{\infty}b_k\|u\|_\mathbb H^m u\langle y,g_k\rangle_{U},
	\end{equation}
	where  $\{g_1,g_2,\cdots\}$ is an orthonormal basis of $U$, and the constants $b_k$, $k\geq 1$, satisfy
	$$\gamma:=\sum_{k=1}^{\infty}b_k^2<\infty.$$
	For $m=p-1$, we  assume that $\gamma$ is a large constant depending on $C_0$. For $m>p-1$, we only need to assume that there is $k\geq 1$ such that $b_k\neq0$.
	
	\vspace{1mm}
For the reader's convenience, we provide a detailed proof showing that Example \ref{exaB} satisfies condition $(\mathbf{A}^2_{\mathcal{B}})$ when $m=p-1$.
	
	\begin{proof}
  By the definition of the adjoint operator $\mathcal{B}(t,u)^*$, we have
		\begin{eqnarray*}
			\|\mathcal{B}(t,u)^*u\|_{U}^2=\!\!\!\!\!\!\!\!&&\langle\mathcal{B}(t,u)^*u,\mathcal{B}(t,u)^*u\rangle_{U}   \\
			=\!\!\!\!\!\!\!\!&&\langle u,\mathcal{B}(t,u)\mathcal{B}(t,u)^*u\rangle_{\mathbb{H}}   \\
			=\!\!\!\!\!\!\!\!&&\sum_{k=1}^{\infty}b_k\|u\|_\mathbb H^{p-1}\langle u, u\rangle_{\mathbb{H}}\langle \mathcal{B}(t,u)^*u,g_k\rangle_{U}   \\
			=\!\!\!\!\!\!\!\!&&\sum_{k=1}^{\infty}b_k\|u\|_\mathbb H^{p+1}\langle u,\sum_{j=1}^{\infty}b_j\|u\|_\mathbb H^{p-1} u\langle g_j,g_k\rangle_{U}\rangle_{\mathbb{H}}   \\
			=\!\!\!\!\!\!\!\!&&\gamma\|u\|_\mathbb H^{2p+2}.
		\end{eqnarray*}
		For a large constant $\gamma$, we can obtain that there exists $\eta_0\in(1,2)$,
		\begin{eqnarray*}
			\big(g(\|u\|_{\mathbb{H}}^2)+\|\mathcal{B}(t,u)\|_{L_2({U},{\mathbb{H}})}^2\big)(1+\|u\|_{\mathbb{H}}^2)\leq\!\!\!\!\!\!\!\!&&(C_0\|u\|_{\mathbb{H}}^{2p}+\gamma\|u\|_\mathbb H^{2p})(1+\|u\|_{\mathbb{H}}^2) \\
			\leq\!\!\!\!\!\!\!\!&&\eta_0\gamma\|u\|_\mathbb H^{2p+2}+C,
		\end{eqnarray*}
		which implies that the condition (\ref{conb1}) in $(\mathbf{A_5})$ holds.
	\end{proof}

	(ii) In particular, this setup applies to the case of a one-dimensional Wiener process.
For $u\in \mathbb{H}$ and $y\in U$, choosing $b_k = c \delta_{1,k}$ with a constant $c \neq 0$ yields
	\begin{equation}
		\mathcal{B}(t,u)y:=c\|u\|_\mathbb H^mu\langle y,g_1\rangle_U.
	\end{equation}
 Similarly, for $m=p-1$, we assume that $c$ is a large constant depending on $C_0$. For any $m>p-1$, we can directly take $c=1$.

\begin{remark}\label{b.02}	
The above example  represents a class of nonlocal forcing terms that characterize the spatial dependence of systems arising in engineering, biology, and fluid mechanics. Such nonlocal mechanisms have been widely investigated in the study of deterministic PDEs, including non-degenerate Kirchhoff-type forcing in wave equations (cf.~\cite{K83}) and Woinowsky-Krieger nonlocal forcing for extensible beams (cf.~\cite{WK50} or \cite[Section 6]{KS18}). We  refer the reader to Remark \ref{remark222} for a more detailed discussion within mathematical fluid mechanics.
\end{remark}

\subsection{Stochastic 3D Navier-Stokes equations}\label{example 1}
3D Navier-Stokes system is a classical model to describe the time evolution of an incompressible fluid (cf.~\cite{Le34,Tem95,ZZ21}), which is given as follows
\begin{eqnarray}\label{NS01}
\left\{
  \begin{aligned}
 &\partial_t{u}=\nu\Delta{u}-(u\cdot{\nabla})u-\nabla{p}+f, \\
    &\text{div}(u)=0, u|_{\partial\Lambda}=0, u(0)=u_0,
  \end{aligned}
\right.
\end{eqnarray}
where  $u$ represents the velocity field of the fluid, $\nu>0$ is the viscosity constant,
$p$ denotes the pressure, and $f$ is an external force field acting on the fluid.

%In the past few decades, 3D Navier-Stokes equations have been intensively studied in the literature, but up to now, their global well-posedness remains still a longstanding and challenging open problem.

Define
$$H^m:=\Big\{v\in W_0^{m,2}(\Lambda;\mathbb{R}^3):\text{div}(v)=0\Big\}.$$
Identifying $H^1$ with its dual space by the Riesz isomorphism, then we will use the following Gelfand triple
\begin{equation} \label{gel1}
\mathbb V:=H^2\subset \mathbb H:=H^1\subset {\mathbb V}^*.
 \end{equation}

Let $\mathcal{P}_{\mathbb H}$ be the orthogonal projection operator on $L^{2}(\Lambda;\mathbb{R}^{3})$ onto $H^0$, which is called the Leray-Helmholtz
projection.  Then, the
classical 3D Navier-Stokes system \eref{NS01} can be reformulated in the following abstract form
\begin{equation}
\partial_t u=Au+B(u)+F,~~~u(0)=u_0,\label{NS02}
\end{equation}
where operators
$$A:\mathbb V\rightarrow {\mathbb V}^*,~~~Au:=\nu P_{\mathbb H}\Delta{u},$$
$$B:\mathbb V\times \mathbb V\rightarrow {\mathbb V}^*,~~~B(u,v):=P_{\mathbb H}[(u\cdot{\nabla})v],~~B(u):=B(u,u)$$
$$F:[0,T]\rightarrow H^0,~~~F_t=\mathcal{P}_{\mathbb H}f_t$$
are well-defined.

Over the past few decades, the 3D Navier-Stokes equations have been intensively studied in the literature. However, their global well-posedness remains a challenging open problem. Motivated by this, we are interested in the following stochastic 3D Navier-Stokes equations:
\begin{eqnarray}\label{NS03}
d{X_t}=\big[AX_t+B(X_t)\big]dt+\mathcal{B}(t,X_t)dW_t,~X_0=x,
\end{eqnarray}
where $\{W_t\}_{t\in [0,T]}$ is an $U$-valued cylindrical Wiener process with $U$ being a separable Hilbert space, and the diffusion coefficient $\mathcal{B}$ satisfies  $(\mathbf{A}^1_{\mathcal{B}})$ and $(\mathbf{A}^2_{\mathcal{B}})$  with $g(x)=Cx^3$ for a large constant $C$.

\begin{theorem}\label{thns}
 For any  initial data $x\in H^1$, Eq.~(\ref{NS03}) has a unique strong solution. Moreover,  we have the  estimate
\begin{equation*}
\sup_{t\in[0,T]}\mathbb{E}\|X_t\|_{1}^{2-\eta_0}+\mathbb{E}\int_0^T\|X_t\|_{2}^{2-\eta_0}dt<\infty.
\end{equation*}
Furthermore, let $X_t(x)$ be the unique solution to Eq.~(\ref{NS03}) with the initial data $x$, $\{x_n\}_{n\in\mathbb{N}}$ and $x$ be a sequence with $\|x_n-x\|_{1}\to 0$,
then
\begin{equation*}
\|X(x_n)-X(x)\|_{\mathbb{C}_T(H^1)} \to 0~~\text{in probability as}~~n\to\infty.
\end{equation*}
In particular, the corresponding Markov semigroup $(\mathcal{T}_t)_{t\geq 0}$ is a Feller semigroup in $C_b(H^1)$.
\end{theorem}

\begin{proof}
By Theorems \ref{th1} and \ref{th2}, we only need to verify that the conditions $(\mathbf{A_1})$-$(\mathbf{A_4})$  hold for Eq.~(\ref{NS03}).

Let
$$\mathcal{A}(t,u):=Au+B(u).$$
The proof of the conditions $(\mathbf{A_1})$-$(\mathbf{A_4})$  for $\mathcal{A}$ is similar to that of Example 5.2.23 in \cite{LR1}, we include it here for completeness.

\vspace{1mm}
The hemicontinuity $(\mathbf{A_1})$ is obvious since $B$ is a bilinear map.

\vspace{1mm}
By the Agmon inequality (see e.g.~(A.29) in~\cite{FMRT}),
\begin{eqnarray}
\|u\|_{L^\infty}\leq C\|u\|_1^{\frac{1}{2}}\|u\|_2^{\frac{1}{2}},~u\in H^2,\label{bb5}
\end{eqnarray}
 we have
\begin{eqnarray*}
\!\!\!\!\!\!\!\!&& {}_{\mathbb{V}^*}\<B(u)-B(v),u-v\>_{\mathbb{V}} \\
 = \!\!\!\!\!\!\!\!&&\<B(u)-B(v), (-\Delta)(u-v)\>_{L^2}  \\
\le\!\!\!\!\!\!\!\!&& \|u-v\|_{2}  \|(u\cdot \nabla)u - (v\cdot \nabla) v\|_{L^2}  \\
\le\!\!\!\!\!\!\!\!&&    \|u-v\|_{2} \left( \|u\|_{L^\infty} \|\nabla u- \nabla v  \|_{L^2} +\|u-v\|_{L^\infty} \|\nabla v\|_{L^2} \right)  \\
\le\!\!\!\!\!\!\!\!&&   \|u-v\|_{2}  \left( \|u\|_{L^\infty} \| u-v  \|_{1} +C\|u-v\|_{2}^{1/2} \|u-v\|_{1}^{1/2} \| v\|_{1} \right)  \\
\le\!\!\!\!\!\!\!\!&&    \frac{\nu}{2} \|u-v\|_{2}^{2} + C \left( \|u\|_{2}\|u\|_{1} +\|v\|_{1}^4 \right)   \|u-v\|_{1}^{2},
\   u, v\in {\mathbb{V}},
 \end{eqnarray*}
then
\begin{equation*}
 \begin{split}
  &   {}_{\mathbb{V}^*}\<Au+B(u)-Av-B(v),u-v\>_{\mathbb{V}}\\
\le& -\frac{\nu}{2} \|u-v\|_{2}^{2}
  +C \left(  \|u\|_{2}\|u\|_{1} +\|v\|_{1}^4 \right)   \|u-v\|_{1}^{2},\ u,v\in {\mathbb{V}},
 \end{split}
\end{equation*}
which implies the local monotonicity $(\mathbf{A_2})$ holds with  $\alpha=2,\beta=m$ and
$$\rho(u):=C\|u\|_{2}\|u\|_{1},
~~\eta(v)=C\|v\|_{1}^{4}.$$
In particular, there exists a constant $C>0$ such that
$$   {}_{\mathbb{V}^*}\<Au+B(u), u\>_{\mathbb{V}}
\le -\frac{\nu}{2} \|u\|_{2}^{2} +  C (1+\|u\|_{1}^6),\   u\in {\mathbb{V}}. $$
Then,   $(\mathbf{A_3})$ holds with $g(x)=C x^3$, i.e.,
\begin{equation*}
   {}_{\mathbb{V}^*}\<\mathcal{A}(t,u), v\>_{\mathbb{V}}\le -\frac{\nu}{2} \|u\|_{2}^2+ g(\|u\|_{1}^2)+C,\ u\in {\mathbb{V}}.
\end{equation*}
Concerning the growth condition, by (\ref{bb5}) we have
\begin{equation*}
  \|B(u)\|_{\mathbb{V}^*}^2\le \|(u\cdot \nabla)u\|_{L^2}^2\le \|u\|_{L^\infty}^2 \|\nabla u\|_{L^2}^2
\le C \|u\|_{2}\|u\|_{1}^3 \le C \|u\|_{2}^2 \|u\|_{1}^2,  \ u\in {\mathbb{V}}.
\end{equation*}
Hence,  $(\mathbf{A_4})$ holds.
 The proof is completed.
\end{proof}

\begin{remark}\label{remark222}
As discussed in Remark \ref{b.02}, the main results established in  Theorem \ref{thns} are applicable to the following class of nonlocal stochastic perturbations
\begin{equation}\label{exaB1}
	\int_0^t\mathcal{B}(s, u_s) dW_s:=\sum_{k=1}^{\infty}\int_0^t b_k u_s\int_{\Lambda}|\nabla u_s(\xi)|^2d\xi  d\beta_k(s),
\end{equation}
where  $\gamma:=\sum_{k=1}^{\infty}b_k^2$ is a large constant, $\{\beta_k(t)\}_{k\in\mathbb{N}}$ is a sequence of independent one-dimensional standard Wiener processes. Physically, the  nonlocal stochastic forcing $(\ref{exaB1})$  characterizes spatially dependent random perturbations whose intensity is modulated by the total energy dissipation of the system.

Such integral-type nonlocal operators plays an important role  in mathematical fluid mechanics. In the celebrated paper \cite{Lad67}, Ladyzenskaja introduced Navier-Stokes equations with this class of deterministic nonlocal viscosity forces, which is related to the total dissipation energy of Newtonian fluid.  Moreover,  as noted in  \cite{CR05}, this type of nonlocal viscosity allows one to interpret the system as the asymptotic limit of certain nonisothermal flows characterized by high thermal conductivity.

\end{remark}

\begin{remark}
(i)	Although stochastic Navier-Stokes equations have been intensively studied over the past few decades (cf.~e.g.~\cite{F15b,O07,RZ09}), the global well-posedness of the stochastic 3D Navier-Stokes equations has not yet been fully resolved. Applying our general framework, we establish that the stochastic 3D Navier-Stokes system \eref{NS03}  driven by a class of nonlinear multiplicative noises is globally well-posed for  general $H^1$ initial data.

(ii)	  It is well established that the 2D Navier-Stokes equations fit the classical local monotonicity framework (cf.~\cite{LR2}) based on the Gelfand triple
	\begin{equation*}
		\mathbb V:=H^1 \subset \mathbb H:= H^0 \subset  \mathbb V^* .
	\end{equation*}
However, this Gelfand triple does not apply to the 3D setting. To address this issue, we introduce an alternative Gelfand triple \eref{gel1} in Theorem \ref{thns}, which serves as the appropriate framework.
\end{remark}

\subsection{Stochastic surface growth models}\label{example 3}
 In this subsection, we apply our general framework
to study the following 1D surface growth model
\begin{eqnarray}\label{surface growth PDE01}
	\left\{
	\begin{aligned}
		&\partial_t{u}= -\partial_x^4 u-\partial_x^2 u+\partial_x^2(\partial_x u)^2, \\
		&u|_{\partial\Lambda}=0, u(0)=u_0,
	\end{aligned}
	\right.
\end{eqnarray}
where $\partial_x, \partial_x^2, \partial_x^4$ denote the first, second and fourth spatial derivatives, respectively. This model appears in the theory of growth of surfaces,
which describes an amorphous material deposited on an initially flat surface in high vacuum (cf.~\cite{RLH}).

Owing to the highly nonlinear nature of model \eref{surface growth PDE01}, global well-posedness for the 1D surface growth model has not yet been resolved in either the deterministic or stochastic settings. In particular, this model is known to share comparable difficulties with the 3D Navier–Stokes equations.

Taking random noise into account, the equation is formulated  as follows
\begin{eqnarray}\label{surface growth PDE}
\left\{
  \begin{aligned}
 &d X_t=\left[ -\partial_x^4 X_t-\partial_x^2 X_t+\partial_x^2(\partial_x X_t)^2 \right] d t +\mathcal{B}(t,X_t)dW_t,\\
    &X_t|_{\partial\Lambda}=0,
  \end{aligned}
\right.
\end{eqnarray}
where the diffusion coefficient $\mathcal{B}$ satisfies  $(\mathbf{A}^1_{\mathcal{B}})$ and $(\mathbf{A}^2_{\mathcal{B}})$  with $g(x)=Cx^3$ for a large constant $C$.

We take the following Gelfand triple to obtain the global well-posedness of Eq.~(\ref{surface growth PDE}), i.e.,
$$  \mathbb V:=W_0^{4,2}(\Lambda;\mathbb{R}) \subset \mathbb H:= W_0^{2,2}(\Lambda;\mathbb{R}) \subset {\mathbb V}^*.   $$

\begin{theorem}\label{thsurface}
 For any  initial data $x\in W_0^{2,2}$, Eq.~(\ref{surface growth PDE}) has a unique strong solution. Moreover,  we have the  estimate
\begin{equation*}
\sup_{t\in[0,T]}\mathbb{E}\|X_t\|_{2}^{2-\eta_0}+\mathbb{E}\int_0^T\|X_t\|_{4}^{2-\eta_0}dt<\infty.
\end{equation*}
Furthermore, let $X_t(x)$ be the unique solution to Eq.~(\ref{surface growth PDE}) with the initial data $x$, $\{x_n\}_{n\in\mathbb{N}}$ and $x$ be a sequence with $\|x_n-x\|_{2}\to 0$,
then
\begin{equation*}
\|X(x_n)-X(x)\|_{\mathbb{C}_T(W_0^{2,2})} \to 0~~\text{in probability as}~~n\to\infty.
\end{equation*}
In particular, the corresponding Markov semigroup $(\mathcal{T}_t)_{t\geq 0}$ is a Feller semigroup in $C_b(W_0^{2,2})$.
\end{theorem}

\begin{proof}
The conditions $(\mathbf{A_1})$-$(\mathbf{A_5})$ hold for Eq.~\eref{surface growth PDE} with  $g(x)=C x^3$, whose proof follows from Section 3.4 in \cite{LR13}, we omit the details.
\end{proof}

\begin{remark}
The solution obtained here for the stochastic surface growth model  is a global strong solution in both  the PDE and probability senses. We note that the local existence of solutions for this model has been established in \cite{LR13,RZZ14}. To the best of our knowledge, Theorem \ref{thsurface} provides the first global well-posedness result for the stochastic surface growth model.
\end{remark}

Our general framework is not restricted to the aforementioned semi-linear SPDEs; it also encompasses a wide class of quasi-linear stochastic dynamics, including stochastic $p$-Laplace equations (with heat sources) and stochastic fast-diffusion equations.
\subsection{Stochastic singular $p$-Laplace equations with heat sources}\label{example lap}
 In this subsection, we investigate  the finite-time extinction of solutions to stochastic 1D singular $p$-Laplace equations with nonlinear sources. Specifically, we consider the following deterministic equation
\begin{eqnarray}\label{p-lap01}
	\left\{
	\begin{aligned}
		&\partial_t{u}= \text{div}(|\nabla u|^{p-2}\nabla u)+u^{2}, \\
		&u|_{\partial\Lambda}=0, u(0)=u_0,
	\end{aligned}
	\right.
\end{eqnarray}
where $1 < p < 2$. This model arises in combustion theory, where $u$ represents the temperature, the term $\operatorname{div}(|\nabla u|^{p-2}\nabla u)$ denotes thermal diffusion, and the nonlinear source $u^{2}$ is physically referred to as a ``hot source'' (cf.~e.g.~\cite{CY07,Nguyen,TY08}).

It is well known that solutions to \eqref{p-lap01} may blow up in finite time under certain initial conditions. Thus, global solutions do not exist in general (see, e.g.,~\cite{LC03,Nguyen}). Taking random noise into account, we study the following stochastic singular $p$-Laplace equation
\begin{eqnarray}\label{Plap02}
	\left\{
	\begin{aligned}
		& d X_t=[\text{div}(|\nabla X_t|^{p-2}\nabla X_t)+X_t^{2}]dt
		+\mathcal{B}(t,X_t)dW_t,\\
		&X_t|_{\partial\Lambda}=0,
	\end{aligned}
	\right.
\end{eqnarray}
where the diffusion coefficient $\mathcal{B}$ satisfies  $(\mathbf{A}^1_{\mathcal{B}})$, $(\mathbf{A}^2_{\mathcal{B}})$,  and $(\mathbf{A}^{2*}_{\mathcal{B}})$  with $g(x)=Cx^{\frac{4p-3}{3p-3}}$ for a large constant $C$.

To analyze \eqref{Plap02}, we employ the Gelfand triple
$$
\mathbb{V}:=W_0^{1,p}(\Lambda;\mathbb{R}) \subset \mathbb{H}:=L^{2}(\Lambda;\mathbb{R})\subset \mathbb{V}^{*}.
$$

\begin{theorem}\label{main result PLap1}
  For any  initial data $x\in L^2$, Eq.~(\ref{Plap02}) has a unique strong solution. Moreover,  we have the  estimate
	\begin{equation*}
	\sup_{t\in[0,T]}\mathbb{E}\|X_t\|_{L^2}^{2-\eta_0}+\mathbb{E}\int_0^T\|X_t\|_{W^{1,p}}^{2-\eta_0}dt<\infty.
	\end{equation*}
	Furthermore, let $\{x_n\}_{n\in\mathbb{N}}$ and $x$ be a sequence with $\|x_n-x\|_{L^2}\to 0$. Let $X_t(x)$ be the unique solution to Eq.~(\ref{Plap02}) with the initial data $x$.
	Then
	\begin{equation*}
		\|X(x_n)-X(x)\|_{\mathbb{C}_T(L^2)} \to 0~~\text{in probability as}~~n\to\infty.
	\end{equation*}
	In particular, the corresponding Markov semigroup $(\mathcal{T}_t)_{t\geq 0}$ is a Feller semigroup in $C_b(L^2)$.
\end{theorem}

\begin{proof}
	Let
	$$\mathcal{A}(t,u):=\text{div}(|\nabla u|^{p-2}\nabla u)+u^{2}.$$
	For any $u,v\in {\mathbb{V}}$, by Young's inequality we have
	\begin{eqnarray*}
		{}_{\mathbb{V}^*}\<u^2-v^2,u-v\>_{\mathbb{V}}
		=\!\!\!\!\!\!\!\!&&  \<u^2-v^2, u-v\>_{L^2}  \\
		\le\!\!\!\!\!\!\!\!&&  C\int_{\Lambda}|u|(u-v)^2dx+C\int_{\Lambda}|v|(u-v)^2dx  \\
		\le\!\!\!\!\!\!\!\!&&  C(\|u\|_{L^{\infty}}+\|v\|_{L^{\infty}}) \|u-v\|_{L^2}^2\\
		\le\!\!\!\!\!\!\!\!&&  C(\|u\|_{W^{1,p}}+\|v\|_{W^{1,p}}) \|u-v\|_{L^2}^2,
	\end{eqnarray*}
	where we have used the Sobolev embedding $W^{1,p}_0(\Lambda;\mathbb{R})\subset L^\infty(\Lambda;\mathbb{R})$ for $p>d=1$.
	
	Then
	\begin{equation*}
		{}_{\mathbb{V}^*}\<\mathcal{A}(t,u)-\mathcal{A}(t,v),u-v\>_{\mathbb{V}}
		\leq C\left(1+ \|u\|_{W^{1,p}}^{p} +\| v\|_{W^{1,p}}^{p} \right)\| u-v\|_{L^2} ^{2},
	\end{equation*}
	which implies the local monotonicity $(\mathbf{A_2})$ holds with $\alpha=p$ and
	$$\rho(u)=\eta(u):=\|u\|_{W^{1,p}}^{p}.$$
	
	By interpolation inequality \eref{GN_inequality} and Young's inequality, for any $u\in {\mathbb{V}}$ there exists constants $C>0$ such that
	\begin{equation}\label{ec01}
		\begin{split}
			{}_{\mathbb{V}^*}\<\mathcal{A}(u),u\>_{\mathbb{V}}=& -\|u\|_{W^{1,p}}^{p}+\| u\|_{L^3}^{3}\\
			\le&  -\|u\|_{W^{1,p}}^{p}+C\|u\|_{W^{1,p}}^{3\theta}\|u\|_{L^2}^{3(1-\theta)}\\
			\le&  -\frac{1}{2}\|u\|_{W^{1,p}}^{p}+C\|u\|_{L^2}^{\frac{8p-6}{3p-3}},
		\end{split}
	\end{equation}
	where $\theta=\frac{p}{9p-6}\in(0,1)$. Thus, we can see that the condition $(\mathbf{A_3})$ holds with $g(x)=Cx^{\frac{4p-3}{3p-3}}$.
	
	For any $u,v\in {\mathbb{V}}$,
	\begin{equation*}
		\begin{split}
			{}_{\mathbb{V}^*}\<\mathcal{A}(t,u),v\>_{\mathbb{V}}=& -\int_\Lambda|\nabla u|^{p-2}\nabla u\cdot\nabla vdx+\int_\Lambda u^2\cdot vdx\\
			\le& \left(\int_\Lambda|\nabla u|^{p}dx\right)^{\frac{p-1}{p}}\left(\int_\Lambda|\nabla v|^{p}dx\right)^{\frac{1}{p}}+C\|v\|_{L^\infty}\left(\int_\Lambda|u|^{2}dx\right)\\
			\le& ~\|u\|_{W^{1,p}}^{p-1}\|v\|_{W^{1,p}}+C\|u\|_{L^2}^2\|v\|_{W^{1,p}}.
		\end{split}
	\end{equation*}
	Therefore, the growth condition $(\mathbf{A_4})$ holds, namely for any $u\in {\mathbb{V}}$,
	$$\|\mathcal{A}(t,u)\|_{\mathbb{V}^*}^{\frac{p}{p-1}}\leq C(\|u\|_{W^{1,p}}^p+\|u\|_{L^2}^{\frac{2p}{p-1}}).$$
	Finally, it is easy to show that $(\mathbf{A_1})$ holds. We complete the proof by applying Theorems \ref{th1} and \ref{th2}.
\end{proof}

Let $\tau_{e}$ be the extinction time  defined by
\begin{equation}\label{extime}
	\tau_{e}:=\inf\big\{t\geq 0:\|X_t\|_{L^2}=0\big\},
\end{equation}
where $(X_t)_{t\geq 0}$ is the solution to \eref{Plap02} with initial value $X_0 = x\in L^2$.  We establish the following finite-time extinction result for \eqref{Plap02}.

\begin{theorem}\label{main result PLap2}
For any initial data $x\in L^2$, we have
	\begin{equation*}
		\|X_t\|_{L^2}=0\quad \text{for all } t\geq \tau_{e}, \ \mathbb{P}\text{-a.s.}
	\end{equation*}
Furthermore, there exists a constant $c_0>0$ such that for any $T>0$,
\begin{equation*}
\mathbb{P}(\tau_{e}>T)
\leq\frac{c_0\|x\|_{L^2}^{2-p}}{T}.
\end{equation*}
Moreover, we also have the moment estimate of the extinction time
\begin{equation*}
\mathbb{E}\tau_{e}\leq c_0 \|x\|_{L^2}^{2-p}.
\end{equation*}
\end{theorem}
\begin{proof}
	It suffices to verify condition  $(\mathbf{A_3^*})$, which follows directly   from  \eref{ec01}.
\end{proof}
\begin{remark}
	Note that solutions of the deterministic singular $p$-Laplace equation \eqref{p-lap01} may blow up in finite time. We show that the introduction of appropriate random noise can induce almost sure finite-time extinction for the stochastic dynamics. This is a novel phenomenon that illustrates the effect of random noise on the long-time behavior of singular $p$-Laplace equations.
\end{remark}

\subsection{Stochastic fast diffusion equations}\label{example pm}
The fast diffusion equation arises in the description of a wide variety of physical phenomena and processes, including fluid flows in porous media, diffusion processes in kinetic gas theory, heat transfer in plasmas, and population dynamics (cf.~\cite{BR12,Gess15}).

In this subsection, we study SFDEs. We assume either $d \in \{1, 2\}$ with $r \in (0, 1)$, or $d \geq 3$ with $r \in \left[\frac{d-2}{d+2}, 1\right)$. Specifically, we consider
\begin{eqnarray}\label{PME}
\left\{\begin{aligned}
	 &dX_t=\Delta (|X_t|^{r-1}X_t)dt+\mathcal{B}(t,X_t)dW_t,\\
	   &X_t|_{\partial\Lambda}=0,
	     \end{aligned}
\right.
\end{eqnarray}
where the diffusion coefficient $\mathcal{B}$ satisfies  $(\mathbf{A}^1_{\mathcal{B}})$ and $(\mathbf{A}^2_{\mathcal{B}})$  with $g(x)=0$.

We employ the Gelfand triple
$$
\mathbb{V}:=L^{r+1}(\Lambda;\mathbb{R}) \subset \mathbb{H}:=H^{-1}(\Lambda;\mathbb{R}) \subset \mathbb{V}^{*},
$$
where $H^{-1}(\Lambda;\mathbb{R})$ denotes the dual space of $W_{0}^{1,2}(\Lambda;\mathbb{R})$, associated with the norm $\|\cdot\|_{-1}$.

Next, we establish the finite-time extinction of solutions to stochastic fast diffusion equations. Let $\tau_{e}$ be the extinction time defined by
\begin{equation*}
	\tau_{e}:=\inf\big\{t\geq 0:\|X_t\|_{-1}=0\big\},
\end{equation*}
where $(X_t)_{t\geq 0}$ is the solution to (\ref{PME}) with initial value $X_0 = x\in H^{-1}$. We present the following result.
\begin{theorem}\label{main result PME3}
For any initial data $x\in H^{-1}$, we have
	\begin{equation*}
		\|X_t\|_{-1}=0\quad \text{for all } t\geq \tau_{e}, \ \mathbb{P}\text{-a.s.}
	\end{equation*}
Furthermore, there exists a constant $c_0>0$ such that for any $T>0$,
\begin{equation*}
\mathbb{P}(\tau_{e}>T)
\leq\frac{c_0\|x\|_{-1}^{2-p}}{T}.
\end{equation*}
Moreover, we also have the moment estimate of the extinction time
\begin{equation*}
\mathbb{E}\tau_{e}\leq c_0 \|x\|_{-1}^{2-p}.
\end{equation*}
\end{theorem}

\begin{proof}
It is well known that the operator $\mathcal{A}(t,u):=\Delta (|u|^{r-1}u)$ satisfies conditions $(\mathbf{A}_1)$--$(\mathbf{A}_5)$ with $g \equiv 0$; we refer to \cite[Example~4.1.11]{LR1} for details. By applying Theorem~\ref{th3}, the assertion follows.
\end{proof}

\begin{remark}
 We note that Theorem~\ref{main result PME3} applies to the finite-time extinction of SFDEs driven by linear multiplicative noise, a setting commonly considered in the existing literature (cf.~\cite{BDR09,BDR09b,BR12,BRR15,Gess15}), i.e.,
\[
X_t=x+\int_0^t\Delta (|X_s|^{r-1}X_s)ds+\sum_{k=1}^{\infty}\int_0^t b_k X_s d\beta_k(s),
\]
where $\sum_{k=1}^{\infty}b_k^2<\infty$ and $\{\beta_k(t)\}_{k\in\mathbb{N}}$ is a sequence of independent one-dimensional standard Wiener processes.

However, to the best of our knowledge, this is the first work to investigate the finite-time extinction of SFDEs driven by general nonlinear multiplicative noise. Furthermore, in contrast to previous studies \cite{BDR09,BDR09b,BR12,BRR15}, we establish finite-time extinction with probability one for arbitrary spatial dimensions (under the standard dimension-dependent range for $r$).
\end{remark}

\begin{remark}
Beyond the aforementioned models, our results apply to all examples in \cite{LR1, RSZ1} under both linear and superlinear multiplicative noise perturbations; we omit the details for brevity.
\end{remark}

\section{Proof of main results}\label{proofM}

This section is devoted to the proofs of Theorems \ref{th1}--\ref{th3}. Specifically, in Subsection \ref{sec2.1}, we construct the Galerkin approximation of SPDE (\ref{eqSPDE}) and derive uniform a priori estimates  by employing suitable Lyapunov functions. In Subsection \ref{sec2.2}, we establish the tightness of the approximating sequence by leveraging a stopping time technique.
Subsection \ref{sec2.3} is concerned with the existence of weak solutions to (\ref{eqSPDE}), which we prove by combining the theory of pseudo-monotone operators, the stochastic compactness method, and Jakubowski's generalization of the Skorokhod representation theorem. In Subsections \ref{sec2.4} and \ref{sec2.5}, we show pathwise uniqueness and   continuous dependence of solutions on the initial data (in probability) along with the Feller property of the associated Markov semigroup.
Finally, in Subsection \ref{sec4.6}, we establish the finite-time extinction of solutions to (\ref{eqSPDE}) by utilizing a distinct Lyapunov function.

\subsection{Energy estimates  }\label{sec2.1}
Let $\{e_1,e_2,\cdots\}\subset {\mathbb{V}}$  be an orthonormal basis (ONB for short) on ${\mathbb{H}}$. Consider the maps
$\mathcal{P}_n:{\mathbb{V}}^{*}\rightarrow {\mathbb{H}}_n,~n\in\mathbb{N},$
given by
$$\mathcal{P}_n x:=\sum\limits_{i=1}^{n}{}_{{\mathbb{V}}^*}\langle x,e_i\rangle_{{\mathbb{V}}
}e_i,~x\in \mathbb{V}
^*,$$
where $\mathbb{H}_n:=\text{span}\{e_1,e_2,\cdots,e_n\}$.

If we restrict $\mathcal{P}_n$ to ${\mathbb{H}}$, denoted by $\mathcal{P}_n|_{{\mathbb{H}}}$, then it is an orthogonal projection onto ${\mathbb{H}}_n$ on ${\mathbb{H}}$.
Denote by $\{g_1,g_2,\cdots\}$ the ONB of $U$. Let
\begin{equation*}
W^{n}_t:=\tilde{\mathcal{P}}_nW_t=\sum\limits_{i=1}^{n}\langle W_t,g_i\rangle_{U}g_i,~n\in\mathbb{N},
\end{equation*}
where $\tilde{\mathcal{P}}_n$ is an orthonormal projection onto $U^n:=\text{span}\{g_1,g_2,\cdots,g_n\}$ on $U$.

For any $n\in\mathbb{N}$, we consider the following stochastic equation on ${\mathbb{H}}_n$,
\begin{eqnarray}\label{eqf}
dX^{(n)}_t=\mathcal{P}_n\mathcal{A}(t,X^{(n)}_t)dt
+\mathcal{P}_n\mathcal{B}(t,X^{(n)}_t)dW^{n}_t,
\end{eqnarray}
with initial value $x^{(n)}:=\mathcal{P}_n x$.
Under $(\mathbf{A_1})$, $(\mathbf{A_3})$, $(\mathbf{A_4})$ and the assumption (\ref{conb2}), it is clear that there exists a weak solution to (\ref{eqf}) in the sense of Definition \ref{dew} up to its life time. Furthermore, under the assumption (\ref{conb1}), the solution is non-explosive (see Lemma \ref{lemma4.10}), namely, there is a global weak solution to (\ref{eqf}).

\vspace{2mm}
We have the following a priori estimates based on choosing a suitable Lyapunov function.
\begin{lemma}\label{lemma4.10}
Suppose that the assumptions in Theorem \ref{th1} hold. Then there is a constant $C_T>0$ such that for any $t\in[0,T]$,
	\begin{equation}\label{e000}
		\sup_{n\in\mathbb{N}}\bigg\{\sup_{t\in[0,T]}\mathbb{E}\|X^{(n)}_t\|_{\mathbb{H}}^{2-\eta_0}+\mathbb{E}\int_0^T\|X^{(n)}_t\|_{\mathbb{V}}^{\alpha-\eta_0}dt\bigg\}
		\leq C_T(1+\|x\|_{\mathbb{H}}^{2-\eta_0}),
	\end{equation}
where the constant $\eta_0$ is same as in 	$(\mathbf{A_5})$.
\end{lemma}

\begin{proof}
By It\^{o}'s formula for $\|\cdot\|_{\mathbb{H}}^2$, for any $t\in[0,T]$,
\begin{eqnarray}\label{e4}
	\|X^{(n)}_t\|_{\mathbb{H}}^2
	=\!\!\!\!\!\!\!\!&&\|x^{(n)}\|_{\mathbb{H}}^2+\int_0^t\Big(2{}_{\mathbb{V}^*}\langle \mathcal{P}_n\mathcal{A}(s,X^{(n)}_s),X^{(n)}_s\rangle_{\mathbb{V}}
	\nonumber \\
	\!\!\!\!\!\!\!\!&&
	+\|\mathcal{P}_n\mathcal{B}(s,X^{(n)}_s)\tilde{\mathcal{P}}_n\|_{L_2(U,\mathbb{H})}^2\Big)ds
	\nonumber \\
	\!\!\!\!\!\!\!\!&&+2\int_0^t\langle X^{(n)}_s,\mathcal{P}_n\mathcal{B}(s,X^{(n)}_s)dW^{n}_s\rangle_{\mathbb{H}}.
	%\nonumber \\
	%\leq\!\!\!\!\!\!\!\!&&C_T(1+\|x\|_{\mathbb{H}}^2)-\delta\int_0^t\|X^{(n)}_s\|_{\mathbb{V}}^{\alpha}ds
	%\nonumber \\
	%\!\!\!\!\!\!\!\!&&
	%+\int_0^t\Big(g(\|X^{(n)}_s\|_{\mathbb{H}}^2)+\|\mathcal{B}(s,X^{(n)}_s)\|_{L_2(U,\mathbb{H})}^2\Big)ds
	%\nonumber \\
	%\!\!\!\!\!\!\!\!&&+2\int_0^t\langle X^{(n)}_s,\mathcal{B}(s,X^{(n)}_s)dW^{(n)}_s\rangle_{\mathbb{H}}.
\end{eqnarray}
%For any $n\in\mathbb{N}$, we set a stopping time
%$$\tau_R^{(n)}:=\inf\Big\{t\in[0,T]:\|X^{(n)}_t\|_{\mathbb{H}}> R\Big\}\wedge T,~R>0,$$
%where we take $\inf{\emptyset}=\infty$.  It's easy to see that
%$$\lim_{R\to\infty}\tau_R^{(n)}=T,~\mathbf{P}_n\text{-a.s.},~n\in\mathbb{N}.$$
Then using It\^{o}'s formula for the Lyapunov function $V(r):=(1+r)^{\gamma}$ with $\gamma\in(0,1)$, by $(\mathbf{A_3})$ we deduce that
\begin{eqnarray}\label{e1}
	\!\!\!\!\!\!\!\!&&(1+\|X^{(n)}_t\|_{\mathbb{H}}^2)^{\gamma}
	\nonumber \\
	\leq\!\!\!\!\!\!\!\!&&(1+\|x^{(n)}\|_{\mathbb{H}}^2)^{\gamma}-\delta\gamma\int_0^t \frac{\|X^{(n)}_s\|_{\mathbb{V}}^{\alpha}}{(1+\|X^{(n)}_s\|_{\mathbb{H}}^2)^{1-\gamma}}ds
	\nonumber \\
	\!\!\!\!\!\!\!\!&&
	+\gamma\int_0^t \frac{g(\|X^{(n)}_s\|_{\mathbb{H}}^2)+\|\mathcal{B}(s,X^{(n)}_s)\|_{L_2(U,\mathbb{H})}^2+C}{(1+\|X^{(n)}_s\|_{\mathbb{H}}^2)^{1-\gamma}}ds
	\nonumber \\
	\!\!\!\!\!\!\!\!&&-2\gamma(1-\gamma)\int_0^t \frac{\|\mathcal{B}(s,X^{(n)}_s)^*X^{(n)}_s\|_{U}^2}{(1+\|X^{(n)}_s\|_{\mathbb{H}}^2)^{2-\gamma}}ds
	+\mathcal{M}_t
	\nonumber \\
	\leq\!\!\!\!\!\!\!\!&&C_T(1+\|x\|_{\mathbb{H}}^{2\gamma})-\delta\gamma\int_0^t \frac{\|X^{(n)}_s\|_{\mathbb{V}}^{\alpha}}{(1+\|X^{(n)}_s\|_{\mathbb{H}}^2)^{1-\gamma}}ds
	\nonumber \\
	\!\!\!\!\!\!\!\!&&
	+\gamma\int_0^t \Bigg\{\frac{\big(g(\|X^{(n)}_s\|_{\mathbb{H}}^2)+\|\mathcal{B}(s,X^{(n)}_s)\|_{L_2(U,\mathbb{H})}^2\big)(1+\|X^{(n)}_s\|_{\mathbb{H}}^2)}{(1+\|X^{(n)}_s\|_{\mathbb{H}}^2)^{2-\gamma}}
	\nonumber \\
	\!\!\!\!\!\!\!\!&&
	-\frac{2(1-\gamma)\|\mathcal{B}(s,X^{(n)}_s)^*X^{(n)}_s\|_{U}^2}{(1+\|X^{(n)}_s\|_{\mathbb{H}}^2)^{2-\gamma}}\Bigg\}ds
+\mathcal{M}_t
	\nonumber \\
	=:\!\!\!\!\!\!\!\!&&
C_T(1+\|x\|_{\mathbb{H}}^{2\gamma})-\delta\gamma\int_0^t \frac{\|X^{(n)}_s\|_{\mathbb{V}}^{\alpha}}{(1+\|X^{(n)}_s\|_{\mathbb{H}}^2)^{1-\gamma}}ds+(\text{I})+\mathcal{M}_t,
\end{eqnarray}
where we denote
$$\mathcal{M}_t:=2\int_0^t \frac{\langle X^{(n)}_s,\mathcal{B}(s,X^{(n)}_s)dW^{n}_s\rangle_{\mathbb{H}}}{1+\|X^{(n)}_s\|_{\mathbb{H}}^2}.$$
In view of the assumption (\ref{conb1}) in $(\mathbf{A_5})$, we can choose  $\gamma=1-\frac{\eta_0}{2}$ such that
\begin{equation}\label{e10}
(\text{I})\leq C\int_0^t \frac{ (1+\|X^{(n)}_s\|_{\mathbb{H}}^2)^2}{(1+\|X^{(n)}_s\|_{\mathbb{H}}^2)^{2-\gamma}}ds= C\int_0^t  (1+\|X^{(n)}_s\|_{\mathbb{H}}^2)^{\gamma}ds.
\end{equation}
For any $R>0$, we define the stopping times
$$\tau_{1,M}^{(n)}:=\inf\bigg\{t\geq 0:\|X^{(n)}_t\|_{\mathbb{H}}\geq M\bigg\},$$
$$\tau_{2,M}^{(n)}:=\inf\bigg\{t\geq 0:\int_0^t \frac{\|X^{(n)}_s\|_{\mathbb{V}}^{\alpha}}{(1+\|X^{(n)}_s\|_{\mathbb{H}}^2)^{\frac{\eta_0}{2}}}ds\geq M\bigg\},$$
and
$$\tau_M^{(n)}:=\tau_{1,M}^{(n)}\wedge \tau_{2,M}^{(n)},$$
with the convention that $\inf\emptyset =\infty$.

Thus, it follows from (\ref{e1}), (\ref{e10}), and Gronwall's lemma that for any $t\in[0,T]$, we have
\begin{equation}\label{e12}
\mathbb{E}(1+\|X^{(n)}_{t\wedge \tau_M^{(n)}}\|_{\mathbb{H}}^2)^{1-\frac{\eta_0}{2}}+\delta\gamma\mathbb{E}\int_0^{t\wedge \tau_M^{(n)}} \frac{\|X^{(n)}_s\|_{\mathbb{V}}^{\alpha}}{(1+\|X^{(n)}_s\|_{\mathbb{H}}^2)^{\frac{\eta_0}{2}}}ds
	\leq C_T(1+\|x\|_{\mathbb{H}}^{2-\eta_0}).
\end{equation}
Then by letting $R\to\infty$ and using Fatou's lemma we derive
\begin{equation}\label{e2}
\mathbb{E}\|X^{(n)}_{t}\|_{\mathbb{H}}^{2-\eta_0}+\delta\gamma\mathbb{E}\int_0^{t} \frac{\|X^{(n)}_s\|_{\mathbb{V}}^{\alpha}}{(1+\|X^{(n)}_s\|_{\mathbb{H}}^2)^{\frac{\eta_0}{2}}}ds
	\leq C_T(1+\|x\|_{\mathbb{H}}^{2-\eta_0}).
\end{equation}

On the other hand,  by Young's inequality we observe that
\begin{eqnarray}\label{ei0}
\mathbb{E}\int_0^T \|X^{(n)}_t\|_{\mathbb{V}}^{\alpha-\eta_0}dt	
=\!\!\!\!\!\!\!\!&&\mathbb{E}\int_0^T \frac{\|X^{(n)}_t\|_{\mathbb{V}}^{\alpha-\eta_0}(1+\|X^{(n)}_t\|_{\mathbb{H}}^2)^{\frac{\eta_0}{2}}}{(1+\|X^{(n)}_t\|_{\mathbb{H}}^2)^{\frac{\eta_0}{2}}}dt
\nonumber \\
\leq\!\!\!\!\!\!\!\!&&\mathbb{E}\int_0^T \frac{\|X^{(n)}_t\|_{\mathbb{V}}^{\alpha-\eta_0}(1+\|X^{(n)}_t\|_{\mathbb{V}}^{\eta_0})}{(1+\|X^{(n)}_t\|_{\mathbb{H}}^2)^{\frac{\eta_0}{2}}}dt
\nonumber \\
\leq\!\!\!\!\!\!\!\!&&C\mathbb{E}\int_0^T \frac{1+\|X^{(n)}_t\|_{\mathbb{V}}^{\alpha}}{(1+\|X^{(n)}_t\|_{\mathbb{H}}^2)^{\frac{\eta_0}{2}}}dt
\nonumber \\
\leq\!\!\!\!\!\!\!\!&&C_T(1+\|x\|_{\mathbb{H}}^{2-\eta_0}).
\end{eqnarray}
Combining (\ref{e2}) with  (\ref{ei0}),
	we complete the proof.
\end{proof}

We also derive the energy estimates in probability, which are useful in proving the tightness of approximating sequences.
\begin{lemma}\label{lem3.0}
Suppose that the assumptions in Theorem \ref{th1} hold. For any $\varepsilon>0$, there exists  $\mathcal{K}>0$ such that for any $p\geq 2$,
\begin{equation}\label{apri}
\sup_{n\in\mathbb{N}}\mathbb{P}\Bigg(\sup_{t\in[0,T]}\|X^{(n)}_t\|_{\mathbb{H}}^p+\int_0^T\|X^{(n)}_t\|_{\mathbb{V}}^{\alpha}dt\geq \mathcal{K}\Bigg)
\leq \varepsilon.
\end{equation}
\end{lemma}

\begin{proof}
According to the definition of $\tau_M^{(n)}$, we know that
\begin{equation}\label{e11}
\mathbb{P}\bigg(\Big\{\|X^{(n)}_{\tau_M^{(n)}}\|_{\mathbb{H}}\geq M\Big\}\cup\Big\{\int_0^{\tau_M^{(n)}}\frac{\|X^{(n)}_t\|_{\mathbb{V}}^{\alpha}}{(1+\|X^{(n)}_t\|_{\mathbb{H}}^2)^{\frac{\eta_0}{2}}}dt\geq M\Big\}\bigg)
=1.
\end{equation}
Due to (\ref{e12})  and (\ref{e11}), we have
\begin{eqnarray}\label{e13}
\!\!\!\!\!\!\!\!&&\mathbb{P}\Big(\sup_{t\in[0,T]}\|X^{(n)}_t\|_{\mathbb{H}}\geq M\Big)+\mathbb{P}\bigg(\int_0^{T}\frac{\|X^{(n)}_t\|_{\mathbb{V}}^{\alpha}}{(1+\|X^{(n)}_t\|_{\mathbb{H}}^2)^{\frac{\eta_0}{2}}}dt\geq M\bigg)
\nonumber \\
=\!\!\!\!\!\!\!\!&&\mathbb{P}(\tau_{1,M}^{(n)}\leq T)+\mathbb{P}(\tau_{2,M}^{(n)}\leq T)
\nonumber \\
\leq\!\!\!\!\!\!\!\!&&2\mathbb{P}\bigg(\{\tau_{M}^{(n)}\leq T\}\cap\Big(\Big\{\|X^{(n)}_{\tau_M^{(n)}}\|_{\mathbb{H}}\geq M\Big\}\cup\Big\{\int_0^{\tau_M^{(n)}}\frac{\|X^{(n)}_t\|_{\mathbb{V}}^{\alpha}}{(1+\|X^{(n)}_t\|_{\mathbb{H}}^2)^{\frac{\eta_0}{2}}}dt\geq M\Big\}       \Big)\bigg)
\nonumber \\
\leq\!\!\!\!\!\!\!\!&&2\mathbb{P}\Big(\|X^{(n)}_{T\wedge\tau_M^{(n)}}\|_{\mathbb{H}}\geq M\Big)+2\mathbb{P}\bigg(\int_0^{T\wedge\tau_M^{(n)}}\frac{\|X^{(n)}_t\|_{\mathbb{V}}^{\alpha}}{(1+\|X^{(n)}_t\|_{\mathbb{H}}^2)^{\frac{\eta_0}{2}}}dt\geq M\bigg)
\nonumber \\
\leq\!\!\!\!\!\!\!\!&&\frac{2\mathbb{E}\|X^{(n)}_{T\wedge\tau_M^{(n)}}\|_{\mathbb{H}}^{2-\eta_0}}{M^{2-\eta_0}}+\frac{2\mathbb{E}\int_0^{T\wedge\tau_M^{(n)}}\frac{\|X^{(n)}_t\|_{\mathbb{V}}^{\alpha}}{(1+\|X^{(n)}_t\|_{\mathbb{H}}^2)^{\frac{\eta_0}{2}}}dt}{M}
\nonumber \\
\leq\!\!\!\!\!\!\!\!&&C_T(1+\|x\|_{\mathbb{H}}^{2-\eta_0})\Big(\frac{1}{M^{2-\eta_0}}+\frac{1}{M}\Big).
\end{eqnarray}
Now we estimate  $\int_0^T \|X^{(n)}_t\|_{\mathbb{V}}^{\alpha}dt$. Using (\ref{e13}), we derive
\begin{eqnarray}\label{e6}
\!\!\!\!\!\!\!\!&&\mathbb{P}\Bigg(\int_0^T \|X^{(n)}_t\|_{\mathbb{V}}^{\alpha}dt\geq R\Bigg)
\nonumber \\
~~~\leq\!\!\!\!\!\!\!\!&&\mathbb{P}\Bigg(\int_0^T \|X^{(n)}_t\|_{\mathbb{V}}^{\alpha}dt\geq R,\tau_{1,M}^{(n)}\geq T\Bigg)+
\mathbb{P}(\tau_{1,M}^{(n)}< T)
\nonumber \\
\leq\!\!\!\!\!\!\!\!&&\mathbb{P}\Bigg(\int_0^{T\wedge\tau_{1,M}^{(n)}} \frac{\|X^{(n)}_t\|_{\mathbb{V}}^{\alpha}}{(1+\|X^{(n)}_t\|_{\mathbb{H}}^2)^{\frac{\eta_0}{2}}}\cdot\big(1+\|X^{(n)}_t\|_{\mathbb{H}}^2\big)^{\frac{\eta_0}{2}} dt\geq R\Bigg)+
\mathbb{P}(\tau_{1,M}^{(n)}< T)
\nonumber \\
\leq\!\!\!\!\!\!\!\!&&\mathbb{P}\Bigg(C_M\int_0^{T} \frac{\|X^{(n)}_t\|_{\mathbb{V}}^{\alpha}}{(1+\|X^{(n)}_t\|_{\mathbb{H}}^2)^{\frac{\eta_0}{2}}} dt\geq R\Bigg)+
\mathbb{P}\Big(\sup_{t\in[0,T]}\|X^{(n)}_t\|_{\mathbb{H}}\geq M\Big).
\nonumber \\
\leq\!\!\!\!\!\!\!\!&&C_{M,T}(1+\|x\|_{\mathbb{H}}^{2-\eta_0})\Big(\frac{1}{R^{2-\eta_0}}+\frac{1}{R}\Big)+C_T(1+\|x\|_{\mathbb{H}}^{2-\eta_0})\Big(\frac{1}{M^{2-\eta_0}}+\frac{1}{M}\Big).
\end{eqnarray}
Consequently, combining (\ref{e13}) with (\ref{e6}), we first take $R\to \infty$ and then $M\to\infty$ to obtain
\begin{equation*}
\lim_{R\to\infty}\mathbb{P}\Bigg(\sup_{t\in[0,T]}\|X^{(n)}_t\|_{\mathbb{H}}^p+\int_0^T \|X^{(n)}_t\|_{\mathbb{V}}^{\alpha}dt\geq R\Bigg)=0.
\end{equation*}
 We complete the proof.
\end{proof}

\vspace{2mm}
Based on Lemma \ref{lem3.0}, we have the following bounds.
\begin{lemma}\label{lem1}
Suppose that the assumptions in Theorem \ref{th1} hold. For any $\varepsilon>0$ there exists  $\mathcal{K}>0$ such that
\begin{eqnarray}
\!\!\!\!\!\!\!\!&&\sup_{n\in\mathbb{N}}\mathbb{P}\Bigg(\int_0^T\|\mathcal{A}(t,X^{(n)}_t)\|_{\mathbb{V}^*}^{\frac{\alpha}{\alpha-1}}dt\geq \mathcal{K}\Bigg)
\leq \varepsilon.\label{apriA}
% \\
%\!\!\!\!\!\!\!\!&&\sup_{n\in\mathbb{N}}\mathbb{P}\Bigg(\int_0^T\|\mathcal{B}(t,X^{(n)}_t)\|_{L_2(U,\mathbb{H})}^{2}dt\geq \mathcal{K}\Bigg)
%\leq \varepsilon.\label{apriB}
\end{eqnarray}
\end{lemma}

\begin{proof}
%We only prove (\ref{apriA}) since the proof of (\ref{apriB}) is completely similar.
Set the following stopping time
\begin{equation}\label{stop1}
\tilde{\tau}_M^{(n)}:=\inf\Bigg\{t\in[0,T]:\|X^{(n)}_t\|_{\mathbb{H}}+\int_0^t \|X^{(n)}_s\|_{\mathbb{V}}^{\alpha}ds\geq M\Bigg\}\wedge T,~~M>0,
\end{equation}
with the convention $\inf\emptyset=\infty$.

According to the assumption (\ref{conA3}), we deduce that
\begin{eqnarray*}
\!\!\!\!\!\!\!\!&&\mathbb{P}\Bigg(\int_0^T\|\mathcal{A}(t,X^{(n)}_t)\|_{\mathbb{V}^*}^{\frac{\alpha}{\alpha-1}}dt\geq R\Bigg)
\nonumber \\
\leq\!\!\!\!\!\!\!\!&&\mathbb{P}\Bigg(\int_0^T\|\mathcal{A}(t,X^{(n)}_t)\|_{\mathbb{V}^*}^{\frac{\alpha}{\alpha-1}}dt\geq R,\tilde{\tau}_M^{(n)}\geq T\Bigg)+
\mathbb{P}(\tilde{\tau}_M^{(n)}< T)
\nonumber \\
\leq\!\!\!\!\!\!\!\!&&\mathbb{P}\Bigg(C\int_0^{T\wedge\tilde{\tau}_M^{(n)}} \big(1+\|X^{(n)}_t\|_{\mathbb{V}}^{\alpha}\big)\big(1+\|X^{(n)}_t\|_{\mathbb{H}}^{\beta}\big) dt\geq R\Bigg)+
\mathbb{P}(\tilde{\tau}_M^{(n)}< T)
\nonumber \\
\leq\!\!\!\!\!\!\!\!&&\frac{C_M}{R}+
\mathbb{P}\Bigg(\sup_{t\in[0,T]}\|X^{(n)}_t\|_{\mathbb{H}}+\int_0^T \|X^{(n)}_s\|_{\mathbb{V}}^{\alpha}ds\geq M\Bigg).
\end{eqnarray*}
By Lemma \ref{lem3.0}, letting $R\to \infty$ and then $M\to\infty$, we complete the assertion.
\end{proof}

\subsection{Tightness of approximating solutions}\label{sec2.2}
Set
\begin{eqnarray}
\!\!\!\!\!\!\!\!&&\mathcal{Z}^1_T:=\mathbb{C}_T(\mathbb{V}^*)\cap L^{\alpha}([0,T];\mathbb{H})\cap L^{\alpha}_w([0,T];\mathbb{V}),\label{spz}
 \\
\!\!\!\!\!\!\!\!&&\mathcal{Z}^2_T:=L^{\frac{\alpha}{\alpha-1}}_w([0,T];\mathbb{V}^*),\nonumber
%\nonumber \\
%\!\!\!\!\!\!\!\!&&\mathcal{Z}^2_T:=L^2_w([0,T];L_2(U,\mathbb{H})),
\end{eqnarray}
where $L^{\alpha}_w([0,T];\mathbb{V})$, $L^{\frac{\alpha}{\alpha-1}}_w([0,T];\mathbb{V}^*)$ denote spaces $L^{\alpha}([0,T];{\mathbb{V}})$, $L^{\frac{\alpha}{\alpha-1}}([0,T];\mathbb{V}^*)$ endowed with the weak topology, respectively.
In this subsection, we will show that $\{X^{(n)}\}_{n\in\mathbb{N}}$, $\{\mathcal{A}(\cdot,X^{(n)}_{\cdot})\}_{n\in\mathbb{N}}$
 are tight in $\mathcal{Z}^1_T$, $\mathcal{Z}^2_T$, respectively.

\begin{remark}\label{k2t}
Here the intersection space $\mathcal{Z}^1_T$ takes the following intersection topology denoted by $\tau_{\mathcal{Z}^1_T}$: the class of
open sets of $\mathcal{Z}^1_T$  are generated by the sets of the form $\mathscr{O}_1\cap\mathscr{O}_2\cap\mathscr{O}_3$, where $\mathscr{O}_1$, $\mathscr{O}_2$ and $\mathscr{O}_3$ are the open sets in $\mathbb{C}_T(\mathbb{V}^*)$,  $L^{\alpha}([0,T];\mathbb{H})$ and  $L^{\alpha}_w([0,T];\mathbb{V})$, respectively. Let $\mathscr{B}(\tau_{\mathcal{Z}^1_T})$ be the corresponding Borel $\sigma$-algebra.
\end{remark}

In the following, we formulate the compactness criterion for the space $\mathcal{Z}^1_T$.
\begin{lemma}\label{lemc}
Let  $\mathcal{K}$ be a subset of $\mathcal{Z}^1_T$ such that the following conditions hold
\begin{enumerate}[(i)]

%\item  $$\sup_{X\in\mathcal{K}}\sup_{ t\in[0, T]}\|X_t\|_{\mathbb{H}}<\infty,$$

\item $$\sup_{X\in\mathcal{K}}\int_0^T\|X_t\|_{\mathbb{V}}^\alpha dt<\infty,$$

\item $$\lim_{\delta\to 0}\sup_{X\in\mathcal{K}}\sup_{s,t\in[0,T],|t-s|\leq \delta}\|X_t-X_s\|_{\mathbb{V}^*}=0.$$
\end{enumerate}
Then $\mathcal{K}$ is relatively compact in $\mathcal{Z}^1_T$.
\end{lemma}
\begin{proof}
  Without loss of generality, we suppose that $\mathcal{K}$ is $\tau_{\mathcal{Z}^1_T}$-closed.
  It is known that  the weak topology in  $L^{\alpha}_w([0,T];\mathbb{V})$ restricted in $\mathcal{K}$ is metrizable. Therefore, the compactness of a subset of  $\mathcal{Z}^1_T$ is equivalent to the sequential compactness.

Now, let  $\{X^{(n)}\}_{n\in\mathbb{N}}$ denote a sequence in $\mathcal{K}$. It is sufficient to prove that there exists an element $X\in\mathcal{K}$ such that along a subsequence still denoted by
$\{X^{(n)}\}_{n\in\mathbb{N}}$, we have
\begin{equation}\label{conx0}
X^{(n)}\to X~\text{in}~\mathcal{Z}^1_T~\text{as}~n\to\infty.
\end{equation}
First, due to the Banach-Alaoglu theorem and the condition (i), it follows that
\begin{equation}\label{conx6}
\mathcal{K}~\text{is compact in}~L^{\alpha}_w([0,T];\mathbb{V}).
\end{equation}

On the other hand, let us consider the measurable set
$$\mathcal{N}:=\big\{t\in[0,T]:\lim_{n\to\infty}\|X^{(n)}_t\|_{\mathbb{V}}=\infty \big\}.$$
We claim that it is a Lebesgue null set.  Indeed, we note that otherwise, by Fatou's lemma
$$\liminf_{n\to\infty}\int_{0}^{T}\|X^{(n)}_t\|_{\mathbb{V}}^{\alpha}dt\geq \int_{\mathcal{N}}\liminf_{n\to\infty}\|X^{(n)}_t\|_{\mathbb{V}}^{\alpha}dt=\infty,$$
which contradicts condition (i).

Therefore,  there exist an $dt$-null set $\mathcal{N}$ and a subsequence  still denoted by
$\{X^{(n)}\}_{n\in\mathbb{N}}$ such that for any $t\in [0,T]\backslash \mathcal{N}$,
$$\{\|X^{(n)}_t\|_{\mathbb{V}}\}_{n\in\mathbb{N}}~~\text{is bounded}.$$
By the assumption of Theorem \ref{th1}, the embedding $\mathbb{V}\subset \mathbb{H}$ is compact, we infer that $\mathbb{V}\subset\mathbb{V}^*$ is compact as well. Thus, the sequence $\{X^{(n)}_t\}_{n\in\mathbb{N}}$ contains a subsequence that is convergent in $\mathbb{V}^*$.

Denote $\{t_k\}_{k\in\mathbb{N}}\subset ([0,T]\backslash \mathcal{N})\cap \mathbb{Q}$. Applying the diagonal method, we can choose a subsequence still denoted by
$\{X^{(n)}\}_{n\in\mathbb{N}}$ such that
\begin{equation}\label{conx1}
\{X^{(n)}_{t_k}\}_{n\in\mathbb{N}}~\text{is convergent in}~\mathbb{V}^*~\text{for all}~ k\in\mathbb{N}.
\end{equation}
Next, we prove that the sequence $\{X^{(n)}\}_{n\in\mathbb{N}}$ is a Cauchy net in $\mathbb{C}_T(\mathbb{V}^*)$. For any $\varepsilon>0$, by the condition (ii) there exists $\delta>0$ such that
\begin{equation}\label{conx2}
\sup_{X\in\mathcal{K}}\sup_{s,t\in[0,T],|t-s|\leq \delta}\|X_t-X_s\|_{\mathbb{V}^*}<\frac{\varepsilon}{3}.
\end{equation}
Fix $t\in[0,T]$. We can find an $t_k\in ([0,T]\backslash \mathcal{N})\cap \mathbb{Q}$  such that  $|t_k-t|\leq \delta$. Combining (\ref{conx1})-(\ref{conx2}), for large enough $n,m\in\mathbb{N}$ we deduce that
$$\|X^{(n)}_t-X^{(m)}_t\|_{\mathbb{V}^*}\leq \|X^{(n)}_t-X^{(n)}_{t_k}\|_{\mathbb{V}^*}+\|X^{(n)}_{t_k}-X^{(m)}_{t_k}\|_{\mathbb{V}^*}+\|X^{(m)}_{t_k}-X^{(m)}_{t}\|_{\mathbb{V}^*}\leq \varepsilon.$$
Since $t\in[0,T]$ is arbitrary, we can obtain that
$$\sup_{t\in[0,T]}\|X^{(n)}_t-X^{(m)}_t\|_{\mathbb{V}^*}\leq \varepsilon,$$
which yields that
\begin{equation}\label{conx7}
\{X^{(n)}\}_{n\in\mathbb{N}}~\text{is a Cauchy net in}~\mathbb{C}_T(\mathbb{V}^*).
\end{equation}
Collecting (\ref{conx6}) and (\ref{conx7}), there exists a subsequence still denoted by
$\{X^{(n)}\}_{n\in\mathbb{N}}$ and $X\in \mathbb{C}_T(\mathbb{V}^*)\cap L^{\alpha}([0,T];\mathbb{V})$ such that
\begin{equation}\label{conx3}
X^{(n)}\to X~\text{in}~\mathbb{C}_T(\mathbb{V}^*)\cap L_w^{\alpha}([0,T];\mathbb{V})~\text{as}~n\to\infty.
\end{equation}

Hence, once we can prove
\begin{equation}\label{conx4}
X^{(n)}\to X~\text{in}~L^{\alpha}([0,T];\mathbb{H})~\text{as}~n\to\infty,
\end{equation}
then (\ref{conx0}) follows.
Since the embedding $\mathbb{V}\subset \mathbb{H}$ is compact, by the Lions lemma \cite{Lion} for any $\varepsilon>0$, there exists a constant $C_{\varepsilon}>0 $ such that for almost all $t\in[0,T]$,
$$\|X^{(n)}_t-X_t\|^{\alpha}_{ \mathbb{H}}\leq \varepsilon\|X^{(n)}_t-X_t\|^{\alpha}_{ \mathbb{V}}+C_{\varepsilon}\|X^{(n)}_t-X_t\|^{\alpha}_{ \mathbb{V}^*},~n\in\mathbb{N},$$
which also yields that
\begin{eqnarray}\label{conx5}
 \!\!\!\!\!\!\!\!&&\|X^{(n)}-X\|_{L^{\alpha}([0,T];\mathbb{H})}^{\alpha}
\nonumber \\
\leq   \!\!\!\!\!\!\!\!&& \varepsilon\|X^{(n)}-X\|^{\alpha}_{L^{\alpha}([0,T];\mathbb{V})}+C_{\varepsilon}\|X^{(n)}-X\|^{\alpha}_{L^{\alpha}([0,T];\mathbb{V}^*)},~n\in\mathbb{N}.
\end{eqnarray}
Taking the upper limit as $n\to\infty$ in (\ref{conx5}) and using the following bounds
$$\|X^{(n)}-X\|^{\alpha}_{L^{\alpha}([0,T];\mathbb{V})}\leq C(\|X^{(n)}\|^{\alpha}_{L^{\alpha}([0,T];\mathbb{V})}+\|X\|^{\alpha}_{L^{\alpha}([0,T];\mathbb{V})})\leq C,$$
we can conclude that
$$\lim_{n\to\infty}\|X^{(n)}-X\|^{\alpha}_{L^{\alpha}([0,T];\mathbb{H})}=0.$$
We complete the proof.
\end{proof}

\vspace{1mm}
We recall the  Aldous condition in the space $\mathbb{V}^*$.
\begin{definition}\label{aldo}
A sequence $\{X^{(n)}\}_{n\in\mathbb{N}}$ is said to satisfy the Aldous condition in $\mathbb{V}^*$ iff for any $\varepsilon,\eta>0$, there exists $\delta>0$ such that for every stopping time sequence $(\zeta_n)_{n\in\mathbb{N}}$ with $\zeta_n\leq T$ one has
\begin{equation*}
\sup_{n\in\mathbb{N}}\sup_{0\leq \Delta\leq \delta}\mathbb{P}(\|X^{(n)}_{\zeta_n+\Delta}-X^{(n)}_{\zeta_n}\|_{\mathbb{V}^*}\geq \eta)\leq \varepsilon.
\end{equation*}
\end{definition}

The following lemma presents a tightness criterion for the laws of sequence $\{X^{(n)}\}_{n\in\mathbb{N}}$ on  $\mathcal{Z}^1_T$.
\begin{lemma}\label{lemt}
Let $\{X^{(n)}\}_{n\in\mathbb{N}}$ be a sequence of continuous $\{\mathscr{F}_t\}$-adapted $\mathbb{V}^*$-valued processes such that
\begin{enumerate}[(i)]

%\item  $$\lim_{R\rightarrow\infty}\sup_{n\in\mathbb{N}}\mathbb{P}\Big(\sup_{ t\in[0, T]}\|X^{(n)}_t\|_{\mathbb{H}}>R\Big)=0,$$

\item $$\lim_{R\rightarrow\infty}\sup_{n\in\mathbb{N}}\mathbb{P}\Big(\int_0^T\|X^{(n)}_t\|_{\mathbb{V}}^\alpha dt>R\Big)=0,$$

\item $\{X^{(n)}\}_{n\in\mathbb{N}}$ satisfies the  Aldous condition in $\mathbb{V}^*$.
\end{enumerate}
Let $\mu_n$ be the law of $X^{(n)}$ on the Borel $\sigma$-algebra $\mathscr{B}(\tau_{\mathcal{Z}_T^1})$. Then for every $\varepsilon>0$ there exists a compact subset $\mathcal{K}_{\varepsilon}$ of $\mathcal{Z}_T^1$ such that
$$\sup_{n\in\mathbb{N}}\mu_n(\mathcal{K}_{\varepsilon}^c)\leq \varepsilon.$$
\end{lemma}

\begin{proof}
%First, in view of (i), for any $\varepsilon>0$ there exists $R_1>0$ such that
%$$\sup_{n\in\mathbb{N}}\mathbb{P}\Big(\sup_{ t\in[0, T]}\|X^{(n)}_t\|_{\mathbb{H}}>R_1\Big)\leq \frac{\varepsilon}{3}.$$
%We denote
%$$\mathcal{K}_{1}:=\Big\{X^{(n)}\in\mathcal{Z}_T^1:\sup_{ t\in[0, T]}\|X^{(n)}_t\|_{\mathbb{H}}\leq R_1\Big\}.$$
In view of (i), for any $\varepsilon>0$ there exists $R>0$ such that
$$\sup_{n\in\mathbb{N}}\mathbb{P}\Big(\int_0^T\|X^{(n)}_t\|_{\mathbb{V}}^{\alpha}dt>R\Big)\leq \frac{\varepsilon}{2}.$$
We denote
$$\mathcal{K}:=\Big\{X^{(n)}\in\mathcal{Z}_T^1:\int_0^T\|X^{(n)}_t\|_{\mathbb{V}}^{\alpha}dt\leq R\Big\}.$$
By Lemmas 3.6 and 3.8 in  \cite{BM13}, in view of (ii) there exists a subset $\mathcal{A}_{\varepsilon}\subset \mathbb{C}_T(\mathbb{V}^*)$ such that
$$ \mu_n(\mathcal{A}_{\varepsilon}^c)\leq \frac{\varepsilon}{2}$$
and
$$\lim_{\delta\to 0}\sup_{X\in\mathcal{A}_\varepsilon}\sup_{s,t\in[0,T],|t-s|\leq \delta}\|X_t-X_s\|_{\mathbb{V}^*}=0.$$
Finally, we denote by $\mathcal{K}_{\varepsilon}$ the closure of the set $ \mathcal{K}\cap \mathcal{A}_{\varepsilon}$ in $\mathcal{Z}_T^1$. Due to the compactness criterion presented in Lemma \ref{lemc}, we conclude that $\mathcal{K}_{\varepsilon}$ is a compact set in $\mathcal{Z}_T^1$.  The proof is complete.
\end{proof}

Based on Lemmas  \ref{lem3.0} and \ref{lemt}, in order to show the tightness of $\{X^{(n)}\}_{n\in\mathbb{N}}$ in $\mathcal{Z}^1_T$ it is sufficient to prove $\{X^{(n)}\}_{n\in\mathbb{N}}$ satisfies the  Aldous condition in $\mathbb{V}^*$, which is given as follows.
\begin{lemma}\label{lem6}
$\{X^{(n)}\}_{n\in\mathbb{N}}$ satisfies the  Aldous condition in $\mathbb{V}^*$ in the sense of Definition \ref{aldo}.
\end{lemma}
\begin{proof}
Recall Lemma \ref{lem3.0}, we know
\begin{equation}\label{e8}
\lim_{M\rightarrow\infty}\sup_{n\in\mathbb{N}}\mathbb{P}\big(\tilde{\tau}_M^{(n)}<T\big)=0,
\end{equation}
where the stopping time $\tilde{\tau}_M^{(n)}$ is defined by (\ref{stop1}).
%Note that the embedding $\mathbb{H}\subset \mathbb{V}^*$ is compact, from (\ref{e8}) and the fact that the
%Skorokhod topology in $D([0,T];\mathbb{V}^*)$ restricted to $\mathbb{C}_T(\mathbb{V}^*)$ coincides with the
%uniform topology,
%in light of \cite[Theorem 3.1]{Ja86} it suffices to
%show that for every $e\in \mathbb{H}_m$, $m\geq 1$, $\{{}_{\mathbb{V}^*}\langle X^{(n)},e\rangle_{\mathbb{V}}\}_{n\in\mathbb{N}}$ is tight in the space $C([0,T],\mathbb{R})$. By Aldous's criterion (cf.~\cite{A1}), it is sufficient to prove that for any stopping time $0\leq \zeta_n\leq T$ and $\varepsilon>0$,
%\begin{equation}\label{escon1}
%\lim_{\Delta\to 0^+}\sup_{n\in\mathbb{N}}\mathbb{P}\Big(|{}_{\mathbb{V}^*}\langle X^{(n)}_{\zeta_n+\Delta}-X^{(n)}_{\zeta_n},e\rangle_{\mathbb{V}}|>\varepsilon\Big)=0,
%\end{equation}
%where $\zeta_n+\Delta:=T\wedge (\zeta_n+\Delta) \vee 0$.
In addition, we have
\begin{eqnarray}\label{sec3es2}
\!\!\!\!\!\!\!\!&&\mathbb{P}\Big(\| X^{(n)}_{\zeta_n+\Delta}-X^{(n)}_{\zeta_n}\|_{\mathbb{V}^*}\geq\varepsilon\Big)
\nonumber \\
\leq \!\!\!\!\!\!\!\!&&\mathbb{P}\Big(\| X^{(n)}_{\zeta_n+\Delta}-X^{(n)}_{\zeta_n}\|_{\mathbb{V}^*}\geq\varepsilon,\tilde{\tau}_M^{(n)}\geq T\Big)+\mathbb{P}\big(\tilde{\tau}_M^{(n)}<T\big)
\nonumber \\
\leq \!\!\!\!\!\!\!\!&&\frac{1}{\varepsilon^{\frac{\alpha}{\alpha-1}}}\mathbb{E}\|X^{(n)}_{(\zeta_n+\Delta)\wedge\tilde{\tau}_M^{(n)}}-X^{(n)}_{\zeta_n\wedge\tilde{\tau}_M^{(n)}}\|_{\mathbb{V}^*}^{\frac{\alpha}{\alpha-1}}+\mathbb{P}\big(\tilde{\tau}_M^{(n)}<T\big).
\end{eqnarray}

Now we estimate the first term on the right hand side of (\ref{sec3es2}). By (\ref{eqf}) and applying B-D-G's inequality, it follows that
\begin{eqnarray}\label{esq1}
\!\!\!\!\!\!\!\!&&\mathbb{E}\| X^{(n)}_{(\zeta_n+\Delta)\wedge\tilde{\tau}_M^{(n)}}-X^{(n)}_{\zeta_n\wedge\tilde{\tau}_M^{(n)}}\|_{\mathbb{V}^*}^{\frac{\alpha}{\alpha-1}}
\nonumber \\
\leq\!\!\!\!\!\!\!\!&&C\mathbb{E}\Bigg\{\int_{\zeta_n\wedge \tilde{\tau}_M^{(n)}}^{(\zeta_n+\Delta)\wedge  \tilde{\tau}_M^{(n)}}\|\mathcal{P}_n \mathcal{A}(s,X^{(n)}_s)\|_{\mathbb{V}^*}ds\Bigg\}^{\frac{\alpha}{\alpha-1}}
\nonumber \\
 \!\!\!\!\!\!\!\!&&+C\mathbb{E}\Bigg\{\int_{\zeta_n\wedge \tilde{\tau}_M^{(n)}}^{(\zeta_n+\Delta)\wedge  \tilde{\tau}_M^{(n)}}\| \mathcal{P}_n \mathcal{B}(s,X^{(n)}_s)\tilde{\mathcal{P}}_n\|_{L_2(U, \mathbb{H})}^2ds\Bigg\}^{\frac{\alpha}{2(\alpha-1)}}
 \nonumber \\
=: \!\!\!\!\!\!\!\!&&\text{(I)}+\text{(II)}.
\end{eqnarray}
For $\text{(I)}$, by $(\mathbf{A_4})$ and H\"{o}lder's inequality, we have
\begin{eqnarray}\label{e9}
\text{(I)}\leq  \!\!\!\!\!\!\!\!&&C |\Delta|^{\frac{1}{\alpha-1}}\cdot \mathbb{E}\Bigg\{\int_{\zeta_n\wedge \tilde{\tau}_M^{(n)}}^{(\zeta_n+\Delta)\wedge  \tilde{\tau}_M^{(n)}}\|\mathcal{P}_n \mathcal{A}(s,X^{(n)}_s)\|_{\mathbb{V}^*}^{\frac{\alpha}{\alpha-1}}ds\Bigg\}
\nonumber \\
\leq \!\!\!\!\!\!\!\!&&C |\Delta|^{\frac{1}{\alpha-1}}\cdot\mathbb{E}\Bigg\{\int_{0}^{T\wedge \tilde{\tau}_M^{(n)}}\big(1+\|X_{s}^{(n)}\|_{\mathbb{V}}^\alpha\big)
\big(1+\|X_{s}^{(n)}\|_{\mathbb{H}}^\beta\big)ds\Bigg\}
\nonumber \\
\leq \!\!\!\!\!\!\!\!&&C_{M,T} |\Delta|^{\frac{1}{\alpha-1}}.
\end{eqnarray}
For $\text{(II)}$, we can get
\begin{eqnarray}\label{esq3}
\text{(II)}\leq  \!\!\!\!\!\!\!\!&&C\mathbb{E}\Bigg\{\int_{\zeta_n\wedge \tilde{\tau}_M^{(n)}}^{(\zeta_n+\Delta)\wedge  \tilde{\tau}_M^{(n)}}
\big(1+\|X_{s}^{(n)}\|_{\mathbb{H}}^{\beta}\big)ds\Bigg\}^{\frac{\alpha}{2(\alpha-1)}}
\nonumber \\
 \leq\!\!\!\!\!\!\!\!&&C_{M}|\Delta|^{\frac{\alpha}{2(\alpha-1)}}.
\end{eqnarray}
%Note that
%$$\mathbb{E}\Bigg\{\int_{0}^{ \tilde{\tau}_M^{(n)}}\big(1+\|X_{s}^{(n)}\|_{\mathbb{V}}^\alpha\big)ds\Bigg\}^{\frac{\alpha}{2(\alpha-1)}}\leq C_{M,T}<\infty.$$
%For any fixed $M>0$, by (\ref{esq3}), the absolute continuity of the Lebesgue integral and
%the dominated convergence theorem, it leads to
%\begin{equation}\label{e10}
%\lim_{\Delta\to 0^+}\sup_{n\in\mathbb{N}}\text{(II)}=0.
%\end{equation}
Combining (\ref{esq1})-(\ref{esq3}) gives
\begin{eqnarray}\label{esq4}
\lim_{\Delta\to 0}\sup_{n\in\mathbb{N}}\mathbb{E}\|X^{(n)}_{(\zeta_n+\Delta)\wedge\tilde{\tau}_M^{(n)}}-X^{(n)}_{\zeta_n\wedge\tilde{\tau}_M^{(n)}}\|_{\mathbb{V}^*}^{\frac{\alpha}{\alpha-1}}=0.
\end{eqnarray}
Finally, taking  into account (\ref{e8}), (\ref{sec3es2}) and (\ref{esq4}) and  letting $\Delta\to 0$  then $M\to\infty$ in (\ref{sec3es2}), we conclude that the Aldous condition holds.
The proof is completed.
\end{proof}

\vspace{2mm}
We now give the tightness of $\{X^{(n)}\}_{n\in\mathbb{N}}$ in $\mathcal{Z}^1_T$.
\begin{lemma}\label{coro1}
$\{X^{(n)}\}_{n\in\mathbb{N}}$ is tight in $\mathcal{Z}^1_T$.
\end{lemma}
\begin{proof}
Combining Lemmas  \ref{lem3.0}, \ref{lemt} and \ref{lem6}, the assertion follows.
\end{proof}

\vspace{1mm}
The following lemma shows the tightness of $\{\mathcal{A}^{(n)}(\cdot):=\mathcal{A}(\cdot,X^{(n)}_{\cdot})\}_{n\in\mathbb{N}}$  in the space $\mathcal{Z}^2_T$.
\begin{lemma}\label{coro2}
$\{\mathcal{A}^{(n)}(\cdot)\}_{n\in\mathbb{N}}$ is tight in $\mathcal{Z}^2_T$.
\end{lemma}
\begin{proof}
The tightness  of $\{\mathcal{A}^{(n)}(\cdot)\}_{n\in\mathbb{N}}$  in $\mathcal{Z}^2_T$ follows directly from the Banach-Alaoglu theorem and the estimate (\ref{apriA}).
\end{proof}

\subsection{Passage to the limit}\label{sec2.3}
In this part, we aim to prove Theorems \ref{th1}, where the Jakubowski's beautiful generalization of the Skorokhod's representation theorem, in the form presented by Brze\'{z}niak and Ondrej\'{a}t, for nonmetric spaces (see Lemma \ref{sko1} in Appendix) and the theory of pseudo-monotone operators play a crucial role.

In order to apply Jakubowski's version of the Skorokhod theorem, we first verify that $\mathcal{Z}_T^1$ is a standard Borel space (see Appendix for the definition), whose proof is inspired by \cite{Liang}.
\begin{lemma}\label{thsb}
$\mathcal{Z}_T^1$ is a standard Borel space.
\end{lemma}

\begin{proof}
Let $\mathcal{Y}:=L^2([0,T];\mathbb{V}^*)$ with Borel $\sigma$-algebra $\mathscr{B}(\mathcal{Y})$. In order to show the assertion, according to Theorem 2.3 of Chapter V in \cite{P67} and the fact that $\mathcal{Y}$ is a standard Borel space, it is sufficient to show the following three claims:

\vspace{1mm}
\noindent(i) the embedding $\mathcal{Z}_T^1\subset \mathcal{Y}$ is continuous;

\vspace{1mm}
\noindent(ii)   $\mathcal{Z}_T^1\in \mathscr{B}(\mathcal{Y})$;

\vspace{1mm}
\noindent(iii) $\mathscr{B}(\tau_{\mathcal{Z}^1_T})=\mathscr{B}(\mathcal{Y}) \cap \mathcal{Z}^1_T$.

\vspace{1mm}
First, the claim (i) follows directly from $\mathbb{C}_T(\mathbb{V}^*)\subset \mathcal{Y}$ continuously. Next, we focus on proving claims (ii) and (iii).

\vspace{1mm}
\noindent\textbf{Proof of (ii).} By claim (i) we have
\begin{equation}\label{a1}
\mathscr{B}(\mathcal{Y}) \cap \mathcal{Z}^1_T\subset \mathscr{B}(\tau_{\mathcal{Z}^1_T}).
\end{equation}
Fix $N\in\mathbb{N}$. Let us denote
$$(L^{\alpha}_N([0,T];\mathbb{V}))_w:=\Big\{x\in L^{\alpha}([0,T];\mathbb{V})\big| \|x\|_{L^{\alpha}([0,T];\mathbb{V})}\leq N\Big\},$$
which is endowed with the weak topology on
$L^{\alpha}([0,T];\mathbb{V})$. Then $(L^{\alpha}_N([0,T];\mathbb{V}))_w$ is a compact and metrizable, hence complete and separable, space with metric $d_1$.
Fix $N\in\mathbb{N}$. Set
$$\mathcal{Z}_T^{(N)}:=(L^{\alpha}_N([0,T];\mathbb{V}))_w\cap \mathbb{C}_T(\mathbb{V}^*)\cap L^{\alpha}([0,T];\mathbb{H}),$$
which is a closed subset of $\mathcal{Z}^1_T$. The metrics on   $\mathbb{C}_T(\mathbb{V}^*)$ and $L^{\alpha}([0,T];\mathbb{H})$ are denoted by $d_2$ and $d_3$, respectively. Now, let $\mathcal{Z}_T^{(N)}$ be endowed with the metric $d:=\max\{d_1,d_2,d_3\}$. Since the intersection of finite separable metric space (with the maximal metric) is a separable metric space, it follows that $\mathcal{Z}_T^{(N)}$ is a separable metric space. We intend to show that $\mathcal{Z}_T^{(N)} $ is complete. To this end, it is sufficient to show that for a sequence $\{x_k\}_{k\in\mathbb{N}}$ converging to $x^{(i)}$ in $d_i$, $i=1,2,3$,  we have
\begin{equation}\label{a2}
 x^{(1)}=x^{(2)}=x^{(3)}.
 \end{equation}
Since $x_k\rightharpoonup x^{(1)}$ and $x_k\rightharpoonup x^{(3)}$ both in $L^{\alpha}([0,T];\mathbb{H})$,  it follows that
$$x^{(1)}=x^{(3)}.$$
In addition, $x_k\to x^{(2)}$ and $x_k\to x^{(3)}$ both in $L^{\alpha}([0,T];\mathbb{V}^*)$, we can deduce that
$$x^{(2)}=x^{(3)}.$$
Hence, (\ref{a2}) follows, then  $\mathcal{Z}_T^{(N)}$ is a complete separable metric space. Furthermore, the following embeddings are continuous
\begin{equation}\label{a4}
(\mathcal{Z}_T^{(N)},d)\subset (\mathcal{Z}^1_T,\tau_{\mathcal{Z}^1_T})\subset \mathcal{Y}.
 \end{equation}
Therefore, in view of  Theorem 2.4 of Chapter V in \cite{P67} we can obtain
\begin{equation}\label{a5}
\mathscr{B}((\mathcal{Z}_T^{(N)},d))\subset \mathscr{B}(\mathcal{Y}).
 \end{equation}
Moreover, by (\ref{a4}) we have
\begin{equation}\label{a6}
\mathscr{B}(\mathcal{Y}) \cap \mathcal{Z}_T^{(N)}\subset \mathscr{B}((\mathcal{Z}_T^{(N)},d)).
 \end{equation}
Combining (\ref{a5})-(\ref{a6}) we obtain
\begin{equation}\label{a3}
\mathscr{B}((\mathcal{Z}_T^{(N)},d))=\mathscr{B}(\mathcal{Y})\cap \mathcal{Z}_T^{(N)},
 \end{equation}
which yields that
$$\mathcal{Z}^1_T=\bigcup_{N=1}^{\infty}\mathcal{Z}_T^{(N)}\in \mathscr{B}(\mathcal{Y}).$$
The proof of claim (ii) is completed.

\vspace{1mm}
\noindent\textbf{Proof of (iii).}  Since $\mathcal{Z}_T^{(N)}$ is the closed subset of $\mathcal{Z}_T^1$, it is clear that $\mathcal{Z}_T^{(N)}\in \mathscr{B}(\tau_{\mathcal{Z}^1_T})$. Then we can obtain
 $$\mathscr{B}(\tau_{\mathcal{Z}^1_T})\cap \mathcal{Z}_T^{(N)}=\Big\{B\in\mathscr{B}(\tau_{\mathcal{Z}^1_T})|B\subset \mathcal{Z}_T^{(N)}\Big\}$$
and
 $$\mathscr{B}(\tau_{\mathcal{Z}^1_T})=\bigcup_{N=1}^{\infty}\Big\{B\in\mathscr{B}(\tau_{\mathcal{Z}^1_T})|B\subset \mathcal{Z}_T^{(N)}\Big\}.$$
Moreover,  the embedding $(\mathcal{Z}_T^{(N)},d)\subset(\mathcal{Z}_T^{(N)},\tau_{\mathcal{Z}^1_T}\cap \mathcal{Z}_T^{(N)})$ is continuous, which implies
\begin{equation}\label{ap01}
\mathscr{B}(\tau_{\mathcal{Z}^1_T}\cap \mathcal{Z}_T^{(N)})\subset\mathscr{B}((\mathcal{Z}_T^{(N)},d)).
\end{equation}
Suppose that $A$  is an $\tau_{\mathcal{Z}^1_T}$-closed subset of $\mathcal{Z}^1_T$. Then $A\cap \mathcal{Z}_T^{(N)}$ is $\tau_{\mathcal{Z}^1_T}$-closed. Then by \eref{a3} and \eref{ap01}
we have
 \begin{eqnarray}\label{ap02}
A\cap \mathcal{Z}_T^{(N)}\in\!\!\!\!\!\!\!\!&&\mathscr{B}(\tau_{\mathcal{Z}^1_T}\cap \mathcal{Z}_T^{(N)})
\nonumber \\
\subset\!\!\!\!\!\!\!\!&&\mathscr{B}((\mathcal{Z}_T^{(N)},d))
\nonumber \\
=\!\!\!\!\!\!\!\!&&\mathscr{B}(\mathcal{Y})\cap \mathcal{Z}_T^{(N)}.
\end{eqnarray}
Note that $ \mathcal{Z}_T^{(N)}\in\mathscr{B}(\mathcal{Y})$,  in view of claim (ii) we can get
$$\mathscr{B}(\mathcal{Y})\cap \mathcal{Z}_T^{(N)}=\Big\{B\in\mathscr{B}(\mathcal{Y})|B\subset \mathcal{Z}_T^{(N)}\Big\}\subset\Big\{B\in\mathscr{B}(\mathcal{Y})|B\subset \mathcal{Z}^1_T\Big\}=\mathscr{B}(\mathcal{Y})\cap\mathcal{Z}^1_T.$$
Then by \eref{ap02} we have
$$A=\bigcup_{N=1}^{\infty}A\cap \mathcal{Z}_T^{(N)}\in\mathscr{B}(\mathcal{Y})\cap\mathcal{Z}^1_T.$$
Thus, $\mathscr{B}(\tau_{\mathcal{Z}^1_T})\subset\mathscr{B}(\mathcal{Y})\cap\mathcal{Z}^1_T$, which together with \eref{a1}
implies $\mathscr{B}(\tau_{\mathcal{Z}^1_T})=\mathscr{B}(\mathcal{Y})\cap\mathcal{Z}^1_T$. The proof is completed.
\end{proof}

Now, we need to check that the spaces $\mathcal{Z}^1_T$, $\mathcal{Z}^2_T$ satisfy the conditions in Lemma \ref{sko1}. In fact, since $\mathbb{C}_T(\mathbb{V}^*)$ and  $L^{\alpha}([0,T];\mathbb{H})$ are separable Banach spaces, it is easy to see the conditions in Lemma \ref{sko1} hold. Furthermore, for the space $L^{\alpha}_w([0,T];{\mathbb{V}})$, it suffices to put
$$f_m(u):=\int_0^T{}_{\mathbb{V}^*}\langle v_m(t),u(t)\rangle_{\mathbb{V}}dt\in\mathbb{R},~u\in L^{\alpha}_w([0,T];{\mathbb{V}}),~m\in\mathbb{N},$$
where $\{v_m\}_{m\geq 1}$ is a dense subset of $L^{\frac{\alpha}{\alpha-1}}([0,T];{\mathbb{V}^*})$. Since $\{v_m\}_{m\geq 1}$ is dense in $L^{\frac{\alpha}{\alpha-1}}([0,T];{\mathbb{V}^*})$, it separates points of $L^{\alpha}([0,T];{\mathbb{V}})$.
Moreover, since $(\mathcal{Z}^1_T,\mathscr{B}(\tau_{\mathcal{Z}^1_T}))$  is a standard Borel space (i.e.~Lemma  \ref{thsb}), the $\sigma$-algebra generated by the sequence of the above continuous functions,
which separate the points in $\mathcal{Z}^1_T$, is exactly $\mathscr{B}(\tau_{\mathcal{Z}^1_T})$ by Theorem \ref{tha1}. Thus all the conditions in Lemma \ref{sko1} are satisfied for $\mathcal{Z}^1_T$. From the same reason, we deduce that the conditions in Lemma \ref{sko1} also hold for $\mathcal{Z}^2_T$.

Set
$$\Xi_T:=\mathcal{Z}^1_T\times\mathcal{Z}^2_T\times \mathbb{C}_T(U_1),$$
where $U_1$ is a Hilbert space such that the embedding $U\subset U_1$ is Hilbert-Schmidt.
Since $\{X^{(n)}\}_{n\in\mathbb{N}}$, $\{\mathcal{A}^{(n)}(\cdot)\}_{n\in\mathbb{N}}$  are tight in $\mathcal{Z}^1_T$, $\mathcal{Z}^2_T$, respectively, it is clear that $\{(X^{(n)},\mathcal{A}^{(n)}(\cdot),W)\}_{n\in\mathbb{N}}$ is also tight in $\Xi_T$.
By the Jakubowski's version of the Skorokhod theorem, there exists a probability space $(\tilde{\Omega},\tilde{\mathscr{F}},\tilde{\mathbb{P}})$, and on this space, $\Xi_T$-valued random variables $$(\tilde{X}^{(n)},\tilde{\mathcal{A}}^{(n)}(\cdot),\tilde{W}^{(n)}), ~(\tilde{X},\tilde{\mathcal{A}}(\cdot),\tilde{W})$$
(here choosing a subsequence if necessary) such that

\vspace{2mm}
\noindent(i) the law of $(\tilde{X}^{(n)},\tilde{\mathcal{A}}^{(n)}(\cdot),\tilde{W}^{(n)})$ under $\tilde{\mathbb{P}}$ is equivalent to the law of $(X^{(n)},\mathcal{A}^{(n)}(\cdot),W)$ under $\mathbb{P}$;

\vspace{2mm}
\noindent(ii) the following convergence hold
\begin{eqnarray}
\!\!\!\!\!\!\!\!&&\tilde{X}^{(n)}\to \tilde{X}~\text{in}~\mathcal{Z}^1_T,~~\tilde{\mathbb{P}}\text{-a.s.},~\text{as}~n\to\infty,\label{es80}
 \\
\!\!\!\!\!\!\!\!&&\tilde{\mathcal{A}}^{(n)}(\cdot)\rightharpoonup \tilde{\mathcal{A}}(\cdot)~\text{in}~L^{\frac{\alpha}{\alpha-1}}([0,T];\mathbb{V}^*),~~\tilde{\mathbb{P}}\text{-a.s.},~\text{as}~n\to\infty,\label{con10}
% \\
%\!\!\!\!\!\!\!\!&&\tilde{\mathcal{B}}^{(n)}(\cdot)\rightharpoonup \tilde{\mathcal{B}}(\cdot)~\text{in}~L^{2}([0,T];L_2(U;\mathbb{H})),~\tilde{\mathbb{P}}\text{-a.s.},~\text{as}~n\to\infty,\label{con2}
\end{eqnarray}
where ``$\rightharpoonup$'' stands for the weak convergence.

\vspace{2mm}
\noindent(iii) $\tilde{W}^{(n)}\to\tilde{W}$ in $\mathbb{C}_T(U_1)$, $\tilde{\mathbb{P}}$-a.s..

\vspace{2mm}
 Let $(\tilde{\mathscr{F}}^{(n)}_t)_{t\in[0,T]}$ be the filtration satisfying the usual conditions and generated by
 $\Big\{\tilde{X}^{(n)}_s,\tilde{W}^{(n)}_s:s\in[0,t]\Big\}.$
 We note that by the claim (i),
 \begin{eqnarray*}
\!\!\!\!\!\!\!\!&&\mathbb{P}(W_t-W_s\in\cdot|\mathscr{F}_s)=\mathbb{P}(W_t-W_s\in\cdot)
 \\
\Rightarrow\!\!\!\!\!\!\!\!&&\tilde{\mathbb{P}}(\tilde{W}^{(n)}_t-\tilde{W}^{(n)}_s\in\cdot|\tilde{\mathscr{F}}^{(n)}_s)=\tilde{\mathbb{P}}(\tilde{W}^{(n)}_t-\tilde{W}^{(n)}_s\in\cdot).
%\\
%\Rightarrow\!\!\!\!\!\!\!\!&&\tilde{\mathbb{P}}(\tilde{W}_t-\tilde{W}_s\in\cdot|\tilde{\mathscr{F}}_s)=\tilde{\mathbb{P}}(\tilde{W}_t-\tilde{W}_s\in\cdot),
\end{eqnarray*}
%where $(\tilde{\mathscr{F}}_t)_{t\in[0,T]}$ is the filtration satisfying the usual conditions and generated by
% $\big\{\tilde{X}_s,\tilde{W}_s:s\in[0,t]\big\}.$
In other word, $\tilde{W}^{(n)}$ is an $(\tilde{\mathscr{F}}^{(n)}_t)$-cylindrical Wiener process on $U$.  Moreover, from (\ref{eqf}) and the claim (i),  the following identity holds $\tilde{\mathbb{P}}$-a.s.
 \begin{eqnarray}\label{eqf2}
\tilde{X}^{(n)}_t
=\!\!\!\!\!\!\!\!&& x^{(n)}+\int_0^t\mathcal{P}_n\mathcal{A}(s,\tilde{X}^{(n)}_s)ds
\nonumber \\
\!\!\!\!\!\!\!\!&&+\int_0^t\mathcal{P}_n\mathcal{B}(s,\tilde{X}^{(n)}_s)\tilde{\mathcal{P}}_nd\tilde{W}^{(n)}_s,~t\in[0,T].
\end{eqnarray}
By the claim (i), the convergence (\ref{con10}) and the fact that the law of $\mathcal{A}(\cdot,\tilde{X}^{(n)}_{\cdot})$ under $\tilde{\mathbb{P}}$ is equivalent to the law of $\mathcal{A}^{(n)}(\cdot)$ under $\mathbb{P}$, we can infer that (at least) along a subsequence still denoted by $\{n\}$ we have
\begin{equation}\label{con1}
\mathcal{A}(\cdot,\tilde{X}^{(n)}_{\cdot})\rightharpoonup \tilde{\mathcal{A}}(\cdot)~\text{in}~L^{\frac{\alpha}{\alpha-1}}([0,T];\mathbb{V}^*),~~\tilde{\mathbb{P}}\text{-a.s.},~\text{as}~n\to\infty.
\end{equation}

In addition, by Lemmas \ref{lemma4.10} and \ref{lem3.0}, we can obtain  same bounds  hold for $ \tilde{X}^{(n)}$. More precisely, we have
\begin{equation*}\label{apri4}
		\sup_{n\in\mathbb{N}}\bigg\{\sup_{t\in[0,T]}\tilde{\mathbb{E}}\|\tilde{X}^{(n)}_t\|_{\mathbb{H}}^{2-\eta_0}+\tilde{\mathbb{E}}\int_0^T\|\tilde{X}^{(n)}_t\|_{\mathbb{V}}^{\alpha-\eta_0}dt\bigg\}
		\leq C_T(1+\|x\|_{\mathbb{H}}^{2-\eta_0}),
	\end{equation*}
and
for any $\varepsilon>0$ there exists  $\mathcal{K}>0$ such that for any $p\geq 2$,
\begin{equation}
\sup_{n\in\mathbb{N}}\tilde{\mathbb{P}}\Bigg(\sup_{t\in[0,T]}\|\tilde{X}^{(n)}_t\|_{\mathbb{H}}^p+\int_0^T\|\tilde{X}^{(n)}_t\|_{\mathbb{V}}^{\alpha}dt\geq \mathcal{K}\Bigg)\leq\varepsilon. \label{apri1}
\end{equation}
Using the lower semicontinuity of norms $\|\cdot\|_{\mathbb{H}}$, $\|\cdot\|_{\mathbb{V}}$ in $\mathbb{V}^*$, by the convergence (\ref{es80}) and Fatou's lemma we can deduce
\begin{eqnarray}\label{apri5}
		\!\!\!\!\!\!\!\!&&\sup_{t\in[0,T]}\tilde{\mathbb{E}}\|\tilde{X}_t\|_{\mathbb{H}}^{2-\eta_0}+\tilde{\mathbb{E}}\int_0^T\|\tilde{X}_t\|_{\mathbb{V}}^{\alpha-\eta_0}dt
\nonumber \\
\leq\!\!\!\!\!\!\!\!&&\sup_{t\in[0,T]}\tilde{\mathbb{E}}\liminf_{n\to\infty}\|\tilde{X}^{(n)}_t\|_{\mathbb{H}}^{2-\eta_0}+\tilde{\mathbb{E}}\int_0^T\liminf_{n\to\infty}\|\tilde{X}^{(n)}_t\|_{\mathbb{V}}^{\alpha-\eta_0}dt
\nonumber \\
\leq\!\!\!\!\!\!\!\!&&\liminf_{n\to\infty}\sup_{t\in[0,T]}\tilde{\mathbb{E}}\|\tilde{X}^{(n)}_t\|_{\mathbb{H}}^{2-\eta_0}+\liminf_{n\to\infty}\tilde{\mathbb{E}}\int_0^T\|\tilde{X}^{(n)}_t\|_{\mathbb{V}}^{\alpha-\eta_0}dt
\nonumber \\
\leq\!\!\!\!\!\!\!\!&&
 C_T(1+\|x\|_{\mathbb{H}}^{2-\eta_0}),
	\end{eqnarray}
and
for any $\varepsilon>0$  we can choose $\mathcal{K}$ such that
\begin{eqnarray}\label{apri2}
\!\!\!\!\!\!\!\!&&\tilde{\mathbb{P}}\Bigg(\sup_{t\in[0,T]}\|\tilde{X}_t\|_{\mathbb{H}}^p+\int_0^T\|\tilde{X}_t\|_{\mathbb{V}}^{\alpha}dt\geq \mathcal{K}\Bigg)
\nonumber \\
\leq\!\!\!\!\!\!\!\!&&\tilde{\mathbb{P}}\Bigg(\sup_{t\in[0,T]}\liminf_{n\to\infty}\|\tilde{X}^{(n)}_t\|_{\mathbb{H}}^p+\int_0^T\liminf_{n\to\infty}\|\tilde{X}^{(n)}_t\|_{\mathbb{V}}^{\alpha}dt\geq \mathcal{K}\Bigg)
\nonumber \\
\leq\!\!\!\!\!\!\!\!&&\tilde{\mathbb{P}}\Bigg(\liminf_{n\to\infty}\Big\{\sup_{t\in[0,T]}\|\tilde{X}^{(n)}_t\|_{\mathbb{H}}^p+\int_0^T\|\tilde{X}^{(n)}_t\|_{\mathbb{V}}^{\alpha}dt\Big\}\geq \mathcal{K}\Bigg)
%+\tilde{\mathbb{P}}\big(\tilde{\tau}_M^{(n)}< T\big)
%\nonumber \\
%\leq\!\!\!\!\!\!\!\!&&\tilde{\mathbb{E}}\Bigg[\liminf_{n\to\infty}\Bigg(\sup_{t\in[0,T\wedge\tilde{\tau}_M^{(n)}]}\|\tilde{X}^{(n)}_t\|_{\mathbb{H}}^2+\int_0^{T\wedge\tilde{\tau}_M^{(n)}}\|\tilde{X}^{(n)}_t\|_{\mathbb{V}}^{\alpha}dt\Bigg)\Bigg]\Big/
%R+\tilde{\mathbb{P}}\big(\tilde{\tau}_M^{(n)}< T\big)
%\nonumber \\
%\leq\!\!\!\!\!\!\!\!&&\Bigg\{\liminf_{n\to\infty}\tilde{\mathbb{E}}\Bigg(\sup_{t\in[0,T\wedge\tilde{\tau}_M^{(n)}]}\|\tilde{X}^{(n)}_t\|_{\mathbb{H}}^2+\int_0^{T\wedge\tilde{\tau}_M^{(n)}}\|\tilde{X}^{(n)}_t\|_{\mathbb{V}}^{\alpha}dt\Bigg)\Bigg\}\Big/
%R+\tilde{\mathbb{P}}\big(\tilde{\tau}_M^{(n)}< T\big)
\nonumber \\
\leq\!\!\!\!\!\!\!\!&&\sup_{n\in\mathbb{N}}\tilde{\mathbb{P}}\Bigg(\sup_{t\in[0,T]}\|\tilde{X}^{(n)}_t\|_{\mathbb{H}}^p+\int_0^T\|\tilde{X}^{(n)}_t\|_{\mathbb{V}}^{\alpha}dt\geq \mathcal{K}\Bigg)
\nonumber \\
\leq\!\!\!\!\!\!\!\!&&\varepsilon.
\end{eqnarray}

In the sequel, we investigate the convergence of (\ref{eqf2}), as $n\to\infty$, where we can choose a subsequence if necessary.

\vspace{1mm}
In the following lemma, we present the convergence of the stochastic integral in (\ref{eqf2}).
 \begin{lemma}\label{lem4}
$\int_0^{\cdot} \mathcal{P}_n\mathcal{B}(s,\tilde{X}^{(n)}_s)\tilde{\mathcal{P}}_nd\tilde{W}^{(n)}_s\to\int_0^{\cdot} \mathcal{B}(s,\tilde{X}_s)d\tilde{W}_s$ in $L^{\infty}([0,T];\mathbb{H})$ in probability as $n\to\infty$.
 \end{lemma}
\begin{proof}
According to the characterization of the convergence in probability for stochastic integral as presented in Lemma 4.3 in \cite{BMX}, it suffices to show that for any $t\in[0,T]$,
\begin{equation}
\int_0^t \|\mathcal{P}_n\mathcal{B}(s,\tilde{X}^{(n)}_s)\tilde{\mathcal{P}}_n-\mathcal{B}(s,\tilde{X}_s)\|_{L_2(U,\mathbb{H})}^2ds\to 0\label{con3}
\end{equation}
in probability as $n\to\infty$.

Note that by (\ref{es80}) we know
\begin{equation}\label{e21}
\int_0^T\|\tilde{X}^{(n)}_t-\tilde{X}_t\|_{\mathbb{H}}^{\kappa}dt\to 0,~~\tilde{\mathbb{P}}\text{-a.s.},~\forall \kappa\in[1,\alpha).
\end{equation}
Then, there exists a $\tilde{\mathbb{P}}\otimes dt$-null set $\mathcal{N}$ such that for any $(t,\omega)\in (\Omega\times[0,T])\backslash \mathcal{N}$, along a subsequence still denoted by $\{n\}$ we have
\begin{equation}\label{e27}
\|\tilde{X}^{(n)}_t(\omega)-\tilde{X}_t(\omega)\|_{\mathbb{H}}\to 0~\text{as}~n\to\infty.
\end{equation}
Moreover, we know
\begin{eqnarray*}
\!\!\!\!\!\!\!\!&&\|\mathcal{P}_n\mathcal{B}(s,\tilde{X}^{(n)}_s)\tilde{\mathcal{P}}_n-\mathcal{B}(s,\tilde{X}_s)\|_{L_2(U,\mathbb{H})}^2
\nonumber \\
\leq\!\!\!\!\!\!\!\!&& C\|\mathcal{P}_n\big(\mathcal{B}(s,\tilde{X}^{(n)}_s)-\mathcal{B}(s,\tilde{X}_s)\big)\tilde{\mathcal{P}}_n\|_{L_2(U,\mathbb{H})}^2
+C\|\mathcal{P}_n\mathcal{B}(s,\tilde{X}_s)\tilde{\mathcal{P}}_n-\mathcal{B}(s,\tilde{X}_s)\|_{L_2(U,\mathbb{H})}^2
\nonumber \\
\leq\!\!\!\!\!\!\!\!&& C\|\mathcal{B}(s,\tilde{X}^{(n)}_s)-\mathcal{B}(s,\tilde{X}_s)\|_{L_2(U,\mathbb{H})}^2
+C\|\mathcal{P}_n\mathcal{B}(s,\tilde{X}_s)\tilde{\mathcal{P}}_n-\mathcal{B}(s,\tilde{X}_s)\|_{L_2(U,\mathbb{H})}^2
\nonumber \\
=:\!\!\!\!\!\!\!\!&& I^{(n)}_s+II^{(n)}_s.
\end{eqnarray*}
In what follows, we only focus on the convergence of $I^{(n)}$, since the convergence of $II^{(n)}$ follows directly from the property of orthonormal projections.

\vspace{1mm}
We intend to prove
\begin{equation*}
\int_0^tI^{(n)}_sds\to 0~~\text{in probability}~\text{as}~n\to\infty.
\end{equation*}
In light of the continuity condition (\ref{con3}) and the convergence (\ref{e27}), it suffices to show
\begin{equation}\label{e26}
\int_0^t\|\mathcal{B}(s,\tilde{X}^{(n)}_s)\|_{L_2(U,\mathbb{H})}^2ds\to \int_0^t\|\mathcal{B}(s,\tilde{X}_s)\|_{L_2(U,\mathbb{H})}^2ds
\end{equation}
in probability as $n\to\infty$.

Let $\chi_R\in C^{\infty}_c(\mathbb{R})$ be a cut-off function with
$$\chi_R(r)=\begin{cases} 1,~~~~|r|\leq R&\quad\\
0,~~~~|r|>2R.&\quad\end{cases}$$
Set
\begin{eqnarray*}
\!\!\!\!\!\!\!\!&&\Psi_R(t,w):=\int_0^t\|\mathcal{B}(s,w_s)\|_{L_2(U,\mathbb{H})}^2\chi_R(\|w_s\|_{\mathbb{H}})ds,
\nonumber \\
\!\!\!\!\!\!\!\!&& \Psi(t,w):=\int_0^t\|\mathcal{B}(s,w_s)\|_{L_2(U,\mathbb{H})}^2ds.
\end{eqnarray*}
On the one hand, by (\ref{conb3}), (\ref{e27}) and the continuity of $\chi_R$   we have as $n\to \infty$
$$\Big|\|\mathcal{B}(s,\tilde{X}^{(n)}_s)\|_{L_2(U,\mathbb{H})}^2\chi_R(\|\tilde{X}^{(n)}_s\|_{\mathbb{H}})-
\|\mathcal{B}(s,\tilde{X}_s)\|_{L_2(U,\mathbb{H})}^2\chi_R(\|\tilde{X}_s\|_{\mathbb{H}})\Big|\to0,~~\tilde{\mathbb{P}}\otimes dt\text{-a.e.},$$
which combining with the dominated convergence theorem yields
\begin{equation}\label{e28}
\Psi_R(t,\tilde{X}^{(n)})\to \Psi_R(t,\tilde{X}),~~\tilde{\mathbb{P}}\text{-a.s.},~\text{as}~n\to\infty.
\end{equation}
On the other hand, by the definition of $\chi_R$ we obtain that for any $\varepsilon>0$,
\begin{eqnarray*}
\!\!\!\!\!\!\!\!&&\tilde{\mathbb{P}}\Big(|\Psi(t,\tilde{X}^{(n)})-\Psi_R(t,\tilde{X}^{(n)})|>\varepsilon\Big)
\nonumber \\
=\!\!\!\!\!\!\!\!&&\tilde{\mathbb{P}}\Big(|\Psi(t,\tilde{X}^{(n)})-\Psi_R(t,\tilde{X}^{(n)})|>\varepsilon,\sup_{t\in[0,T]}\|\tilde{X}^{(n)}_t\|_{\mathbb{H}}\leq R\Big)
\nonumber \\
\!\!\!\!\!\!\!\!&&+ \tilde{\mathbb{P}}\Big(|\Psi(t,\tilde{X}^{(n)})-\Psi_R(t,\tilde{X}^{(n)})|>\varepsilon,\sup_{t\in[0,T]}\|\tilde{X}^{(n)}_t\|_{\mathbb{H}}> R\Big)
\nonumber \\
\leq\!\!\!\!\!\!\!\!&& \sup_{n\in\mathbb{N}}\tilde{\mathbb{P}}\Big(\sup_{t\in[0,T]}\|\tilde{X}^{(n)}_t\|_{\mathbb{H}}> R\Big).
\end{eqnarray*}
In view of (\ref{apri1}), letting $n\to\infty$ and $R\to \infty$  we derive
\begin{equation}\label{e29}
|\Psi(t,\tilde{X}^{(n)})-\Psi_R(t,\tilde{X}^{(n)})|\to 0~~\text{in probability}.
\end{equation}
Applying similar argument, we also obtain
\begin{equation}\label{e31}
|\Psi(t,\tilde{X})-\Psi_R(t,\tilde{X})|\to 0~~\text{in probability}.
\end{equation}
Collecting (\ref{e28})-(\ref{e31}), we conclude that (\ref{e26}) follows.

\vspace{1mm}

We complete the proof.
\end{proof}

From Lemma \ref{lem4}, along a subsequence still denoted by $\{n\}$ we have that as $n\to\infty$,
\begin{equation}\label{con6}
\sup_{t\in[0,T]}\Big\|\int_0^{t} \mathcal{P}_n\mathcal{B}(s,\tilde{X}^{(n)}_s)\tilde{\mathcal{P}}_nd\tilde{W}^{(n)}_s-\int_0^{t} \mathcal{B}(s,\tilde{X}_s)d\tilde{W}_s\Big\|_{\mathbb{H}}\to0,~~\tilde{\mathbb{P}}\text{-a.s.}.
\end{equation}
Thus applying the convergence (\ref{es80})-(\ref{con1}) and (\ref{con6}), for any $v\in\cup_{n\geq 1}\mathbb{H}_n(\subset \mathbb{V})$, $\varphi\in L^{\infty}([0,T]\times \Omega;\mathbb{R})$ we obtain
\begin{eqnarray*}
\!\!\!\!\!\!\!\!&&\int_0^T{}_{\mathbb{V}^*}\langle \tilde{X}_t,\varphi_tv\rangle_{\mathbb{V}}dt
\nonumber \\
=\!\!\!\!\!\!\!\!&&\lim_{n\to\infty}\int_0^T{}_{\mathbb{V}^*}\langle \tilde{X}^{(n)}_t,\varphi_tv\rangle_{\mathbb{V}}dt
\nonumber \\
=\!\!\!\!\!\!\!\!&&\lim_{n\to\infty}\Bigg({}_{\mathbb{V}^*}\langle x^{(n)},v\rangle_{\mathbb{V}}\int_0^T\varphi_tdt+
\int_0^T\int_0^t{}_{\mathbb{V}^*}\langle\mathcal{P}_n\mathcal{A}(s,\tilde{X}^{(n)}_s),\varphi_tv\rangle_{\mathbb{V}}dsdt
\nonumber \\
\!\!\!\!\!\!\!\!&&+\int_0^T\langle\int_0^t\mathcal{P}_n\mathcal{B}(s,\tilde{X}^{(n)}_s)\tilde{\mathcal{P}}_nd\tilde{W}^{(n)}_s,\varphi_tv\rangle_{\mathbb{H}}dt\Bigg)
\nonumber \\
=\!\!\!\!\!\!\!\!&&\lim_{n\to\infty}\Bigg({}_{\mathbb{V}^*}\langle x^{(n)},v\rangle_{\mathbb{V}}\int_0^T\varphi_tdt+
\int_0^T{}_{\mathbb{V}^*}\langle\mathcal{P}_n\mathcal{A}(s,\tilde{X}^{(n)}_s),\int_s^T\varphi_tdt\cdot v \rangle_{\mathbb{V}}ds
\nonumber \\
\!\!\!\!\!\!\!\!&&+\int_0^T\langle\int_0^t\mathcal{P}_n\mathcal{B}(s,\tilde{X}^{(n)}_s)\tilde{\mathcal{P}}_nd\tilde{W}^{(n)}_s,\varphi_tv\rangle_{\mathbb{H}}dt\Bigg)
\nonumber \\
=\!\!\!\!\!\!\!\!&&\int_0^T{}_{\mathbb{V}^*}\langle x+
\int_0^t\tilde{\mathcal{A}}(s)ds
+\int_0^t\mathcal{B}(s,\tilde{X}_s)d\tilde{W}_s,\varphi_tv\rangle_{\mathbb{V}}dt,~~\tilde{\mathbb{P}}\text{-a.s.}.
\end{eqnarray*}
Thus we define
\begin{equation}\label{es1}
\bar{X}_t:=x+\int_0^t\tilde{\mathcal{A}}(s)ds+\int_0^t\mathcal{B}(s,\tilde{X}_s)d\tilde{W}_s,~t\in[0,T].
\end{equation}
It is clear that
\begin{equation}\label{es2}
\tilde{X}=\bar{X},~~\tilde{\mathbb{P}}\otimes dt\text{-a.e.}.
\end{equation}

%The  bounds (\ref{apri1})-(\ref{apri3}) imply that there exists $\hat{X}$, $\hat{\mathcal{A}}$ and $\hat{\mathcal{B}}$,
%such that
%\begin{eqnarray*}
%\!\!\!\!\!\!\!\!&& \text{(i)}~~ \tilde{X}^{(n)}\rightharpoonup\hat{X}~~~\text{in}~L^{\alpha}([0,T];\mathbb{V}),~~~\tilde{\mathbb{P}}\text{-a.s.};
%\nonumber \\
%\!\!\!\!\!\!\!\!&&\text{(ii)}~~\tilde{\mathcal{A}}^{(n)}(\cdot)\rightharpoonup\bar{\mathcal{A}}~~~\text{in}~L^{\frac{\alpha}{\alpha-1}}([0,T];\mathbb{V}^*),
%~~~\tilde{\mathbb{P}}\text{-a.s.};
%\nonumber \\
%\!\!\!\!\!\!\!\!&&\text{(iii)}~~\mathcal{P}_n\mathcal{B}(\cdot,\tilde{X}^{(n)}_{\cdot})\tilde{\mathcal{P}}_n\rightharpoonup\bar{\mathcal{B}}~~~\text{in}~L^{2}([0,T];L_2(U,\mathbb{H})),
%~~~\tilde{\mathbb{P}}\text{-a.s.};
%\nonumber \\
%\!\!\!\!\!\!\!\!&&\text{(iv)}~~ \int_{0}^{\cdot}\mathcal{P}_n\mathcal{B}(s,\tilde{X}^{(n)}_{s})\tilde{\mathcal{P}}_nd\tilde{W}_s\rightharpoonup\int_{0}^{\cdot}\bar{\mathcal{B}}d\tilde{W}_s~~~\text{in}~L^{\infty}([0,T];\mathbb{H}),
%~~~\tilde{\mathbb{P}}\text{-a.s.},
%\end{eqnarray*}
%here choosing a subsequence if necessary, and the notation ``$\rightharpoonup$'' stands for the weak convergence.

%Set
%$$X_t:=x+\int_0^t\bar{\mathcal{A}}(s)ds+\int_0^t\bar{\mathcal{B}}(s)d\tilde{W}_s,~~t\in[0,T].$$
%It is clear that
%$$X=\hat{X}=\tilde{X},~~\tilde{\mathbb{P}}\otimes dt\text{-a.s.},$$
%which follows from \cite{LR1} and the uniqueness of the limits.
In what follows, without loss of generality, we replace $(\tilde{\Omega},\tilde{\mathscr{F}},\tilde{\mathbb{P}})$, $(\tilde{X}^{(n)},\tilde{W}^{(n)})$, and $(\tilde{X},\tilde{W})$ by $(\Omega,\mathscr{F},\mathbb{P})$, $(X^{(n)},W^{(n)})$, and $(X,W),$ respectively.

\vspace{2mm}
In order to characterize the limit of $\mathcal{A}^{(n)}(\cdot)$, we recall the following lemma concerning the pseudo-monotonicity (cf. \cite{LR13}).
\begin{lemma}\label{lemps}
Assume that the embedding $\mathbb{V}\subset\mathbb{H}$ is compact, $(\mathbf{A_1})$ and $(\mathbf{A_2})$ hold. Then $\mathcal{A}(t,\cdot)$ is pseudo-monotone from $\mathbb{V}$ to $\mathbb{V}^*$ for a.e.~$t\in[0,T]$.
\end{lemma}

The next important step is to show $\tilde{\mathcal{A}}(\cdot)= \mathcal{A}(\cdot,X_{\cdot})$. To this end, we introduce the following lemma.

\begin{lemma}\label{lem2}
Assume that the embedding $\mathbb{V}\subset\mathbb{H}$ is compact, $(\mathbf{A_1})$-$(\mathbf{A_4})$ hold. If
\begin{eqnarray*}
\!\!\!\!\!\!\!\!&& (\text{i})~~ X^{(n)}\rightharpoonup X~~~\text{in}~L^{\alpha}([0,T];\mathbb{V}),~~~\mathbb{P}\text{-a.s.},
\nonumber \\
\!\!\!\!\!\!\!\!&&(\text{ii})~~ \mathcal{A}(\cdot,X^{(n)}_{\cdot})\rightharpoonup \tilde{\mathcal{A}}~~~\text{in}~L^{\frac{\alpha}{\alpha-1}}([0,T];\mathbb{V}^*),~~~\mathbb{P}\text{-a.s.},
\nonumber \\
\!\!\!\!\!\!\!\!&&(\text{iii})~~\liminf_{n\to\infty}\int_0^T{}_{\mathbb{V}^*}\langle \mathcal{A}(t,X^{(n)}_{t}),X^{(n)}_{t}\rangle_{\mathbb{V}}dt\geq \int_0^T{}_{\mathbb{V}^*}\langle \tilde{\mathcal{A}}(t),X_t\rangle_{\mathbb{V}}dt,
~~~\mathbb{P}\text{-a.s.},
\end{eqnarray*}
then $\tilde{\mathcal{A}}(\cdot)= \mathcal{A}(\cdot,X_{\cdot})$, $\mathbb{P}\otimes dt$\text{-a.e.}.
\end{lemma}

\begin{proof}
First, due to $(\mathbf{A_3})$ and $(\mathbf{A_4})$, we can obtain that there exists $\delta_0>0$ such that
\begin{eqnarray}\label{e15}
\!\!\!\!\!\!\!\!&& {}_{\mathbb{V}^*}\langle \mathcal{A}(t,X^{(n)}_t),X^{(n)}_t-X_t\rangle_{\mathbb{V}}
\nonumber \\
\leq\!\!\!\!\!\!\!\!&&-\delta\|X^{(n)}_t\|_{\mathbb{V}}^{\alpha}+C+g(\|X^{(n)}_t\|_{\mathbb{H}}^2)+\|\mathcal{A}(t,X^{(n)}_t)\|_{\mathbb{V}^*}\|X_t\|_{\mathbb{V}}
\nonumber \\
\leq\!\!\!\!\!\!\!\!&&-\delta\|X^{(n)}_t\|_{\mathbb{V}}^{\alpha}+C+g(\|X^{(n)}_t\|_{\mathbb{H}}^2)
\nonumber \\
\!\!\!\!\!\!\!\!&&+C\big(1+\|X^{(n)}_t\|_{\mathbb{V}}^{\alpha}\big)^{\frac{\alpha-1}{\alpha}}\big(1+\|X^{(n)}_t\|_{\mathbb{H}}^{\beta}\big)^{\frac{\alpha-1}{\alpha}}\|X_t\|_{\mathbb{V}}
\nonumber \\
\leq\!\!\!\!\!\!\!\!&&-\delta_0\|X^{(n)}_t\|_{\mathbb{V}}^{\alpha}+C+g(\|X^{(n)}_t\|_{\mathbb{H}}^2)
+C\big(1+\|X^{(n)}_t\|_{\mathbb{H}}^{\beta}\big)^{\alpha-1}\|X_t\|_{\mathbb{V}}^{\alpha}.
\end{eqnarray}
For convenience, we denote
\begin{eqnarray*}
\phi^{(n)}(t,\omega):=\!\!\!\!\!\!\!\!&& {}_{\mathbb{V}^*}\langle \mathcal{A}(t,X^{(n)}_t(\omega)),X^{(n)}_t(\omega)-X_t(\omega)\rangle_{\mathbb{V}},
\nonumber \\
F^{(n)}(t,\omega):=\!\!\!\!\!\!\!\!&& C+g(\|X^{(n)}_t(\omega)\|_{\mathbb{H}}^2)
+C\big(1+\|X^{(n)}_t(\omega)\|_{\mathbb{H}}^{\beta}\big)^{\alpha-1}\|X_t(\omega)\|_{\mathbb{V}}^{\alpha}.
\end{eqnarray*}
Then (\ref{e15}) reads as
\begin{equation}\label{e16}
\phi^{(n)}(t,\omega)\leq-\delta_0\|X^{(n)}_t(\omega)\|_{\mathbb{V}}^{\alpha}+F^{(n)}(t,\omega).
\end{equation}
The proof of this lemma is divided into the following four steps.

\vspace{1mm}
\noindent\textbf{Step 1.}  In this step, we prove that for a.e.~$(t,\omega)$,
\begin{equation}\label{e17}
\limsup_{n\to \infty}\phi^{(n)}(t,\omega)\leq 0.
\end{equation}

According to  (\ref{e21}),
 there exists a $\mathbb{P}\otimes dt$-null set $\mathcal{N}$ such that for any $(t,\omega)\in (\Omega\times[0,T])\backslash \mathcal{N}$, along a subsequence still denoted by $\{n\}$ we have
\begin{equation}\label{e18}
\|X^{(n)}_t(\omega)-X_t(\omega)\|_{\mathbb{H}}\to 0~\text{as}~n\to\infty.
\end{equation}
 From now on, we fix $(t,\omega)\in (\Omega\times[0,T])\backslash \mathcal{N}$ and suppose that
\begin{equation*}
\limsup_{n\to \infty}\phi^{(n)}(t,\omega)>0.
\end{equation*}
Thus we can take a subsequence $\{n_k\}_{k\in\mathbb{N}}$ such that
\begin{equation}\label{e25}
\lim_{k\to \infty}\phi^{(n_k)}(t,\omega)>0.
\end{equation}

It follows from  (\ref{e16}) and (\ref{e18}) that
\begin{equation}\label{e19}
\Big\{\|X^{(n_k)}_t(\omega)\|_{\mathbb{V}}^{\alpha}\Big\}_{k\in\mathbb{N}}~~\text{is bounded}.
\end{equation}
Therefore, there exists an element $z\in\mathbb{V}$ such that
$$X^{(n_k)}_t(\omega)\rightharpoonup z~~\text{in}~\mathbb{V}~\text{as}~k\to\infty.$$
From the convergence (\ref{e18}), it is clear that $z=X_t(\omega)$ and
$$X^{(n_k)}_t(\omega)\rightharpoonup X_t(\omega)~~\text{in}~\mathbb{V}~\text{as}~k\to\infty.$$
Using the fact that $\mathcal{A}(t,\cdot)$ is pseudo-monotone (cf.~Lemma \ref{lemps}), we deduce that
\begin{equation}\label{e20}
\limsup_{k\to\infty}\phi^{(n_k)}(t,\omega)\leq 0,
  \end{equation}
which contradicts to  (\ref{e25}).
  Hence,  (\ref{e17}) holds.

\vspace{1mm}
\noindent\textbf{Step 2.} In this step, we prove that  along a subsequence $\{n_k\}_{k\in\mathbb{N}}$ for a.e.~$(t,\omega)$,
$$\lim_{k\to \infty} \phi^{(n_k)}(t,\omega)=0.$$

First, by the conditions (ii)-(iii) in this lemma, the control (\ref{apri2}), (\ref{e18}),  and Fatou's lemma, we have for a.s.~$\omega$,
\begin{eqnarray*}
0 \leq\!\!\!\!\!\!\!\!&& \liminf_{n\to\infty}\int_0^T\phi^{(n)}(t,\omega)dt
\leq \limsup_{n\to\infty}\int_0^T\phi^{(n)}(t,\omega)dt
\nonumber \\
\leq\!\!\!\!\!\!\!\!&& \int_0^T\limsup_{n\to\infty}\phi^{(n)}(t,\omega)dt\leq 0.
\end{eqnarray*}
Hence
\begin{equation}\label{e23}
\lim_{n\to\infty}\int_0^T\phi^{(n)}(t,\omega)dt=0.
\end{equation}
Then combining (\ref{e17}) and (\ref{e23}) and applying the dominated convergence theorem, it follows that
\begin{equation}\label{e24}
\lim_{n\to\infty}\int_0^T\phi^{(n)}_+(t,\omega)dt=0,
\end{equation}
where $\phi^{(n)}_+(t,\omega):=\text{max}\{\phi^{(n)}(t,\omega),0\}$.

By (\ref{e23})-(\ref{e24}) and the fact that $|\phi^{(n)}|=2\phi^{(n)}_{+}-\phi^{(n)}$, we have
\begin{equation*}
\lim_{n\to\infty}\int_0^T|\phi^{(n)}(t,\omega)|dt=0,
\end{equation*}
which implies that the claim follows.

\vspace{1mm}
\noindent\textbf{Step 3.} In this step, we prove that $\tilde{\mathcal{A}}(\cdot)= \mathcal{A}(\cdot,X_{\cdot})$, $\mathbb{P}\otimes dt$\text{-a.e.}.

\vspace{1mm}
Combining the claim of \textbf{Step 2}, (\ref{e16}) and (\ref{e18}),  we can get  (\ref{e19}) holds. Thus,
$$X^{(n_k)}_t(\omega)\rightharpoonup X_t(\omega)~~\text{in}~\mathbb{V}~\text{as}~k\to\infty,$$
which combining with  the pseudo-monotonicity of  $\mathcal{A}(t,\cdot)$ implies that
$$\mathcal{A}(t,X^{(n_k)}_{t}(\omega))\rightharpoonup \mathcal{A}(t,X_{t}(\omega))~~\text{in}~\mathbb{V}^*~\text{as}~k\to\infty,$$
Consequently, by the condition (ii) in the lemma and the uniqueness of the limit, it is clear that $\tilde{\mathcal{A}}(\cdot)= \mathcal{A}(\cdot,X_{\cdot})$, $\mathbb{P}\otimes dt$\text{-a.e.}.
\end{proof}

Now we have all the ingredients to prove the existence of weak solutions to (\ref{eqSPDE}).\vspace{2mm}\\
\textbf{Proof of  existence of weak solutions.} We aim to show that $(X,W)$ obtained above is a weak solution to Eq.~(\ref{eqSPDE}) in the sense of Definition \ref{dew}. Combining (\ref{es80}), (\ref{con1}), (\ref{es1}), (\ref{es2}) and Lemma \ref{lem2}, we need to show that the condition (iii) in Lemma \ref{lem2} holds and $X\in\mathbb{C}_T(\mathbb{H}),~\mathbb{P}\text{-a.s.}$.  The proof is separated into the following two steps.

%Let us denote
%$$\tau(M,n):=\tau_M\wedge\tau_M^{(n)},$$
%where the stopping time $\tau_M^{(n)}$ is defined in  (\ref{stop1}).
%
\noindent\textbf{Step 1.} In this step, we prove that the condition (iii) in Lemma \ref{lem2} holds. First, in view of (\ref{con1}) it implies that
\begin{equation*}
\|\tilde{\mathcal{A}}(\cdot)\|_{L^{\frac{\alpha}{\alpha-1}}([0,T];\mathbb{V}^*)}<\infty,~\mathbb{P}\text{-a.s.}.
\end{equation*}
Then we denote a stopping time
\begin{equation*}
\tau^{\mathcal{A}}_M:=\inf\Bigg\{t\in[0,T]:\int_0^t\|\tilde{\mathcal{A}}(s)\|_{\mathbb{V}^*}^{\frac{\alpha}{\alpha-1}}ds\geq M\Bigg\}\wedge T,~~M>0.
\end{equation*}
It is easy to see that
$$\lim_{M\to\infty}\tau^{\mathcal{A}}_M=T,~\mathbb{P}\text{-a.s.}.$$

On the other hand, the estimate (\ref{apri2}) yields
\begin{equation}\label{con11}
\sup_{t\in[0,T]}\|X_t\|_{\mathbb{H}}+\int_0^T\|X_t\|_{\mathbb{V}}^{\alpha}dt<\infty,~\mathbb{P}\text{-a.s.}.
\end{equation}
By (\ref{es80}), we know that
\begin{equation}\label{con12}
X\in\mathbb{C}_T(\mathbb{V}^*),~\mathbb{P}\text{-a.s.}.
\end{equation}
Since $\mathbb{H}\subset\mathbb{V}^*$ is dense, in view of (\ref{con11}) and (\ref{con12}) we deduce that $X_{\cdot}$ is weakly continuous in $\mathbb{H}$, so that $\|X_{\cdot}\|_{\mathbb{H}}$ is lower semicontinuous.
Therefore, we set
\begin{equation*}
\tau_M^X:=\inf\Bigg\{t\in[0,T]:\|X_t\|_{\mathbb{H}}+\int_0^t \|X_s\|_{\mathbb{V}}^{\alpha}ds\geq M\Bigg\}\wedge T,~~M>0,
\end{equation*}
 which is a stopping time and
$$\lim_{M\to\infty}\tau_M^X=T,~\mathbb{P}\text{-a.s.}.$$
Let us denote
$$\hat{\tau}_M:=\tau^{\mathcal{A}}_M\wedge\tau_M^X.$$
Recall Eq.~(\ref{es1}), it is easy to see that
\begin{equation*}
X_{t\wedge\hat{\tau}_M}=x+\int_0^t\mathbf{1}_{\{s\leq\hat{\tau}_M \}}\tilde{\mathcal{A}}(s)ds+\int_0^t\mathbf{1}_{\{s\leq\hat{\tau}_M \}}\mathcal{B}(s,X_s)dW_s,~t\in[0,T].
\end{equation*}
Let us denote
$$ Y(t):=\mathbf{1}_{\{t\leq\hat{\tau}_M \}}\tilde{\mathcal{A}}(t),~Z(t):=\mathbf{1}_{\{t\leq\hat{\tau}_M \}}\mathcal{B}(t,X_t).$$
Then it is clear that
$$X_{\cdot\wedge\hat{\tau}_M}\mathbf{1}_{\{\cdot\leq\hat{\tau}_M \}}\in L^{\alpha}([0,T]\times \Omega;\mathbb{V}),~Y(\cdot)\in L^{\frac{\alpha}{\alpha-1}}([0,T]\times \Omega;\mathbb{V}^*)$$
and
$$Z(\cdot)\in L^{2}([0,T]\times \Omega;\mathbb{H}).$$
Therefore, according to Proposition 4.2 in \cite{GC1}, we can apply It\^{o}'s formula  and deduce that
\begin{eqnarray}\label{ito1}
\|X_t\|_{\mathbb{H}}^2
=\!\!\!\!\!\!\!\!&&\|x\|_{\mathbb{H}}^2+\int_0^t\Big(2{}_{\mathbb{V}^*}\langle \tilde{\mathcal{A}}(s),X_s\rangle_{\mathbb{V}}
+\|\mathcal{B}(s,X_s)\|_{L_2(U,\mathbb{H})}^2\Big)ds
\nonumber \\
\!\!\!\!\!\!\!\!&&+2\int_0^t\langle \mathcal{B}(s,X_s)dW_s,X_s\rangle_{\mathbb{H}}~\text{on}~\{t\leq \hat{\tau}_M\}.
\end{eqnarray}
Note that $\lim_{M\to\infty}\hat{\tau}_M=T$, $\mathbb{P}$-a.s., which implies that the equality (\ref{ito1}) holds for all $t\in[0,T]$.

Next, applying It\^{o}'s formula for $\|X^{(n)}_t\|_{\mathbb{H}}^2$, we have
\begin{eqnarray}\label{ito2}
\|X^{(n)}_t\|_{\mathbb{H}}^2
=\!\!\!\!\!\!\!\!&&\|x^{(n)}\|_{\mathbb{H}}^2+\int_0^t\Big(2{}_{\mathbb{V}^*}\langle \mathcal{A}(s,X^{(n)}_s),X^{(n)}_s\rangle_{\mathbb{V}}
\nonumber \\
\!\!\!\!\!\!\!\!&&
+\|\mathcal{P}_n\mathcal{B}(s,X^{(n)}_s)\tilde{\mathcal{P}}_n\|_{L_2(U,\mathbb{H})}^2\Big)ds
\nonumber \\
\!\!\!\!\!\!\!\!&&+2\int_0^t\langle \mathcal{P}_n\mathcal{B}(s,X^{(n)}_s)\tilde{\mathcal{P}}_ndW^{(n)}_s,X^{(n)}_s\rangle_{\mathbb{H}}.
\end{eqnarray}

The following lemma concerns the convergence of the martingale term in formula (\ref{ito2}).
\begin{lemma}
Along a subsequence still denoted by $\{n\}$, we have
\begin{equation}\label{con8}
\int_0^{\cdot} \langle\mathcal{P}_n\mathcal{B}(s,X^{(n)}_s)\tilde{\mathcal{P}}_ndW^{(n)}_s,X^{(n)}_s\rangle_{\mathbb{H}}\to\int_0^{\cdot} \langle\mathcal{B}(s,X_s)dW_s,X_s\rangle_{\mathbb{H}}
\end{equation}
in $L^{\infty}([0,T];\mathbb{R})$, $\mathbb{P}$-a.s.,  as $n\to\infty$.
\end{lemma}
\begin{proof}
The proof is exactly similar to that of Lemma \ref{lem4}, we omit the details.
\end{proof}

Note that by the lower semicontinuity of norm $\|\cdot\|_{\mathbb{H}}$ in $\mathbb{V}^*$ and the convergence (\ref{es80}), we have
\begin{equation}\label{con9}
\|X_t\|_{\mathbb{H}}^2\leq \liminf_{n\to\infty}\|X^{(n)}_t\|_{\mathbb{H}}^2,~\mathbb{P}\text{-a.s.}.
\end{equation}
Finally, combining (\ref{es80}), (\ref{con3}), (\ref{ito1})-(\ref{con9}), it follows that
$$\liminf_{n\to\infty}\int_0^T{}_{\mathbb{V}^*}\langle \mathcal{A}(t,X^{(n)}_{t}),X^{(n)}_{t}\rangle_{\mathbb{V}}dt\geq \int_0^T{}_{\mathbb{V}^*}\langle \tilde{\mathcal{A}}(t),X_t\rangle_{\mathbb{V}}dt,
~\mathbb{P}\text{-a.s.},$$
namely, the condition (iii) in Lemma \ref{lem2} holds.

\vspace{2mm}
\noindent\textbf{Step 2.} In this step, we prove $X\in\mathbb{C}_T(\mathbb{H})$, $\mathbb{P}$-a.s.. In Step 1, we have shown that  $X_{\cdot}$ is weakly continuous in $\mathbb{H}$.
Thus it suffices to prove
$t\mapsto\|X_{t}\|_{\mathbb{H}}$ is  continuous on $[0,T]$.  Based on Step 1, we deduce that
\begin{eqnarray}\label{ito3}
\|X_t\|_{\mathbb{H}}^2
=\!\!\!\!\!\!\!\!&&\|x\|_{\mathbb{H}}^2+\int_0^t\Big(2{}_{\mathbb{V}^*}\langle \mathcal{A}(s,X_s),X_s\rangle_{\mathbb{V}}
+\|\mathcal{B}(s,X_s)\|_{L_2(U,\mathbb{H})}^2\Big)ds
\nonumber \\
\!\!\!\!\!\!\!\!&&+2\int_0^t\langle \mathcal{B}(s,X_s)dW_s,X_s\rangle_{\mathbb{H}}.
\end{eqnarray}
Since the right-hand side of (\ref{ito3}) is continuous on $[0,T]$, so must be its left-hand side.

We complete the proof of the  existence of weak solutions.  \hspace{\fill}$\Box$

\subsection{Proof of existence and uniqueness}\label{sec2.4}

We first prove the  pathwise uniqueness of solutions to (\ref{eqSPDE}). \vspace{2mm}\\
\textbf{Proof of pathwise uniqueness.} Let $X,Y$ be two solutions of (\ref{eqSPDE}) with same initial value $x\in \mathbb{H}$. Then, the difference process $Z:=X-Y$ solves the following equation
 $$Z_t=\int_0^t\big(\mathcal{A}(s,X_s)-\mathcal{A}(s,Y_s)\big)ds+\int_0^t\big(\mathcal{B}(s,X_s)-\mathcal{B}(s,Y_s)\big)dW_s,~t\in[0,T].$$

Set
\begin{eqnarray*}
\!\!\!\!\!\!\!\!&&\tau_M^X:=\inf\Bigg\{t\in[0,T]:\|X_t\|_{\mathbb{H}}+\int_0^t \|X_s\|_{\mathbb{V}}^{\alpha}ds\geq M\Bigg\}\wedge T,~~M>0,
\nonumber \\
\!\!\!\!\!\!\!\!&&\tau_M^Y:=\inf\Bigg\{t\in[0,T]:\|Y_t\|_{\mathbb{H}}+\int_0^t \|Y_s\|_{\mathbb{V}}^{\alpha}ds\geq M\Bigg\}\wedge T,~~M>0.
\end{eqnarray*}
It is clear that $\tau_M^X,\tau_M^Y$ are stopping times and $\lim_{M\to\infty}\tau_M^X=T, \lim_{M\to\infty}\tau_M^Y=T,~\mathbb{P}\text{-a.s.}$.

\vspace{1mm}
Let $\tau_M:=\tau_M^X\wedge\tau_M^Y$. Applying It\^{o}'s formula to $\|Z_t\|_{\mathbb{H}}^2$, which follows from the same argument as (\ref{ito1}), and by $(\mathbf{A_2})$ we derive
\begin{eqnarray}\label{e100}
\|Z_t\|_{\mathbb{H}}^2
=\!\!\!\!\!\!\!\!&&\int_0^t\Big(2{}_{\mathbb{V}^*}\langle \mathcal{A}(s,X_s)-\mathcal{A}(s,Y_s),Z_s\rangle_{\mathbb{V}}
\nonumber \\
\!\!\!\!\!\!\!\!&&
+\|\mathcal{B}(s,X_s)-\mathcal{B}(s,Y_s)\|_{L_2(U,\mathbb{H})}^2\Big)ds
\nonumber \\
\!\!\!\!\!\!\!\!&&+2\int_0^t\langle \big(\mathcal{B}(s,X_s)-\mathcal{B}(s,Y_s)\big)dW_s,Z_s\rangle_{\mathbb{H}}
\nonumber \\
\leq\!\!\!\!\!\!\!\!&&C\int_0^t(1+\rho(X_s)+\eta(Y_s))\|Z_s\|_{\mathbb{H}}^2ds
\nonumber \\
\!\!\!\!\!\!\!\!&&+2\int_0^t\langle \big(\mathcal{B}(s,X_s)-\mathcal{B}(s,Y_s)\big)dW_s,Z_s\rangle_{\mathbb{H}}.
\end{eqnarray}
Then, we have
\begin{equation*}
\mathbb{E}\|Z_{t\wedge\tau_M}\|_{\mathbb{H}}^2
\leq C\mathbb{E}\int_0^{t\wedge\tau_M}(1+\rho(X_s)+\eta(Y_s))\|Z_s\|_{\mathbb{H}}^2ds.
\end{equation*}
Applying stochastic Gronwall's lemma (cf.~\cite[Lemma 5.3]{GZ1}), it leads to
\begin{equation*}
\mathbb{E}\|Z_{t\wedge\tau_M}\|_{\mathbb{H}}^2
\leq 0,~t\in[0,T],
\end{equation*}
which implies that
\begin{equation}\label{e14}
\mathbb{E}\|Z_{t}\|_{\mathbb{H}}^2\leq \liminf_{M\to\infty}\mathbb{E}\|Z_{t\wedge\tau_M}\|_{\mathbb{H}}^2\leq 0,~t\in[0,T].
\end{equation}
Therefore, the pathwise uniqueness follows from (\ref{e14}) and the pathwise continuity in $\mathbb{H}$.

\hspace{\fill}$\Box$

\vspace{2mm}
Now we have all the ingredients to verify  Theorem \ref{th1}. \vspace{2mm}\\
\textbf{Proof of Theorem \ref{th1}.} By combining the results in Subsection \ref{sec2.3} with the pathwise uniqueness established above, the unique existence of strong solutions to (\ref{eqSPDE}) is a direct consequence of the infinite-dimensional version of the Yamada-Watanabe theorem, where the estimates (\ref{apri0}) and (\ref{apri6}) follow from (\ref{apri5}) and (\ref{apri2}), respectively.

Finally, the Markov property follows from standard arguments, such as those in \cite[Proposition 4.3.3]{LR1}, based on the established unique existence of solutions. This completes the proof. \hspace{\fill}$\Box$

\subsection{Proof of Feller property}\label{sec2.5}

Set the stopping time
\begin{eqnarray*}
\tau_M^{n}:=\!\!\!\!\!\!\!\!&&\inf\Bigg\{t\in[0,T]:\|X_t(x_n)\|_{\mathbb{H}}+\|X_t(x)\|_{\mathbb{H}}
\nonumber \\
\!\!\!\!\!\!\!\!&&
+\int_0^t \|X_s(x_n)\|_{\mathbb{V}}^{\alpha}ds+\int_0^t \|X_s(x)\|_{\mathbb{V}}^{\alpha}ds\geq M\Bigg\}\wedge T,~~M>0.
\end{eqnarray*}
According to the estimate (\ref{apri0}) and the convergence of $x_n$, we can deduce that
\begin{equation}\label{con4}
\lim_{M\to\infty}\sup_{n\in\mathbb{N}}\mathbb{P}(\tau_M^{n}<T)=0.
\end{equation}

In the following, we first prove the continuous dependence on initial data in probability, and then we derive the Feller property of the transition semigroup.
\vspace{2mm}\\
\textbf{Proof of Theorem \ref{th2}.} In view of the proof of (\ref{e100}), by B-D-G's inequality we have
\begin{eqnarray*}
\!\!\!\!\!\!\!\!&&\mathbb{E}\Big[\sup_{t\in[0,T\wedge\tau_M^{n}]}\|X_{t}(x_n)-X_{t}(x)\|_{\mathbb{H}}^2\Big]
\nonumber \\
\leq\!\!\!\!\!\!\!\!&&
\|x_n-x\|_{\mathbb{H}}^2+C\mathbb{E}\int_0^{T\wedge\tau_M^{n}}(1+\rho(X_s(x_n))+\eta(X_s(x)))\|X_s(x_n)-X_s(x)\|_{\mathbb{H}}^2ds
\nonumber \\
\!\!\!\!\!\!\!\!&&+C\mathbb{E}\Bigg(\int_0^{T\wedge\tau_M^{n}}\|X_s(x_n)-X_s(x)\|_{\mathbb{H}}^2\|\mathcal{B}(s,X_{s}(x_n))-\mathcal{B}(s,X_{s}(x))\|_{L_2(U,\mathbb{H})}^2ds\Bigg)^{\frac{1}{2}}
\nonumber \\
\leq\!\!\!\!\!\!\!\!&&
\|x_n-x\|_{\mathbb{H}}^2+C\mathbb{E}\int_0^{T\wedge\tau_M^{n}}(1+\rho(X_s(x_n))+\eta(X_s(x)))\|X_s(x_n)-X_s(x)\|_{\mathbb{H}}^2ds
\nonumber \\
\!\!\!\!\!\!\!\!&&+\frac{1}{2}\mathbb{E}\Big[\sup_{t\in[0,T\wedge\tau_M^{n}]}\|X_{t}(x_n)-X_{t}(x)\|_{\mathbb{H}}^2\Big]
\nonumber \\
\!\!\!\!\!\!\!\!&&+C\mathbb{E}\int_0^{T\wedge\tau_M^{n}}\|\mathcal{B}(s,X_{s}(x_n))-\mathcal{B}(s,X_{s}(x))\|_{L_2(U,\mathbb{H})}^2ds
\nonumber \\
\leq\!\!\!\!\!\!\!\!&&
\|x_n-x\|_{\mathbb{H}}^2+C\mathbb{E}\int_0^{T\wedge\tau_M^{n}}(1+\rho(X_s(x_n))+\eta(X_s(x)))\|X_s(x_n)-X_s(x)\|_{\mathbb{H}}^2ds
\nonumber \\
\!\!\!\!\!\!\!\!&&+\frac{1}{2}\mathbb{E}\Big[\sup_{t\in[0,T\wedge\tau_M^{n}]}\|X_{t}(x_n)-X_{t}(x)\|_{\mathbb{H}}^2\Big],
\end{eqnarray*}
where we used the assumption (\ref{conb4}) in the last step.

\vspace{1mm}
Applying stochastic Gronwall's lemma, we have
\begin{equation*}
\mathbb{E}\Big[\sup_{t\in[0,T\wedge\tau_M^{n}]}\|X_{t}(x_n)-X_{t}(x)\|_{\mathbb{H}}^2\Big]
\leq
C_{M}\|x_n-x\|_{\mathbb{H}}^2.
\end{equation*}
Consequently, for any $\varepsilon>0$,
\begin{eqnarray*}
\!\!\!\!\!\!\!\!&&\mathbb{P}\Big(\sup_{t\in[0,T]}\|X_{t}(x_n)-X_{t}(x)\|_{\mathbb{H}}>\varepsilon\Big)
\nonumber \\
\leq\!\!\!\!\!\!\!\!&&\mathbb{P}\Big(\sup_{t\in[0,T]}\|X_{t}(x_n)-X_{t}(x)\|_{\mathbb{H}}>\varepsilon,T\leq\tau_M^{n}\Big)+\mathbb{P}(\tau_M^{n}<T)
\nonumber \\
\leq\!\!\!\!\!\!\!\!&&\frac{1}{\varepsilon^2}\mathbb{E}\Big[\sup_{t\in[0,T\wedge\tau_M^{n}]}\|X_{t}(x_n)-X_{t}(x)\|_{\mathbb{H}}^2\Big]+\mathbb{P}(\tau_M^{n}<T)
\nonumber \\
\leq\!\!\!\!\!\!\!\!&&\frac{C_{M}}{\varepsilon^2}\|x_n-x\|_{\mathbb{H}}^2+\mathbb{P}(\tau_M^{n}<T).
\end{eqnarray*}
Taking (\ref{con4}) into account, we conclude that (\ref{apri3}) holds.

\vspace{2mm}
Now we prove the Feller property of the transition semigroup. For any $t\geq 0$ and $\varphi\in C_b(\mathbb{H})$, we show
$\mathcal{T}_t\varphi\in C_b(\mathbb{H})$, i.e.,
\begin{equation}\label{feller}
\mathbb{E}\varphi(X_t(x_n))\to \mathbb{E}\varphi(X_t(x)),~ \text{if}~ x_n\to x~\text{in}~ \mathbb{H}.
\end{equation}
Note that by (\ref{apri3}),  we have
$$X_{t}(x_n)\to X_{t}(x)\quad\text{in probability as}~n\to\infty.$$  Since $ \varphi\in C_b(\mathbb{H})$, we derive
$$\varphi(X_{t}(x_n))\to \varphi(X_{t}(x))\quad \text{in probability as} ~n\to\infty.$$ Consequently,  (\ref{feller}) follows from the Lebesgue dominated convergence theorem.
We complete the proof of Theorem \ref{th2}.
\hspace{\fill}$\Box$

\subsection{Proof of finite-time extinction }\label{sec4.6}

In this subsection, we intend to prove Theorem \ref{th3}.
We first present the following lemma.

\begin{lemma}
 Suppose that the assumptions in Theorem \ref{th3} hold. Then there is a constant $C>0$ such that for any $t\geq 0$,
 \begin{equation}\label{e116}
\mathbb{E}(1+\|X_t\|_{\mathbb{H}}^2)^{1-\frac{\alpha}{2}}
\leq e^{Ct}(1+\|x\|_{\mathbb{H}}^2)^{1-\frac{\alpha}{2}}.
\end{equation}

\end{lemma}

\begin{proof}
Recall that for any $t\geq 0$,
\begin{eqnarray*}
\|X_t\|_{\mathbb{H}}^2
=\!\!\!\!\!\!\!\!&&\|x\|_{\mathbb{H}}^2+\int_0^t\Big(2{}_{\mathbb{V}^*}\langle \mathcal{A}(s,X_s),X_s\rangle_{\mathbb{V}}
+\|\mathcal{B}(s,X_s)\|_{L_2(U,\mathbb{H})}^2\Big)ds
\nonumber \\
\!\!\!\!\!\!\!\!&&+2\int_0^t\langle X_s,\mathcal{B}(s,X_s)dW_s\rangle_{\mathbb{H}}.
\end{eqnarray*}
Applying It\^{o}'s formula for the Lyapunov function $V(r):=(1+r)^{1-\frac{\alpha}{2}}$, by $(\mathbf{A_3^*})$ for any $t\geq 0$,
\begin{eqnarray}\label{e115}
\!\!\!\!\!\!\!\!&&(1+\|X_t\|_{\mathbb{H}}^2)^{1-\frac{\alpha}{2}}
\nonumber \\
\leq\!\!\!\!\!\!\!\!&&(1+\|x\|_{\mathbb{H}}^2)^{1-\frac{\alpha}{2}}-\delta(1-\frac{\alpha}{2})\int_0^t \frac{\|X_s\|_{\mathbb{V}}^{\alpha}}{(1+\|X_s\|_{\mathbb{H}}^2)^{\frac{\alpha}{2}}}ds
\nonumber \\
\!\!\!\!\!\!\!\!&&
+(1-\frac{\alpha}{2})\int_0^t \frac{g(\|X_s\|_{\mathbb{H}}^2)+\|\mathcal{B}(s,X_s)\|_{L_2(U,\mathbb{H})}^2}{(1+\|X_s\|_{\mathbb{H}}^2)^{\frac{\alpha}{2}}}ds
\nonumber \\
\!\!\!\!\!\!\!\!&&-\frac{\alpha}{2}(1-\frac{\alpha}{2})\int_0^t \frac{2\|\mathcal{B}(s,X_s)^*X_s\|_{U}^2}{(1+\|X_s\|_{\mathbb{H}}^2)^{\frac{\alpha}{2}+1}}ds+\mathcal{M}_t^0
\nonumber \\
=\!\!\!\!\!\!\!\!&&(1+\|x\|_{\mathbb{H}}^2)^{1-\frac{\alpha}{2}}-\delta(1-\frac{\alpha}{2})\int_0^t \frac{\|X_s\|_{\mathbb{V}}^{\alpha}}{(1+\|X_s\|_{\mathbb{H}}^2)^{\frac{\alpha}{2}}}ds+\mathcal{M}_{t}^0
\nonumber \\
\!\!\!\!\!\!\!\!&&
~~~+(1-\frac{\alpha}{2})\int_0^t \frac{\big(g(\|X_s\|_{\mathbb{H}}^2)+\|\mathcal{B}(s,X_s)\|_{L_2(U,\mathbb{H})}^2\big)(1+\|X_s\|_{\mathbb{H}}^2)}{(1+\|X_s\|_{\mathbb{H}}^2)^{\frac{\alpha}{2}+1}}ds
\nonumber \\
\!\!\!\!\!\!\!\!&&
-(1-\frac{\alpha}{2})\int_0^t \frac{\alpha\|\mathcal{B}(s,X_s)^*X_s\|_{U}^2}{(1+\|X_s\|_{\mathbb{H}}^2)^{\frac{\alpha}{2}+1}}ds,
\end{eqnarray}
where
$$\mathcal{M}_{t}^0:=2(1-\frac{\alpha}{2})\int_0^t \frac{\langle X_s,\mathcal{B}(s,X_s)dW_s\rangle_{\mathbb{H}}}{(1+\|X_s\|_{\mathbb{H}}^2)^{\frac{\alpha}{2}}}.$$
Taking into account the assumption (\ref{conb1}) and (\ref{e115}),  we can get
\begin{equation*}
(1+\|X_t\|_{\mathbb{H}}^2)^{1-\frac{\alpha}{2}}
\leq(1+\|x\|_{\mathbb{H}}^2)^{1-\frac{\alpha}{2}}
+C\int_0^t (1+\|X_s\|_{\mathbb{H}}^2)^{1-\frac{\alpha}{2}}ds+\mathcal{M}_{t}^0.
\end{equation*}
By the standard localization argument and the estimate (\ref{apri0}), Gronwall' lemma implies
\begin{equation*}
\mathbb{E}(1+\|X_t\|_{\mathbb{H}}^2)^{1-\frac{\alpha}{2}}
\leq e^{Ct}(1+\|x\|_{\mathbb{H}}^2)^{1-\frac{\alpha}{2}}.
\end{equation*}
We complete the proof.
\end{proof}

The following lemma plays an important role for proving (\ref{fte1}) and (\ref{fte2}).
\begin{lemma}
 Suppose that the assumptions in Theorem \ref{th3} hold. Then there is $c^*>0$ such that for any $0\leq r<t$,
\begin{eqnarray}\label{e113}
\!\!\!\!\!\!\!\!&&\|X_t\|_{\mathbb{H}}^{2-\alpha}+\delta(c^*)^{\alpha}(1-\frac{\alpha}{2})\int_r^t \mathbf{1}_{\{\|X_s\|_\mathbb{H}>0\}}ds
\nonumber \\
\leq\!\!\!\!\!\!\!\!&&
\|X_r\|_{\mathbb{H}}^{2-\alpha}+2(1-\frac{\alpha}{2})\int_r^t \mathbf{1}_{\{\|X_s\|_\mathbb{H}>0\}}\frac{\langle X_s,\mathcal{B}(s,X_s)dW_s\rangle_{\mathbb{H}}}{\|X_s\|_{\mathbb{H}}^\alpha}.
\end{eqnarray}
\end{lemma}

%

% \vspace{3mm}
%\textbf{Proof of Theorem \ref{th3}:} From now on, we focus on proving (\ref{fte1}) and (\ref{fte2}).

\begin{proof}
%Recall that  for any $t\geq 0$,
%\begin{eqnarray*}
%\|X_t\|_{\mathbb{H}}^2
%=\!\!\!\!\!\!\!\!&&\|x\|_{\mathbb{H}}^2+\int_0^t\Big(2{}_{\mathbb{V}^*}\langle \mathcal{A}(s,X_s),X_s\rangle_{\mathbb{V}}
%+\|\mathcal{B}(s,X_s)\|_{L_2(U,\mathbb{H})}^2\Big)ds
%\nonumber \\
%\!\!\!\!\!\!\!\!&&+2\int_0^t\langle X_s,\mathcal{B}(s,X_s)dW_s\rangle_{\mathbb{H}}.
%\end{eqnarray*}
For any $\varepsilon>0$, using It\^{o}'s formula for the Lyapunov function $V^{\varepsilon}(r):=(\varepsilon+r)^{1-\frac{\alpha}{2}}$, then by $(\mathbf{A_3^*})$ we can obtain  that for any $0\leq r<t$,
\begin{eqnarray}\label{e111}
\!\!\!\!\!\!\!\!&&(\varepsilon+\|X_t\|_{\mathbb{H}}^2)^{1-\frac{\alpha}{2}}
\nonumber \\
\leq\!\!\!\!\!\!\!\!&&(\varepsilon+\|X_r\|_{\mathbb{H}}^2)^{1-\frac{\alpha}{2}}-\delta(1-\frac{\alpha}{2})\int_r^t \frac{\|X_s\|_{\mathbb{V}}^{\alpha}}{(\varepsilon+\|X_s\|_{\mathbb{H}}^2)^{\frac{\alpha}{2}}}ds
\nonumber \\
\!\!\!\!\!\!\!\!&&
+(1-\frac{\alpha}{2})\int_r^t \frac{g(\|X_s\|_{\mathbb{H}}^2)+\|\mathcal{B}(s,X_s)\|_{L_2(U,\mathbb{H})}^2}{(\varepsilon+\|X_s\|_{\mathbb{H}}^2)^{\frac{\alpha}{2}}}ds
\nonumber \\
\!\!\!\!\!\!\!\!&&-\frac{\alpha}{2}(1-\frac{\alpha}{2})\int_r^t \frac{2\|\mathcal{B}(s,X_s)^*X_s\|_{U}^2}{(\varepsilon+\|X_s\|_{\mathbb{H}}^2)^{\frac{\alpha}{2}+1}}   ds+\mathcal{M}^{\varepsilon}_{r,t}
\nonumber \\
=\!\!\!\!\!\!\!\!&&(\varepsilon+\|X_r\|_{\mathbb{H}}^2)^{1-\frac{\alpha}{2}}-\delta(1-\frac{\alpha}{2})\int_r^t \frac{\|X_s\|_{\mathbb{V}}^{\alpha}}{(\varepsilon+\|X_s\|_{\mathbb{H}}^2)^{\frac{\alpha}{2}}}ds
\nonumber \\
\!\!\!\!\!\!\!\!&&
+(1-\frac{\alpha}{2})\int_r^t \frac{\big(g(\|X_s\|_{\mathbb{H}}^2)+\|\mathcal{B}(s,X_s)\|_{L_2(U,\mathbb{H})}^2\big)(\varepsilon+\|X_s\|_{\mathbb{H}}^2)}{(\varepsilon+\|X_s\|_{\mathbb{H}}^2)^{\frac{\alpha}{2}+1}}ds
\nonumber \\
\!\!\!\!\!\!\!\!&&-(1-\frac{\alpha}{2})\int_r^t \frac{\alpha\|\mathcal{B}(s,X_s)^*X_s\|_{U}^2}{(\varepsilon+\|X_s\|_{\mathbb{H}}^2)^{\frac{\alpha}{2}+1}}ds+\mathcal{M}^{\varepsilon}_{r,t},
\end{eqnarray}
where
$$\mathcal{M}^{\varepsilon}_{r,t}:=2(1-\frac{\alpha}{2})\int_r^t \frac{\langle X_s,\mathcal{B}(s,X_s)dW_s\rangle_{\mathbb{H}}}{(\varepsilon+\|X_s\|_{\mathbb{H}}^2)^{\frac{\alpha}{2}}}.$$
Recall the fact that there is a constant $c^*>0$ such that $ \|u\|_{\mathbb{V}}\geq c^*\|u\|_{\mathbb{H}} $. Then, it follows that
\begin{eqnarray*}
\!\!\!\!\!\!\!\!&&(\varepsilon+\|X_t\|_{\mathbb{H}}^2)^{1-\frac{\alpha}{2}}+\delta(c^*)^{\alpha}(1-\frac{\alpha}{2})\int_r^t \frac{\|X_s\|_{\mathbb{H}}^{\alpha}}{(\varepsilon+\|X_s\|_{\mathbb{H}}^2)^{\frac{\alpha}{2}}}ds
\nonumber \\
\leq\!\!\!\!\!\!\!\!&&(\varepsilon+\|X_r\|_{\mathbb{H}}^2)^{1-\frac{\alpha}{2}}+\mathcal{M}^{\varepsilon}_{r,t}
\nonumber \\
\!\!\!\!\!\!\!\!&&+
(1-\frac{\alpha}{2})\int_r^t \frac{\big(g(\|X_s\|_{\mathbb{H}}^2)+\|\mathcal{B}(s,X_s)\|_{L_2(U,\mathbb{H})}^2\big)(\varepsilon+\|X_s\|_{\mathbb{H}}^2)}{(\varepsilon+\|X_s\|_{\mathbb{H}}^2)^{\frac{\alpha}{2}+1}}ds
\nonumber \\
\!\!\!\!\!\!\!\!&&-(1-\frac{\alpha}{2})\int_r^t \frac{\alpha\|\mathcal{B}(s,X_s)^*X_s\|_{U}^2}{(\varepsilon+\|X_s\|_{\mathbb{H}}^2)^{\frac{\alpha}{2}+1}}ds.
\end{eqnarray*}
Then, due to $(\mathbf{A_5^*})$ we deduce that
\begin{eqnarray}\label{e112}
\!\!\!\!\!\!\!\!&&(\varepsilon+\|X_t\|_{\mathbb{H}}^2)^{1-\frac{\alpha}{2}}+\delta(c^*)^{\alpha}(1-\frac{\alpha}{2})\int_r^t \frac{\|X_s\|_{\mathbb{H}}^{\alpha}}{(\varepsilon+\|X_s\|_{\mathbb{H}}^2)^{\frac{\alpha}{2}}}\mathbf{1}_{\{\|X_s\|_\mathbb{H}>0\}}ds
\nonumber \\
\leq\!\!\!\!\!\!\!\!&&(\varepsilon+\|X_r\|_{\mathbb{H}}^2)^{1-\frac{\alpha}{2}}+2(1-\frac{\alpha}{2})\int_r^t \mathbf{1}_{\{\|X_s\|_\mathbb{H}>0\}}\frac{\langle X_s,\mathcal{B}(s,X_s)dW_s\rangle_{\mathbb{H}}}{(\varepsilon+\|X_s\|_{\mathbb{H}}^2)^{\frac{\alpha}{2}}}
\nonumber \\
\!\!\!\!\!\!\!\!&&+
(1-\frac{\alpha}{2})\int_r^t \frac{g(\|X_s\|_{\mathbb{H}}^2)(\varepsilon+\|X_s\|_{\mathbb{H}}^2)}{(\varepsilon+\|X_s\|_{\mathbb{H}}^2)^{\frac{\alpha}{2}+1}}\mathbf{1}_{\{\|X_s\|_\mathbb{H}>0\}}ds
\nonumber \\
\!\!\!\!\!\!\!\!&&+
(1-\frac{\alpha}{2})\int_r^t \frac{\|\mathcal{B}(s,X_s)\|_{L_2(U,\mathbb{H})}^2(\varepsilon+\|X_s\|_{\mathbb{H}}^2)}{(\varepsilon+\|X_s\|_{\mathbb{H}}^2)^{\frac{\alpha}{2}+1}}\mathbf{1}_{\{\|X_s\|_\mathbb{H}>0\}}ds
\nonumber \\
~~~~\!\!\!\!\!\!\!\!&&~~~-(1-\frac{\alpha}{2})\int_r^t \frac{\alpha\|\mathcal{B}(s,X_s)^*X_s\|_{U}^2}{(\varepsilon+\|X_s\|_{\mathbb{H}}^2)^{\frac{\alpha}{2}+1}}\mathbf{1}_{\{\|X_s\|_\mathbb{H}>0\}}ds.
\end{eqnarray}
Taking $\varepsilon\to 0$ on both sides of (\ref{e112}), we get
\begin{eqnarray*}
\!\!\!\!\!\!\!\!&&\|X_t\|_{\mathbb{H}}^{2-\alpha}+\delta(c^*)^{\alpha}(1-\frac{\alpha}{2})\int_r^t \mathbf{1}_{\{\|X_s\|_\mathbb{H}>0\}}ds
\nonumber \\
\leq\!\!\!\!\!\!\!\!&&\|X_r\|_{\mathbb{H}}^{2-\alpha}+\mathcal{M}_{r,t}
\nonumber \\
\!\!\!\!\!\!\!\!&&+
(1-\frac{\alpha}{2})\int_r^t \frac{\big(g(\|X_s\|_{\mathbb{H}}^2)+\|\mathcal{B}(s,X_s)\|_{L_2(U,\mathbb{H})}^2\big)\|X_s\|_{\mathbb{H}}^2}{\|X_s\|_{\mathbb{H}}^{\alpha+2}}\mathbf{1}_{\{\|X_s\|_\mathbb{H}>0\}}ds
\nonumber \\
\!\!\!\!\!\!\!\!&&
-(1-\frac{\alpha}{2})\int_r^t \frac{\alpha\|\mathcal{B}(s,X_s)^*X_s\|_{U}^2}{\|X_s\|_{\mathbb{H}}^{\alpha+2}}\mathbf{1}_{\{\|X_s\|_\mathbb{H}>0\}}ds,
\end{eqnarray*}
where
$$\mathcal{M}_{r,t}:=2(1-\frac{\alpha}{2})\int_r^t \mathbf{1}_{\{\|X_s\|_\mathbb{H}>0\}}\frac{\langle X_s,\mathcal{B}(s,X_s)dW_s\rangle_{\mathbb{H}}}{\|X_s\|_{\mathbb{H}}^\alpha}.$$
According to the assumption $(\mathbf{A_5^*})$, we have
\begin{equation*}
\|X_t\|_{\mathbb{H}}^{2-\alpha}+\delta(c^*)^{\alpha}(1-\frac{\alpha}{2})\int_r^t \mathbf{1}_{\{\|X_s\|_\mathbb{H}>0\}}ds
\leq\|X_r\|_{\mathbb{H}}^{2-\alpha}+\mathcal{M}_{r,t}.
\end{equation*}
We complete the proof.
\end{proof}

\noindent\textbf{Proof of Theorem \ref{th3}.} \textbf{Step 1.} first prove (\ref{fte1}). For any $0\leq r<t$, by (\ref{e116}) and (\ref{e113}) we deduce
\begin{equation}\label{e114}
\mathbb{E}\big[\|X_t\|_{\mathbb{H}}^{2-\alpha}|\mathscr{F}_r\big]
\leq\|X_r\|_{\mathbb{H}}^{2-\alpha}+\mathbb{E}\big[\mathcal{M}_{r,t}|\mathscr{F}_r\big]=\|X_r\|_{\mathbb{H}}^{2-\alpha},
\end{equation}
which implies  that the process
$$t\mapsto\|X_t\|_{\mathbb{H}}^{2-\alpha}$$
is an $(\mathscr{F}_t)$-nonnegative supermartingale. This combining with (\ref{e116}) yields that for every pair of stopping times $\tau^1<\tau^2$,
$$\mathbb{E}\|X_{\tau^2}\|_{\mathbb{H}}^{2-\alpha}\leq \mathbb{E}\|X_{\tau^1}\|_{\mathbb{H}}^{2-\alpha}.$$
In particular, for any $t>\tau_{e}$, we have
$$\mathbb{E}\|X_t\|_{\mathbb{H}}^{2-\alpha}\leq \mathbb{E}\|X_{\tau_{e}}\|_{\mathbb{H}}^{2-\alpha}=0.$$
Thus, it follows that for any $t\geq\tau_{e}$
$$\|X_t\|_{\mathbb{H}}=0,~\mathbb{P}\text{-a.s.}.$$

\vspace{1mm}
\noindent \textbf{Step 2.}  As for (\ref{fte2}) and (\ref{fte3}), we set $r=0$ and take expectation on both sides of (\ref{e113}), by a standard localization argument we obtain that for all $t\geq 0$,
\begin{equation*}
\delta(c^*)^{\alpha}(1-\frac{\alpha}{2})\int_0^t \mathbb{P}(\tau_{e}>s)ds
\leq\|x\|_{\mathbb{H}}^{2-\alpha}.
\end{equation*}
This implies
\begin{equation*}
\mathbb{P}(\tau_{e}>T)
\leq\|x\|_{\mathbb{H}}^{2-\alpha}\Big/\big(\delta(c^*)^{\alpha}(1-\frac{\alpha}{2})\big)T.
\end{equation*}
Letting $T\to\infty$, we derive
\begin{equation*}
\mathbb{P}(\tau_{e}<\infty)
=1.
\end{equation*}

On the other hand, we set  $r=0$ and $t=\tau_{e}$ and take expectation on both sides of (\ref{e113}) to obtain
$$\delta(c^*)^{\alpha}(1-\frac{\alpha}{2})\mathbb{E}\tau_{e}\leq\|x\|_{\mathbb{H}}^{2-\alpha}.  $$
We complete the proof of (\ref{fte2}) and (\ref{fte3}) with $c_0=\frac{1}{\delta(c^*)^{\alpha}(1-\frac{\alpha}{2})}$.  \hspace{\fill}$\Box$

\section{Appendix}\label{appendix}

%\subsection{Jakubowski's version of Skorokhod theorem}
The classical Skorokhod theorem can only be applied in metric space. In this work, we use the following Jakubowski's version of the Skorokhod theorem in the form presented by Brze\'{z}niak and Ondrej\'{a}t \cite{BO}.
\begin{lemma}\label{sko1}$($Skorokhod Theorem$)$
Let $\mathcal{Y}$ be a topological space such that there exists a sequence of continuous functions $f_m:\mathcal{Y}\to \mathbb{R}$ that separates points of  $\mathcal{Y}$. Let us denote by $\mathscr{S}$ the $\sigma$-algebra generated by the maps $f_m$. Then

(i) every compact subset of $\mathcal{Y}$ is metrizable;

(ii) if $(\mu_m)$ is tight sequence of probability measures on $(\mathcal{Y},\mathscr{S})$, then there exists a subsequence  denoted also by $(m)$, a probability space $(\Omega,\mathscr{F},\mathbb{P})$ with $\mathcal{Y}$-valued Borel measurable variables $\xi_m$, $\xi$ such that $\mu_m$ is the law of $\xi_m$ and $\xi_m$ converges to $\xi$ almost surely on $\Omega$. Moreover, the law of $\xi$ is a Random measure.
\end{lemma}

 We first recall the definitions of the countably generated Borel space and the standard
Borel space in the sense of Parthasarathy (cf. \cite[Chapter V, Definition 2.1 and 2.2]{P67}).

\begin{definition}
(Countably generated Borel space) A
Borel space $(\mathcal{X}, \mathscr{B})$ is said to be countably generated if there exists a denumerable class
$\mathcal{D} \subset \mathscr{B}$ such that $\mathcal{D}$ generates $\mathscr{B}$.
\end{definition}

\begin{definition}\label{de5}
(Standard Borel space) A countably
generated Borel space $(\mathcal{X}, \mathscr{B})$ is called standard if there exists a Polish
space $\mathcal{Y}$ such that the $\sigma$-algebras $\mathscr{B}$ and $\mathcal{Y}$ are $\sigma$-isomorphic.
\end{definition}

In order to apply Lemma \ref{sko1}, we recall the following result from \cite{Liang}.

\begin{theorem}\label{tha1}$($\cite[Theorem B.4]{Liang}$)$
Let $(\mathcal{X}, \mathscr{B})$ be any standard Borel space. Suppose that $\{f_m\}_{m\in\mathbb{N}}$ is an
$\mathscr{B}$-measurable sequence from $\mathcal{X}$ to $\mathbb{R}$, which separate the points of $\mathcal{X}$. Denote by $\sigma_0(\mathcal{X})  $ the $\sigma$-algebra generated by  $\{f_m\}_{m\in\mathbb{N}}$. Then $\sigma_0(\mathcal{X})=\mathscr{B}$.
\end{theorem}
%\subsection{$\mathcal{Z}_T^1$ is standard Borel space}
%Recall the space
%$$\mathcal{Z}^1_T:=\mathbb{C}_T(\mathbb{V}^*)\cap L^{\alpha}([0,T];\mathbb{H})\cap L^{\alpha}_w([0,T];\mathbb{V}).$$
%Let $\tau_{\mathcal{Z}^1_T}$ denote the intersection topology of three spaces. Let $\mathscr{B}(\tau_{\mathcal{Z}^1_T})$
% be the associated Borel $\sigma$-algebra.

%%%%%%%%%%%%%%%%%%%%%%%%%%%%%%%%%%%%%%%%%%%%%%
%% Funding information, if any,             %%
%% should be provided in the                %%
%% funding section.                         %%
%%%%%%%%%%%%%%%%%%%%%%%%%%%%%%%%%%%%%%%%%%%%%%
\vspace{3mm}

%Declarations

\noindent\textbf{Data availability} Data sharing is not applicable to this article as no datasets were generated or analysed during
the current study.

\noindent\textbf{Statements and Declarations} On behalf of all authors, the corresponding author states that there is no conflict of interest.

%\vspace{0.5cm}
%\noindent\textbf{Acknowledgment}   This work is supported by National Key R\&D program of China (No. 2023YFA1010101).
%The research of S. Li is also supported by NSFC (No.~12371147). The research of W. Liu is also supported by NSFC (No.~12171208, 12090011,12090010) and the PAPD of Jiangsu Higher Education Institutions.


\begin{thebibliography}{2}
\bibitem{AV24}
\textsc{Agresti, A.} and \textsc{Veraar, M.} (2024).
The critical variational setting for stochastic evolution equations.
 \textit{Probab. Theory Relat. Fields} \textbf{188}, 957--1015.

\bibitem{A1}
\textsc{Aldous, D.} (1978).
Stopping times and tightness,
\textit{Ann. Probab.} \textbf{6} (2), 335-340.

\bibitem{AMR}
\textsc{Appleby, J. A.D., Mao, X.} and \textsc{Rodkina, A} (2008).
Stabilization and destabilization of nonlinear differential equations by noise.
\textit{IEEE Trans. Automat. Control}
 \textbf{53}, 683-691.


\bibitem{BMX}
\textsc{Bagnara, M.,  Maurelli, M.} and  \textsc{Xu, F.} (2025).
No blow-up by nonlinear It\^{o} noise for the Euler equations.
\textit{Electron. J. Probab.}  \textbf{30}, Paper No. 81, 29 pp.

% \bibitem{Bar10}
%\textsc{Barbu, V.}  (2010).
%Nonlinear differential equations of monotone types in Banach spaces.
%\textit{Springer Monographs in Mathematics}, Springer, New York.

\bibitem{BDR09}
\textsc{Barbu, V.  Da Prato, G.} and  \textsc{R\"{o}ckner, M.} (2009).
Stochastic porous media equations and self-organized criticality.
\textit{Comm. Math. Phys.} \textbf{285}, 901-923.

\bibitem{BDR09b}
\textsc{Barbu, V.  Da Prato, G.} and  \textsc{R\"{o}ckner, M.} (2009).
 Finite time extinction for solutions to fast diffusion stochastic porous media equations.
 \textit{C. R. Math. Acad. Sci. Paris} \textbf{347}, 81-84.

%\bibitem{BDR12}
%\textsc{Barbu, V.  Da Prato, G.} and  \textsc{R\"{o}ckner, M.} (2012).
% Finite time extinction of solutions to fast diffusion equations driven by linear multiplicative noise.
% \textit{J. Math. Anal. Appl.} \textbf{389}, 147-164.

\bibitem{BR12}
\textsc{Barbu, V.} and  \textsc{R\"{o}ckner, M.} (2012).
 Stochastic porous media equations and self-organized criticality: convergence to the critical state in all dimensions.
 \textit{Comm. Math. Phys.} \textbf{311}, 539-555.

\bibitem{BRR15}
\textsc{Barbu, V., R\"{o}ckner, M.} and \textsc{Russo, F.} (2015).
Stochastic porous media equations in $\mathbb{R}^d$.
\textit{J. Math. Pures Appl.}
 \textbf{103}, 1024-1052.

% \bibitem{BT73}
%\textsc{Bensoussan, A.} and  \textsc{Temam, R.} (1973).
%\'{E}quations stochastiques du type Navier-Stokes.
%\textit{J. Funct. Anal.} \textbf{13}, 195-222.

%\bibitem{BFR09}
%\textsc{Bl\"{o}mker, D.,  Flandoli F.}  and \textsc{Romito, M.} (2009).
%Markovianity and ergodicity for a surface growth PDE.
%\textit{Ann. Probab.}
%\textbf{37}, 275-313.

\bibitem{Bro63}
\textsc{Browder, F.E.} (1963).
{N}onlinear elliptic boundary value problems.
 \textit{Bull. Amer. Math. Soc.}
 \textbf{69}, 862-874.

\bibitem{Bro64}
\textsc{Browder, F.E.} (1964).
Non-linear equations of evolution,
 \textit{Ann. Math.}
  \textbf{80}, 485-523.


\bibitem{Browder1977}
\textsc{Browder, F. E.}  (1977).
Pseudo-monotone operators and nonlinear elliptic boundary value problems on
    unbounded domains. \textit{Proc. Nat. Acad. Sci.} \textbf{74}, 2659-2661.

%\bibitem{BN15}
%\textsc{Bidaut-Veron, M.}  and  \textsc{Nguyen, Q.} (2015).
%Evolution equations of $p$-Laplace type with absorption or source terms and measure data.
%\textit{Commun. Contemp. Math.} \textbf{17}, 1550006, 25 pp.

\bibitem{BM13}
\textsc{Brze\'{z}niak, Z.}  and  \textsc{Motyl, E.} (2013).
Existence of a martingale solution of the stochastic Navier-Stokes equations in unbounded 2D and 3D domains.
\textit{J. Differential Equations}
\textbf{254}, 1627-1685.

%\bibitem{BM19}
%\textsc{Brze\'{z}niak, Z.}  and  \textsc{Motyl, E.} (2019).
%Fractionally dissipative stochastic quasi-geostrophic type equations on $\mathbb{R}^d$.
%\textit{SIAM J. Math. Anal.}
%\textbf{51}, 2306-2358.

\bibitem{BO}
\textsc{Brze\'{z}niak, Z.}  and  \textsc{Ondrej\'{a}t, M.} (2013).
Stochastic geometric wave equations with values in compact Riemannian homogeneous spaces.
\textit{Ann. Probab.}
\textbf{41}, 1938-1977.

%\bibitem{BNS20}
%\textsc{Buckmaster, T., Nahmod, A., Staffilani, G.} and \textsc{Widmayer, K.} (2020).
%The surface quasi-geostrophic equation with random diffusion.
%\textit{Int. Math. Res. Not. IMRN}
%\textbf{23}, 9370-9385.


%\bibitem{CV}
%\textsc{Caffarelli, L.} and \textsc{Vasseur, A.} (2010).
%Drift diffusion equations with fractional diffusion and the quasi-geostrophic equation.
%\textit{Ann. of Math.} \textbf{171}, 1903-1930.

\bibitem{CY07}
\textsc{Chen, X., Qi, Y.} and \textsc{Wang, M.} (2007).
Singular solutions of parabolic $p$-Laplacian with absorption.
\textit{Trans. Amer. Math. Soc.} \textbf{359}, 5653-5668.


\bibitem{CR05}
\textsc{Consiglieri, L.} and \textsc{Rodrigues, J. F.} (2005).
On stationary flows with energy dependent nonlocal viscosities.
\textit{J. Math. Sci.}
\textbf{127}, 1875-1885

%\bibitem{CN18}
%\textsc{Constantin, P.} and \textsc{Nguyen, H.Q.} (2018).
%Global weak solutions for SQG in bounded domains.
%\textit{Comm. Pure Appl. Math.} \textbf{71}, 2323-2333.



% \bibitem{Dib93}
%\textsc{Dibenedetto, E.} (1993).
%  Degenerate Parabolic Equations, Springer, New York.

\bibitem{DHW23}
\textsc{Diening, L. Hofmanov\'{a}, M.}  and \textsc{Wichmann, J.} (2023).
 An averaged space-time discretization of the stochastic $p$-Laplace system.
 \textit{Numer. Math.} \textbf{153}, 557-609.



%\bibitem{FPR}
%\textsc{Filippucci, R., Pucci, P.}  and \textsc{Robert, F.} (2009).
%On a $p$-Laplace equation with multiple critical nonlinearities.
%\textit{J. Math. Pures Appl.}, \textbf{91},  156-177.



\bibitem{F15b}
\textsc{Flandoli, F.} (2015).
A stochastic view over the open problem of well-posedness for the 3D Navier-Stokes equations. Stochastic analysis: a series of lectures,
\textit{Progr. Probab.}, \textbf{68}, Birkh\"{a}user/Springer, Basel, 221-246.



\bibitem{FMRT}
\textsc{Foias, C., Manley, O., Rosa, R.} and \textsc{Temam, R.} (2001).
Navier-Stokes equations and turbulence (Vol. 83).
Cambridge University Press.


\bibitem{Gess15}
\textsc{Gess, B.}  (2015).
Finite time extinction for stochastic sign fast diffusion and self-organized criticality.
\textit{Comm. Math. Phys.} \textbf{335}, 309-344.

\bibitem{GT14}
\textsc{Gess, B.} and \textsc{T\"{o}lle, J. M.} (2014).
Multi-valued, singular stochastic evolution inclusions.
\textit{J. Math. Pures Appl.} \textbf{101}, 789-827.

\bibitem{GT16}
\textsc{Gess, B.} and \textsc{T\"{o}lle, J. M.} (2016).
 Ergodicity and local limits for stochastic local and nonlocal $p$-Laplace equations.
\textit{SIAM J. Math. Anal.} \texttt{48}, 4094-4125.


\bibitem{GZ1}
\textsc{Glatt-Holtz, N.} and  \textsc{Ziane, M.} (2009).
Strong pathwise solutions of the stochastic Navier-Stokes system.
\textit{Adv. Differential Equations} \textbf{14}, 567-600.

%\bibitem{GRZ}
%\textsc{Goldys, B., R\"{o}ckner, M.} and \textsc{Zhang, X.} (2009).
%Martingale solutions and Markov selections for stochastic partial differential equations.
%\textit{Stochastic Process. Appl.} \textbf{119}, 1725-1764.


\bibitem{GC1}
\textsc{Goodair, D.} and \textsc{Crisan, D.} (2024).
Stochastic calculus in infinite dimensions and SPDEs.
SpringerBriefs in Mathematics. Springer, Cham.


%\bibitem{G82}
%\textsc{Gy{}{\"o}ngy, I.} (1982).
% \emph{{O}n stochastic equations with respect to semimartingale
%  {III}}. \textit{Stochastics} \textbf{7}, 231-254.

\bibitem{HS66}
\textsc{Hartman, P.} and \textsc{Stampacchia,  G.} (1966).
 On some nonlinear elliptic differential equations,
 \textit{Acta. Math.} \textbf{115}, 271-310.

\bibitem{Hensel21}
\textsc{Hensel, S.} (2021).
Finite time extinction for the 1D stochastic porous medium equation with transport noise.
\textit{Stoch. Partial Differ. Equ. Anal. Comput.}
\textbf{9}, 892-939.


\bibitem{HHL}
\textsc{Hong, W., Hu, S.} and \textsc{Liu, W.} (2024).
 McKean-Vlasov SDEs and SPDEs with locally monotone coefficients,
\textit{Ann. Appl. Probab.}
\textbf{34}, 2136-2189.

\bibitem{HLL25}
\textsc{Hong, W., Li, S.} and \textsc{Liu, W.} (2025).
McKean-Vlasov stochastic partial differential equations: existence, uniqueness and propagation of chaos.
\textit{Probab. Theory Related Fields}
 \textbf{193}, 717-793.


\bibitem{HLL26}
\textsc{Hong, W., Li, S.} and \textsc{Liu, W.} (2026).
Mean field stochastic partial differential equations with nonlinear kernels.
\textit{Ann. Appl. Probab.}
 \textbf{36}, 206-274.






\bibitem{Ja86}
\textsc{Jakubowski A.} (1986).
 On the Skorokhod topology.
 \textit{Ann. Inst. H. Poincar\'{e} Probab. Statist.}
 \textbf{22}, 263-285.


\bibitem{KP} \textsc{Kato, T.} and \textsc{Ponce, G. } (1988)
Commutator estimates and the Euler and Navier-Stokes equations.
\textit{Comm. Pure Appl. Math.} \textbf{41}, 891-907.


\bibitem{KS18}
\textsc{Kavallaris, N.} and \textsc{Suzuki, T.} (2018).
Non-Local Partial Differential Equations for Engineering and Biology.
Mathematical Modeling and Analysis. Mathematics for Industry (Tokyo), \textit{Springer, Cham.}
\textbf{31.}   300 pp.






%\bibitem{KNV07}
%\textsc{Kiselev, A., Nazarov, F.} and \textsc{Volberg, A. } (2007).
%Global well-posedness for the critical 2D dissipative quasi-geostrophic
%equation.
%\textit{Invent. Math.} \textbf{167}, 445-453.

\bibitem{K83}
\textsc{Kirchhoff, G.} (1883).
Vorlesungen \"{u}ber Mechanik.
\textit{Teubner, Sluttgart}.


\bibitem{KR}
\textsc{Krylov, N.V.} and \textsc{Rozovskii, B.L.} (1981).
Stochastic evolution equations.
\textit{Translated from Itogi Naukii Tekhniki, Seriya Sovremennye Problemy
Matematiki.Plenum Publishing Corp.}
\textbf{14}, 71-146.
%
%\bibitem{KR05}
%\textsc{Krylov, N.V.} and \textsc{R\"{o}ckner, M.} (2005).
%Strong solutions of stochastic equations with singular time dependent drift.
%\textit{Probab. Theory Related Fields}
%\textbf{131}, 154-196.

\bibitem{Lad67}
\textsc{Ladyzenskaja, O.A.} (1967).
 New equations for the description of the motions of viscous incompressible fluids, and global solvability for their boundary value problems.
 \textit{Trudy Mat. Inst. Steklov.} \textbf{102}, 85-104.




\bibitem{Le34}
\textsc{Leray, J.} (1934).
Sur le mouvement d'un liquide visqueux emplissant l'espace.
\textit{Acta Math.}
\textbf{63}, 193–248.


\bibitem{LL65}
\textsc{Leray, J.} and \textsc{Lions, J.L.} (1965).
 Quelques r\'{e}sultats de Visik sur les probl\`{e}mes elliptiques non lin\'{e}aires par les m\'{e}thodes de Minty-Browder,
 \textit{Bull. Soc. Math. France} \textbf{93}, 97-107.


\bibitem{LC03}
\textsc{Li, Y.} and \textsc{Xie, C.} (2003).
Blow-up for $p$-Laplacian parabolic equations.
\textit{Electron. J. Differential Equations} \textbf{20}, 12 pp.


\bibitem{Liang}
\textsc{Liang, S.} (2021).
Stochastic hypodissipative
hydrodynamic equations:
well-poseness, stationary
solutions and ergodicity.
PhD Thesis, Bielefeld University.

\bibitem{Lion}
\textsc{J.L. Lions} (1969).
 Quelques m\'{e}thodes de r\'{e}solution des probl\`{e}mes aux limites non lin\'{e}aires, Dunod, Paris.



\bibitem{LR2}
\textsc{Liu, W.} and \textsc{R\"{o}ckner, M.} (2010).
SPDE in Hilbert space with locally monotone coefficients.
\textit{J. Funct. Anal.}
\textbf{259},  2902-2922.

\bibitem{LR13}
\textsc{Liu, W.} and \textsc{R\"{o}ckner, M.} (2013).
Local and global well-posedness of SPDE with generalized coercivity conditions.
\textit{J. Differential Equations}
\textbf{254}, 725-755.

\bibitem{LR1}
\textsc{Liu, W.} and \textsc{R\"{o}ckner, M.} (2015).
Stochastic Partial Differential Equations: An Introduction.
Universitext, Springer.

 \bibitem{LR21}
\textsc{Liu, W., R\"{o}ckner, M.} and \textsc{da Silva J. L.} (2021).
Strong dissipativity of generalized time-fractional derivatives and quasi-linear (stochastic) partial differential equations.
 \textit{J. Funct. Anal.} \textbf{281}, 109135.


\bibitem{MS}
\textsc{Maslowski, B.} and  \textsc{Seidler, J.} (1999).
On sequentially weakly Feller solutions to SPDE's.
 \textit{Rend. Lincei Mat. Appl.} \textbf{10}, 69-78.


\bibitem{Mi62}
\textsc{Minty, G. J. } (1962).
{M}onotone (non-linear) operators in {H}ilbert space.
 \textit{Duke. Math. J.} \textbf{29}, 341-346.

%\bibitem{Mi63}
%\textsc{Minty, G. J. } (1963).
%On a monotonicity method for the solution of non-linear
%  equations in {B}anach space.
%\textit{Proc. Nat. Acad. Sci. USA} \textbf{50}
%  1038--1041.

\bibitem{NTT21}
\textsc{Nguyen, P., Tawri, K.} and \textsc{Temam, R.} (2021).
Nonlinear stochastic parabolic partial differential equations with a monotone operator of the Ladyzenskaya-Smagorinsky type, driven by a L\'{e}vy noise.
\textit{J. Funct. Anal.} \textbf{281}, Paper No. 109157, 74 pp.

\bibitem{Ni59}
\textsc{Nirenberg, L.} (1959).
On elliptic partial differential equations.
\textit{Ann. Scuola Norm. Sup. Pisa Cl. Sci.}
\textbf{13}, 115-162.

\bibitem{Nguyen}
\textsc{Nguyen, Q.} (2023).
Potential estimates and quasilinear parabolic equations with measure data.
\textit{Mem. Amer. Math. Soc.}
\textbf{291}, v+123 pp.


\bibitem{O07}
\textsc{Odasso, C.} (2007).
Exponential mixing for the 3D stochastic Navier-Stokes equations.
\textit{Comm. Math. Phys.}
\textbf{270}, 109-139.


\bibitem{Par75}
\textsc{Pardoux, E.} (1975).
 {E}quations aux d{\'e}riv{\'e}es partielles stochastiques non
  lin{\'e}aires monotones,
  \textit{Ph.D. thesis}, Universit{\'e} Paris XI.

\bibitem{P67}
\textsc{Parthasarathy, K.R.} (1967).
 Probability Measures on Metric Spaces,
 Academic Press.




\bibitem{PR07}
\textsc{Pr\'{e}v\^{o}t, C.} and \textsc{R\"{o}ckner, M.} (2007).
A concise course on stochastic partial differential equations.
Lecture Notes in Mathematics, 1905. Springer, Berlin, vi+144 pp.



%
%\bibitem{RRW07}
%\textsc{Ren, J.  R{\"o}ckner, M.}  and \textsc{Wang, F.-Y.} (2007).
%Stochastic generalized porous media and fast diffusion equations.
% \textit{J. Differential
%  Equations} \textbf{238}, 118-152.

\bibitem{RLH}
\textsc{Raible, M., Linz, S. J.} and \textsc{H\"{a}nggi, P.} (2000). Amorphous thin film growth:  Minimal deposition equation.
\textit{Physical Review E} \textbf{62},
1691-1705.


\bibitem{RRW07}
\textsc{Ren, J., R\"{o}ckner, M.} and \textsc{Wang, F.-Y.} (2007).
Stochastic generalized porous media and fast diffusion equations.
\textit{J. Differential Equations}
\textbf{238}, 118-152.





\bibitem{Res95}
\textsc{Resnick, S.} (1995).
Danamical Problems in Non-linear Advective Partial Differential Equations.
Ph.D. Thesis, University
of Chicago, Chicago.

\bibitem{RSZ}
\textsc{R\"{o}ckner, M., Schmuland, B.} and \textsc{Zhang, X.} (2008).
Yamada-Watanabe theorem for stochastic evolution equations in infinite dimensions.
\textit{Condens. Matter Phys.}
\textbf{54}, 247-259.
%

\bibitem{RSZ1}
\textsc{R\"{o}ckner, M., Shang, S.} and \textsc{Zhang, T.} (2024).
Well-posedness of stochastic partial differential equations with fully local monotone coefficients.
 \textit{Math. Ann.} \textbf{390}, 3419-3469.

\bibitem{RW13}
\textsc{R\"{o}ckner, M.} and \textsc{Wang, F.-Y.} (2013).
 General extinction results for stochastic partial differential equations and applications.
 \textit{J. Lond. Math. Soc.}
  \textbf{87}, 545-560.




\bibitem{RZ09}
\textsc{R\"{o}ckner, M.} and \textsc{Zhang, X.} (2009).
Stochastic tamed 3D Navier-Stokes equations: existence, uniqueness and ergodicity.
\textit{Probab. Theory Related Fields} \textbf{145}, 211-267.



\bibitem{RZZ14}
\textsc{R\"{o}ckner, M., Zhu, R.-C.} and \textsc{Zhu, X.-C.} (2014).
 Local existence and non-explosion of solutions for stochastic fractional partial differential equations driven by multiplicative noise.
\textit{Stochastic Process. Appl.} \textbf{124}, 1974-2002.

%\bibitem{RZZ15}
%\textsc{R\"{o}ckner, M., Zhu, R.-C.} and \textsc{Zhu, X.-C.} (2015).
%Sub and supercritical stochastic quasi-geostrophic equation. \textit{Ann. Probab.} \textbf{43}, 1202-1273.



\bibitem{Ste70} \textsc{Stein, E.} (1970).
Singular Integrals and Differentiability Properties of Functions.
\textit{Princeton, NJ: Princeton University Press}.

\bibitem{S3}
\textsc{Simon J.} (1987).
Compact sets in the space $L^p(0,T;B)$.
\textit{Ann. Mat. Pura Appl.}
\textbf{164}, 65-96.


%\bibitem{Tang23}
%\textsc{Tang, H.} (2023).
% On the stochastic Euler-Poincar\'{e} equations driven by pseudo-differential/multiplicative noise.
% \textit{J. Funct. Anal.}
%  \textbf{285}, Paper No. 110075, 61 pp.


%\bibitem{TW22}
%\textsc{Tang, H.} and \textsc{Wang, F. Y.} (2022).
%A general framework for solving singular SPDEs with applications to fluid models driven by pseudo-differential noise.
%\textit{arXiv:2208.08312}.

\bibitem{Tem95}
\textsc{Temam, R.} (1995).
Navier–Stokes Equations and Nonlinear Functional Analysis, second edition.
\textit{CBMS-NSF Regional Conference Series in Applied Mathematics},
vol. 66, Society for Industrial and Applied Mathematics (SIAM), Philadelphia, PA.

\bibitem{TY08}
\textsc{Tian, Y.} and \textsc{Mu, C.} (2008).
Extinction and non-extinction for a $p$-Laplacian equation with nonlinear source.
\textit{Nonlinear Anal.}
 \textbf{69}, 2422-2431.


\bibitem{Ver80}
\textsc{Veretennikov, A. J.} (1980).
Strong solutions and explicit formulas for solutions of stochastic integral equations.
\textit{Mat. Sb. (N.S.)}
\textbf{111}, 434-452, 480.


\bibitem{W26}
\textsc{Wang, B.}  (2026).
Well-posedness of stochastic partial differential equations with polynomial drift driven by nonlinear noise with critical growth in unbounded domains.
\textit{J. Differential Equations} \texttt{478}, Paper No.~114596.




\bibitem{W15}
\textsc{Wang, F.-Y.} (2015).
Asymptotic couplings by reflection and applications for nonlinear monotone SPDES.
\textit{Nonlinear Anal.} \textbf{117}, 169-188.



\bibitem{WW21}
\textsc{Wang, R.} and \textsc{Wang, B.} (2021).
Random dynamics of non-autonomous fractional stochastic $p$-Laplacian equations on $\mathbb{R}^N$.
\textit{Banach J. Math. Anal.} \texttt{15}, Paper No. 19, 42 pp.



\bibitem{Wic23}
\textsc{Wichmann, J.} (2023).
On temporal regularity of strong solutions to stochastic $p$-Laplace systems.
\textit{SIAM J. Math. Anal.} \texttt{55}, 3713-3730.


\bibitem{WK50}
\textsc{Woinowsky-Krieger S.
} (1950).
The effect of an axial force on the vibration of hinged bars.
\textit{J. Appl. Mech.}
\textbf{17}, 35-36.

%\bibitem{YJ07}
%\textsc{Yin, J.} and \textsc{Jin, C.} (2007).
%Critical extinction and blow-up exponents for fast diffusive $p$-Laplacian with sources.
% \textit{Math. Methods Appl. Sci.}
% \textbf{30}, 1147-1167.


\bibitem{Z90}
\textsc{Zeidler, E.} (1990).
 {N}onlinear functional analysis and its applications,
  {II}/{B}, nonlinear monotone operators.
   \textit{Springer-Verlag}, New York.

\bibitem{ZZ21}
\textsc{Zhang, X.} and \textsc{Zhao, G.} (2021).
Stochastic Lagrangian path for Leray's solutions of 3D Navier-Stokes equations.
\textit{Comm. Math. Phys.} \texttt{381}, 491-525.

\end{thebibliography}
\end{document}